\documentclass[reqno]{amsart}
\usepackage{epsfig,latexsym,amsfonts,amssymb,amsmath,amscd}
\usepackage{amsthm}
\usepackage{textcomp}
\usepackage[mathscr]{eucal}
\usepackage{mathrsfs}
\usepackage{mathbbol}
\usepackage[colorinlistoftodos,bordercolor=orange,backgroundcolor=orange!20,linecolor=orange,textsize=scriptsize]{todonotes}
\usepackage{hyperref}
\usepackage{cleveref}
\newcommand{\nsets}{{\mathcal P}^*}

\def\period{period}%
\newtheorem{theoalpha}{Theorem}

\usepackage{eepic, epic}

\usepackage{ifthen}

\newcommand{\Det}[1]{ |{#1}|}

\def\lemref{Lemma~\ref}

\def\Null{{\operatorname{Null}}}

\def\sgn{{\operatorname{sgn}}}

\def\mcA{{\mathcal A}}
\def\mcM{{\mathcal M}}
\def\mcH{{\mathcal H}}
\def\mcV{{\mathcal V}}
\def\Ach{{\mathcal A}_{\operatorname{char}}}

\def\a{\alpha}
\def\dplus{\overset{\circ}{+}}
\renewcommand{\geq}{\geqslant}
\renewcommand{\leq}{\leqslant}
\renewcommand{\ge}{\geqslant}
\renewcommand{\le}{\leqslant}
\renewcommand{\preceq}{\preccurlyeq}
\renewcommand{\succeq}{\succcurlyeq}
\newcommand{\A}[1]{{\rm A#1}} 
\newcommand{\Al}[1]{{\rm A#1}}

\usepackage{tikz}
\usepackage{epsfig,latexsym,amsfonts,amssymb,amsmath,amscd,graphics}
\usepackage{amsfonts,amssymb,amsmath,amscd,amsthm}
\usepackage[mathscr]{eucal}
\usepackage{mathrsfs}

\usepackage{mathbbol}

\usepackage{oldgerm,units}
\usepackage{wrapfig,epsfig}

\usepackage{ifthen}
\usepackage{mathbbol}

\usepackage{amsthm}
\usepackage[mathscr]{eucal}
\usepackage{mathrsfs}
\usepackage{mathbbol}
\usepackage{oldgerm,units}
\usepackage{wrapfig}

\newtheorem{theorem}{Theorem}[section]
\newtheorem{proposition}[theorem]{Proposition}
\theoremstyle{definition}
\newtheorem{definition}[theorem]{Definition}

\theoremstyle{plain}

\newtheorem{lemma}[theorem]{Lemma}

\newtheorem{corollary}[theorem]{Corollary}

\theoremstyle{remark}

\newtheorem{example}[theorem]{Example}
\newtheorem{remark}[theorem]{Remark}
\theoremstyle{plain}

\theoremstyle{remark}

\newcommand{\Real}{\mathbb R}

\newcommand{\Net}{\mathbb N}

\newcommand{\one}{\mathbb{1}}
\newcommand{\zero}{\mathbb{0}}

\newcommand{\trop}[1]{\mathcal{#1}}

\newcommand{\tG}{\trop{G}}

\newcommand{\tT}{\trop{T}}

\newcommand{\id}{\inva{\aad}}

    \ifx\proof\undefined
    \newenvironment{proof}{
    \smallskip
    \noindent\emph{Proof.}}{\hfill\(\Box\)
    \bigskip
    } \fi

\definecolor{lgray}{gray}{0.90}

\def\Dref{Definition~\ref}
\def\Eref{Example~\ref}
\def\Tref{Theorem~\ref}

\def\Hzero{\zero_{\mathcal{H}}}

\def\mcF{\mathcal F}

\def\mcI{\mathcal I}

\def\ctw{\cdot_{\operatorname{tw}}}
\def\({\left(}
\def\){\right)}
\def\Z{{\mathbb Z}}
\def\Q{{\mathbb Q}}
\def\C{{\mathbb C}}

\def\sg{\sigma}

\def\pipe{{\underset{{\ \, }}{\mid}}}

\def\vsemifield0{$\nu$-semifield$^\dagger$}
\def\vsemiring0{$\nu$-semiring$^\dagger$}

\def\pipe1{{\underset{{1}}{\mid}}}
\def\lmod1{\mathrel  \pipe1  \joinrel \joinrel =}

\newcommand{\nablaT}{\hyperref[bal1]{\nabla_T\, }}
\newcommand{\nablaminus}{\hyperref[bal1]{\nabla_{\hspace{-1.5pt}(-)}}}
\newcommand{\triple}{\hyperref[negmap]{triple}}
\newcommand{\nabladag}{\hyperref[bal1]{\nabla_{\dag}}}
\newcommand{\preceqzero}{\hyperref[Su1]{\preceq_0}}
\newcommand{\PropertyN}{\hyperref[propN]{Property~N}}

\def\CFunFF1{\operatorname{CFun} (F,F)}

\def\semiring0{semiring$^{\dagger}$}
\def\Semiring0{Semiring$^{\dagger}$}
\def\Semirings0{Semirings$^{\dagger}$}
\def\semidomain0{semidomain$^{\dagger}$}
\def\semifield0{semifield$^{\dagger}$}
\def\semifields0{semifields$^{\dagger}$}
\def\vsemifields0{$\nu$-semifields$^{\dagger}$}
\def\domain0{domain$^{\dagger}$}
\def\predomain0{pre-domain$^{\dagger}$}
\def\predomains0{pre-domains$^{\dagger}$}
\def\domains0{domains$^{\dagger}$}
\def\vdomains0{$\nu$-domains$^{\dagger}$}

\def\domains0{domains$^\dagger$}

\def\tGz{\tG_{\zero}}

\def\bfv{\textbf{v}}

\newcommand{\ekind}[1]{\renewcommand{\theenumi}{(#1{enumi})}%
\renewcommand{\labelenumi}{(#1{enumi})}}
\newcommand{\eikind}[1]{\renewcommand{\theenumii}{--(#1{enumii})}%
\renewcommand{\labelenumii}{(#1{enumii})}}
\ekind{\roman}
\eikind{\alph}

\def\pipe{{\underset{{\tG}}{\mid}}}

\def\lmod{\mathrel  \pipe \joinrel \joinrel =}
\def\pipe{{\underset{{\tG}}{\mid}}}

\def\ghost0{\operatorname{ghost}}

\def\diag{\operatorname{diag}}

\def\a{\alpha}

\theoremstyle{plain}
\theoremstyle{remark}
\newtheorem{rem}[theorem]{Remark}
\newtheorem{prop*}{proposition}

\theoremstyle{remark}
\newtheorem*{examp*}{Example}
\newtheorem*{examples*}{Examples}
\newtheorem*{remark*}{Remark}

\newtheorem{INote}[theorem]{Important Note}
\theoremstyle{definition}
\newtheorem*{definition*}{Definition}
\newtheorem{assump}[theorem]{Assumption}
\newtheorem*{note*}{Note}

\def\R{\Real}
\def\la{\lambda}

\def\tT{\mathcal T}

\def\tTz{\tT_\zero}

\numberwithin{equation}{section}

\def\M0{M_{\zero}}

\def\perm{\operatorname{perm}}
\def\id{\operatorname{id}}

\def\Rref{Remark~\ref}

\def\tAz{\mathcal A^\circ}

\def\PS{P}

\def\semirings0{semirings$^\dagger$}

\newcommand{\nPS}[1]{\PS_{(!#1)}}
\newcommand{\nPSo}[1]{\nPS{\one}}

\def\Per{\operatorname{perm}}

\newcommand{\absl}[1]{||{#1}||}

\newcommand{\adj}[1]{\operatorname{adj}({#1})}
\newcommand{\adjsym}[1]{\operatorname{adj}^{\operatorname{doub}}({#1})}

\newcommand{\hypersemiring}{hypermodule}

\begin{document}


\title[Linear algebra over semiring pairs]
{Linear algebra over semiring pairs}


\author[M.~Akian ]{Marianne Akian }
\address{Marianne Akian, INRIA and CMAP, \'Ecole polytechnique CNRS. Address:
CMAP, \'Ecole Polytechnique, Route de Saclay, 91128 Palaiseau Cedex,
France} \email{marianne.akian@inria.fr }

\author[S.~Gaubert]{Stephane~Gaubert}
\address{Stephane~Gaubert, INRIA and CMAP, \'Ecole polytechnique CNRS. Address:
CMAP, \'Ecole Polytechnique, Route de Saclay, 91128 Palaiseau Cedex,
France}
 \email{stephane.gaubert@inria.fr }

%
\author[L.~Rowen]{Louis Rowen}
\address{Louis Rowen,
Department of Mathematics, Bar-Ilan University, Ramat-Gan 52900,
Israel} \email{rowen@math.biu.ac.il}


\date{\today}


\thanks{The authors  thank O.~Lorscheid for helpful comments.}
\thanks{The authors also thank J.~Jun for pointing out references \cite{BZ} and \cite{MasM}.}
\thanks{The research of the third author was supported by the ISF grant 1994/20 and the Anshel Peffer Chair.}



\begin{abstract}This is part of an ongoing project  to find a general  algebraic framework for
semiring and hyperfield theory. The structure theory is quite challenging, largely because of the lack of negation, and such basic properties such as unique factorization of polynomials, multiplicativity of determinants, and the characteristic polynomial of a matrix, all fail. (In fact in the max-plus algebra, the sum of two nonzero elements is {\it never} zero!) Consequently  $\zero$ is replaced by a distinguished $\tT$-submodule $\mcA_0$ of $\mcA,$ and $(\mcA,\mcA_0)$ is called a {\it pair}.

This paper treats linear algebra over  a (not necessarily distributive) semiring pair,
with  a   range of applications to tropical algebra as well as related
areas such as hyperrings and fuzzy rings.  First we present   pairs with their morphisms,
called ``weak morphisms;''  we generalize earlier results about triples to pairs in the first three sections. We pay special attention to supertropical pairs,  hyperpairs,
and the doubling functor.

Then  we turn
to matrices and the question of whether the row rank,
column rank, and submatrix rank of a matrix are equal. The
submatrix rank is less than or equal to the row rank and the column
rank in many cases, including ``metatangible pairs'' with unique negation, but there is a counterexample to equality, discovered some time ago by the second author, which we provide in a more general setting (``pairs of
the second kind'') that includes the hyperfield of signs.  We do
find situations when equality holds, encompassing results by Akian,
Gaubert, Guterman, Izhakian, Knebusch, and Rowen, including   versions of  Cramer's rule. We pay special attention to the question of Baker and Zhang whether $n+1$ vectors of length $n$ need be dependent.

At the conclusion of the main part, we consider surpassing relations, which permit us to tighten our results.  The    categorical setting is given in the appendix.
\end{abstract}

\maketitle


  {\small \tableofcontents}



\section{Introduction}

\numberwithin{equation}{section}  In \cite{Row19} and \cite[Definition~2.15, Definition~2.34]{Row21},  general
algebraic frameworks called  a ``triple'' 
and a ``system'' 
were introduced, in order to unify
various algebraic theories including ``classical'' algebra, tropical
algebra, the algebra underlying $F_1$-geometry, hyperrings \cite{krasner,Vi,CC}, 
and fuzzy rings \cite{DrW}.
The connection to  hyperrings and fuzzy rings was explored further in \cite{AGR2}.

The goal of putting tropical mathematics into
a framework that can be understood in general algebraic terms
related to classical algebra  is particularly apt for matrix theory and linear algebra,
given results in
\cite{AGG1,AGG2,Ga,IzhakianRowen2009TropicalRank,IzhakianRowen2008Matrices,IzhakianRowen2008Matrices2},
which mysteriously parallel classical results from matrix theory and
linear algebra, specifically in the supertropical setting.
Similar approaches were made by M.~Baker and O.~Lorscheid who introduced other general structures like ``blueprints" ~\cite{Lor1,Lor2} and ``tracts'' cf.~\cite{BB2,BL,BZ}.

We aim for the most fundamental structure over which       linear algebra can be studied  meaningfully. For starters we need an abelian group $(\mcA,+)$ with some subset $\tTz$ of scalars acting on it. In order to utilize determinants of matrices, we want $\tTz$ to be a multiplicative monoid (with $\zero$ adjoined). We also need some notion of singularity.

  \cite{Row21} and \cite{AGR2}   also made use of a negation map $(-)$ and ``surpassing relation'' $\preceq.$
Recently a  theory of $\tT$-\textbf{pairs} was initiated in \cite{JMR1}, in which, surprisingly, one
can forego with both $(-)$ and $(\preceq)$, merely having a $\tT$-module $\mcA$ over a set $\tT$ and a $\tT$-invariant subset $\mcA_0$ which plays the role of the set of zero elements, not necessarily endowed with any extra structure.
 In this paper, we present this theory
in a general form, which also encompasses tracts, and a generalization of hyperrings in \S\ref{hyps0}, with emphasis on its linear algebra.

 Dependence of vectors and singularity of matrices have significance in the context of pairs, justifying the effort to develop these aspects.  We start this paper by developing pairs in  context for the theory.
As has become clear in the hyperring and fuzzy ring literature, the category of pairs   can be supplied with two types of morphisms, the ``weak morphisms''  used in most of the body of this paper, and ``$\preceq$-morphisms'' which arise  from a ``surpassing relation'' $(\preceq)$ in the structure. Weak morphisms  are appropriate for linear algebra viewed via matrix theory, whereas the
$\preceq$-morphisms, more appropriate to the module-theoretical viewpoint,  are described in~\S\ref{sd}.

 The transition from classical linear algebra to pairs
turns out to be rather delicate, and most of our effort here is to
investigate the connection between the row (and column) rank and the submatrix rank
of a set of vectors over a pair.

\subsubsection{The shape of the paper}

After reviewing basic concepts in \S\ref{Ab}, including  modules over a set,  we turn in \S\ref{prs} to our main structural interest, ``pairs'' (and their ``weak morphisms''), and introduce the properties to be required of pairs.  We present ways of relating elements, especially ``balance relations''  and ``surpassing relations.'' The most significant class of pairs for us are the ``metatangible'' pairs, whose structure is given in \Cref{theoalphaA}. In  \Cref{theoalphaD}, we show how the definition of tangible balancing in metatangible pairs generalizes balancing for triples as defined in \cite{Row21}.  Also we consider ``(weak) moduli,'' which exist for metatangible pairs, by \Cref{theoalphaB}.    We prove these three theorems in \S\ref{modtrp3}, using    preparatory results about metatangible pairs.

In \Cref{pairex} we bring in  major examples, some of which  (\Cref{metex}) are not triples. ``Supertropical pairs,'' generalizing \cite{IR}, are given in~\Cref{sut}.
The other main (not necessarily metatangible) construction of pairs is via hyperrings, cf.~\Cref{hyps0}. Of special interest is a general version of a construction of Krasner, cf.~\Cref{theoalphaE}.

To obtain a negation map, one can employ
``doubling,''  given in \S\ref{Making}   (described  categorically later in \Cref{theoalphaF}),   reminiscent of Grothendieck's construction of the integers from the natural numbers.

In  \S\ref{lina}, we turn to the main subject of this paper, linear algebra of pairs. Starting with
matrices and their determinants in  \S\ref{matra0}, we give  versions of classical results results, such as the generalized Laplace identity   and the Cauchy-Binet formula (\Cref{theoalphaG}) and the Cayley-Hamilton Theorem (\Cref{theoalphaH}). 
Next, we introduce dependence of vectors over a pair,   bringing in the main hypotheses explored in  this paper, Conditions \A{1}-\A{6}, to compare the different notions of rank.  Counterexamples are given for Conditions   \A{2}, \A{3} and \A{4}, cf.~\Cref{theoalphaI}, \Cref{countq}, and \Cref{hkrashy22}, but also we present an assortment of theorems verifying   these conditions in assorted situations. Basic tools include Cramer's rule (\Cref{theoalphaJ10}, \Cref{theoalphaJ2}, and \Cref{theoalphaJ}).

We complete the proofs of our  theorems about rank in \S\ref{Jacobi1}, also obtaining, under extra assumptions, the Jacobi algorithm for pairs.


Although the most of our results involve a symmetric  ``balance relation,'' several theorems can be framed better in terms of a surpassing relation $(\preceq)$. The
$\preceq$-theory is stronger, since, by \Cref{impl}, $b_1\preceqzero b_2$ implies $b_1$ and $b_2$ are balanced, although the converse obviously fails since $\preceq$ is not symmetric.
 In \S\ref{sd} we  define dependence of vectors in the context of $(\preceq)$.

In Appendix A we expand on categories arising in this paper.

\section{Algebraic background}\label{Ab}$ $

$\Net^+$ denotes the positive natural numbers, and $\Net = \Net^+ \cup{\zero}.$

\begin{definition} $ $
\begin{enumerate}    \item A \textbf{semigroup} is a set with a binary
associative operation.  An \textbf{additive semigroup} is a commutative semigroup with the operation denoted by "$+$" and a zero element $\zero.$

 \item  An additive semigroup $(\mathcal{A},+,\zero)$
is
\textbf{idempotent} if $b + b = b$ for all $b\in \mathcal{A}.$
$(\mathcal{A},+,\zero)$
is  \textbf{bipotent} if $b_1+b_2 \in \{b_1,b_2\}$ for all $b_1, b_2 \in \mcA.$


 \item    A \textbf{semiring}
\cite{golan92} satisfies all the properties of a ring (including associativity and distributivity of multiplication over addition), but without
negation. We shall denote multiplication by concatenation, and assume that semirings have a $\zero$ element that is additively neutral and also is   multiplicatively absorbing, and have a unit element~$\one$. When speaking about idempotence in a
semiring, we always mean with respect to the additive structure.

 \item    An     \textbf{nd-semiring} satisfies all of the properties of semirings except  distributivity, i.e., it is both an additive semigroup and a multiplicative monoid.
\end{enumerate}
 \end{definition}

\subsection{Modules over a set}$ $

  \begin{definition}\label{def-gen}
  We   assume that $(\tT,\one)$   is  a  set with a distinguished element $\one$.
\begin{enumerate}
  \item
A (left) \textbf{module}  over $\tT$, or $\tT$-module, 
is an additive semigroup
$(\mathcal A,+,\zero_\mcA)$ together with a  (left)   $\tT$-action $\tT\times \mathcal A \to \mathcal A$ (denoted  as
concatenation), which is
\begin{enumerate}
 \item zero absorbing, i.e. $a \zero _\mcA = \zero_\mcA, \  \text{for all}\; a \in \tT.$

  \item \textbf{distributive},  in the sense that
$$a(b_1+b_2) = ab_1 +ab_2,\quad \text{for all}\; a \in \tT,\; b_i \in \mathcal A.$$

\item  $\one b= b$ for all $b\in \mcA.$

\end{enumerate}

  Define formally $\tTz = \tT \cup \{\zero\}$, and view $\mcA$ as a $\tTz$-module by declaring $\zero a= \zero_\mcA$ for all $a\in \mcA$.

Right modules are defined analogously.

  \item  We call  $\mcA$  a  $\tT$-\textbf{mon-module} when $\tT$  is a  multiplicative monoid
 $(\tT,\cdot,\one)$ without $\zero$, and $(a_1a_2)b = a_1(a_2b)$ for $a_i \in \tT,$ $b\in \mcA$.

\item A  \textbf{bimodule} over $\tT$, or $\tT$-bimodule,
 is a semigroup $(\mathcal A,+,\zero_\mcA)$ which is  a left and right module over $\tT,$    satisfying $(a_1b)a_2 = a_1(ba_2) $ {for all} $a_i \in \tT,\; b \in \mathcal A.$

  \item A $\tT$-\textbf{sub-bimodule} of $\mcA$ is a sub-semigroup of $\mcA$ which is closed under the $\tT$-actions.

\item We say that $\mathcal B$ is a \textbf{spanning subset}  of $\mcA$ if $\mathcal B\subset \mcA$ and the semigroup $(\mcA,+, \zero_\mcA)$ is spanned by~$\mathcal B$.

 \item   A $\tT$-bimodule $\mcA$    is \textbf{weakly admissible} if $\tTz \subseteq \mcA,$  $\zero_\mcA = \zero$, and $a(\one)=a=(\one) a$ for all $a\in \tT$, where for any elements $a,a'$ of $\tT$, $a(a')$ and $(a')a$ mean the left and right $\tT$-module actions of $a\in \tT$ on the element $a'$
 considered as an element of $\mcA$. In this case, we call  the elements of $\tT$ \textbf{tangible}.

 \item   A weakly admissible bimodule $\mcA$ over $\tT$ is \textbf{admissible} if $\tTz$ is a spanning subset of $\mcA.$


  \item  The \textbf{height} of an
element $c $ in an admissible $\tT$-bimodule $\mathcal A$ is the minimal $t$ such that $c = \sum
_{i=1}^t a_i$ with each $a_i \in \tT.$ (We say that $\zero$ has
height 0.) The \textbf{height} of $\mathcal A$   is the
supremum of the heights of its elements. (It could be $\infty$.)

\item A \textbf{$\tT$-semiring} (resp.~\textbf{$\tT$-nd-semiring})
 is a semiring (resp.~nd-semiring) $(\mathcal A,+,\zero,\cdot,\one)$, with $\tT\subset \mcA$, such that the $\tT$-module action on $\mcA$ as
$\tT$-bimodule   and multiplication within $\mcA$ are the same.
   \item An  \textbf{ideal} of a  $\tT$-{semiring} (resp.~$\tT$-nd-semiring) $\mcA$ is  a   $\tT$-submodule $\mcI$ of $(\mcA,+,\zero)$ satisfying $by,yb\in\mcI$ for all $b\in\mcA,$ $y\in\mcI.$
   \end{enumerate}
\end{definition}
\newcommand{\admissible}{\hyperref[def-gen]{admissible}}
\begin{rem}
    For a weakly \admissible\  bimodule  $\mcA$, we have
that, for all $a_1,a_2\in\tT$, $a_1(a_2)$,
the left action of $a_1$ on $a_2$ equals $(a_1)a_2$,
the right action of $a_2$ on $a_1$. This justifies the notation
$a_1a_2$ for both of them. (However, we do not necessarily have $a_1a_2\in \tT$.) This makes multiplication on $\tT$ associative in the sense that $a_1(a_2a_3) = (a_1a_2)a_3$,
since $\mcA$ is a bimodule.
\end{rem}

There is a method of making an \admissible\ $\tT$-module into a $\tT$-semiring.
\begin{proposition}(as in  \cite[Prop.\ 3.1]{AGR2})\label{fix1} If $\mcA$ is a weakly \admissible\ $\tT$-module, then take  $\mcA^\natural$ to be the additive sub-semigroup of $\mcA$ spanned by $\tT$, with the multiplication:
\begin{equation}\label{eq-mult}
\left(\sum_i
 a_i \right)\left(\sum_j a_j'\right) = \sum_{i,j}
 (a_ia_j')\end{equation}
In particular if $\mcA$ is \admissible, then  $\mcA^\natural$
has the same $\tT$-module structure as $\mcA.$
\end{proposition}

\subsubsection{The characteristic of a weakly \admissible\ $\tT$-module}$ $


 \begin{definition}$ $
 \begin{enumerate}
     \item  In a weakly \admissible\ $\tT$-module $\mathcal A$, $\Ach$ is the sub-semigroup of $(\mcA,+)$ generated by~$\one.$
     Moreover,
      write  $\mathbf k := k\one:= \one +\cdots +\one$, $k$ times, for all  $k\in \Net$; then $\Ach=\{\mathbf k\mid k\in \Net\} $.
     \item
      $\mathcal A$  has \textbf{characteristic~$(p,q)$} with $p\geq 1$ and $q\geq 0$
if   $\mathbf p + \mathbf q = \mathbf q$, with $p+q$  minimal such. We say that $\mathcal A$ has \textbf{characteristic
~$0$} if there are no such numbers $p,q$.
\item  When $\mathcal A$  has characteristic~$(p,q)$, the smallest integer $m\geq 1$ such that  $\mathbf m + \mathbf m = \mathbf m$ is called the \textbf{\period} of $\mcA$.
 \end{enumerate}
 \end{definition}
\begin{rem}

$(\Ach,+,\zero)$ is a sub-semigroup.
By the theory of finite semigroups (frying pan lemma, see for
instance \cite{semigroups1}), $p+q$ is the minimal number such that
${\mathbf 0},\ldots, \mathbf{p+q-1}$ are distinct,
so it is also the cardinality of $\Ach$.
Also there is a unique idempotent element $\mathbf m$ of the semigroup;
   $m$ is the smallest multiple of $p$ which is at least~$q$.
    \end{rem}

\begin{example}
   $ $
  Given $p,q$, define the semiring $\Net_{p,q} := \{\mathbf 0, \dots, \mathbf {p+q-1} \},$ with usual addition coupled with the rule $\mathbf {(p+q-1)} + \mathbf 1 = \mathbf q.$
    $\Net_{p,q}$ has characteristic $(p,q),$ since $\mathbf p+\mathbf q =( \mathbf {p+q-1}) + \mathbf 1 = \mathbf q. $

    $\Net_{p,0}$ is isomorphic to the ring of integers modulo $p.$
\end{example}

 \begin{rem}
In general,
       $\Ach$ is additively isomorphic either to $\Net$ or $\Net_{p,q}$, where $(p,q)$ is the characteristic of $\mcA$. (The proof is an easy exercise.)
       If $\mathcal A$ is also a semiring, then
       the isomorphisms are also semiring isomorphisms.
 \end{rem}

 In the special case that  $\mathcal{A}$ is a semifield, meaning that  $\mathcal{A}$ is a semiring and all   non-zero elements of~$\mathcal{A}$ are invertible, it follows from Theorem~1.4 of~\cite{tahar} that the characteristic of $\mathcal{A}$ is either $0$ or $(1,1)$
 or of the form $(p,0)$ for $p$ prime. In that way, in the case of semifields, the number  $p$ coincides with the standard notion of the characteristic of a semifield, which either is $0$ or $1$,
 or else it is a prime number and then $\mathcal{A}$ is a field.

  Any nontrivial idempotent semigroup has characteristic $(1,1).$
  In the tropical literature, often $\Ach$ is the idempotent Boolean semifield $\mathbf B = \{ \zero, \one\}$, in which case $\mcA$ is called a $\mathbf B $-algebra. In the supertropical case \cite{IR}, $\mcA$ has characteristic $(1,2).$

\begin{definition}\label{ZSF0}
    A semiring $\mcA$ is \textbf{zero sum free}
(ZSF) if $b_1+b_2=\zero$ implies $b_1=b_2=\zero.$\end{definition}

    We thank Kalina Mincheva for pointing out that all idempotent semirings are ZSF. Likewise we have:

\begin{proposition}
       If $ A$ has characteristic $ (1,q)$ for some $q$, then $A$ is ZSF.
\end{proposition}
\begin{proof}   Suppose $a+b = \zero.$
 First we claim that $qa+qb =\zero.$ Indeed, $qa+qb = a+(q-1)a + (q-1)b +b = a+b,$ by induction.

 But now $a = a +(qa+qb)= (a+qa)+qb = qa+qb = \zero.$
\end{proof}
\subsubsection{Pre-ordered and ordered modules}$ $

A \textbf{pre-order} is a reflexive and transitive relation. A
\textbf{partial order} (PO) is an antisymmetric  pre-order; an
\textbf{order} is a total PO.

\begin{definition}\label{modu-order} A \textbf{pre-ordered (resp.~PO, ordered) $\tT$-module} is a
$\tT$-module  $\mathcal A$ with a
pre-order relation $\le$ on $\mcA$, and, and for $\mcA$ weakly \admissible, a
\textbf{$\tT$-module pre-order} (resp.~PO, order)
relation, also denoted $\le$,  which is defined as satisfying the following, for $a, a'\in \tT$ and
$b,b_i,b_i' \in \mathcal A$:
\begin{enumerate}

\item If $b_i \le b_i'$ for $i= 1,2$, then  $b_1 + b_2 \le b_1'
   + b_2'.$
    \item  If   $b_1 \le b_1'$ then $a b_1 \le ab_1'.$
       \item  If  $a \le a'$ for $a,a'\in\tT$,  then $a b\le a' b.$
\end{enumerate}
\end{definition}

\begin{rem}\label{modu-order2}
    If $a_i \le a_i' \in \tT$ and $b_i\le b_i' \in \mcA,$
    then $\sum a_i b_i \le \sum a_i' b_i',$ by an easy induction on Definition~\ref{modu-order}(ii),(iii).
\end{rem}

Semirings may be less familiar to the reader than ordered semigroups, so we next recall
\textbf{Green's identification}.

\begin{rem}
 \label{Gre}$ $
 \begin{enumerate}
     \item  Any  abelian semigroup $(\tG,+)$ has the  partial order
     given by $a \le b$ iff $a=b$  or $a+b=b.$ This is a total order iff the semigroup is bipotent.

   \item   More generally,  recall  that a (sup)-\textbf{semilattice} is a set $\tG$ with a sup function
$\vee$. We can get a semilattice from an idempotent semigroup by defining $a \vee b = a+b,$ and in the other direction, any semilattice defines an idempotent semigroup.

When $\tG$ is a multiplicative monoid compatible with a sup $\vee,$ we get
an idempotent semiring $(\tG\cup\{\zero\},\vee, \cdot)$
by adjoining to $\tG$ a bottom element denoted by $\zero$.
 \end{enumerate}
\end{rem}

 \section{Pairs}\label{prs}

Before getting to  our main results, we need to introduce the main concepts of the theory of pairs, bringing in the properties that we need. We follow \cite{JMR1}, where pairs were defined in the case of $\tT$-nd-semirings, $\tT$-semirings, and $\tT$-modules. After that,  we   state several structure theorems in this section, with proofs further on.

 \begin{definition}\label{symsyst} $ $
 \begin{enumerate}
     \item A
 \textbf{ $\tT$-pair}  $(\mcA,\mcA_0)$ is a  bimodule $\mcA$ over $\tT$ together with a sub-bimodule
   $\mcA_0$. 

We omit $\tT$ in the notation when $\tT$ is understood, and write \textbf{pair} for $\tT$-pair.

             \item A
 \textbf{(weakly) \admissible\ pair}  is a pair  $(\mcA,\mcA_0)$ such that $\mcA$ is a (weakly) \admissible\ $\tT$-bimodule with
 $\mcA_0\cap \tTz =\{ \zero\}$.

     \item  When   $\mcA$ also is a $\tT$-mon-module, we call $(\mcA,\mcA_0)$ a \textbf{mon-pair}.  When in addition $\tT$ is a multiplicative group we call $(\mcA,\mcA_0)$  a
 \textbf{gp-pair}.

\item
A weakly \admissible\  pair $(\mcA,\mcA_0)$ is of the \textbf{first kind} if
$\one + \one \in \mcA_0$  and of the
\textbf{second kind} if  $a + a \notin \mcA_0$ for all $a\in \tT$.

   \item  Our maps $f$ from a $\tT$-bimodule $\mathcal A$ to a $\tT'$-bimodule  $\mathcal A'  $  always will
 satisfy
\begin{enumerate}
    \item $f (\tT) \subseteq \tT',$
\item $f(\zero) = \zero',$  and $f(\one) = \one'.$
\end{enumerate} The map $f$ is \textbf{multiplicative} if    $f(ab) = f(a) f(b)$ and
   $f(ba) = f(b) f(a)$ for all~$a\in \tT, \ b \in \mcA.$
   \item\label{weak-morphism} A \textbf{weak morphism of pairs} $(\mcA,\mcA_0)\to (\mcA',\mcA'_0)$ is a multiplicative map $f:\mcA\to \mcA'$ satisfying $\sum_{i=1}^n b_i \in \mcA_0$ implies $\sum _{i=1}^n f(b_i)\in \mcA'_0$, for $b_i\in \mcA$. (In particular $f(\mcA_0)\subseteq \mcA_0'.$)


 \item An  \textbf{nd-semiring pair} is an
 \admissible\ mon-pair $(\mcA,\mcA_0)$ such that  $\mcA$ is an nd-semiring.

  \item A  \textbf{semiring pair} is an
 \admissible\ mon-pair $(\mcA,\mcA_0)$ such that  $\mcA$ is a  semiring. \footnote{When one wants matrices over an \admissible\ semiring pair to be an \admissible\ semiring pair, one may weaken the definition of mon-modules, and mon-pairs, in order to
 permit
 straightforward sets of tangible elements $\tT^{\mathcal M}$ of the semiring of matrices
 which have zero-divisors
 (such as the matrix units over $\tT$ together with the identity matrix), but such that $\tTz^{\mathcal M}$
 is a monoid.}

   \end{enumerate}
\end{definition}




\begin{rem}\label{tr1a}
    Any $\tT$-module $\mcA$ can be viewed trivially as a pair by putting $\mcA_0 = \{\zero\}.$
\end{rem}

 For the remainder of this paper, we make the following assumptions:
\begin{assump}\label{Note1}\

   \begin{enumerate} \item $(\mcA,\mcA_0)$ is a  weakly \admissible\ pair.


       \item
     If   $a b \in \mcA_0$  or $ba  \in \mcA_0$, for  $a\in \tT$ and $b\in \mcA$, then $b\in \mcA_0.$
          \item If $a b _1= ab_2$ or $b_1a =b_2a$, for  $a\in \tT$ and $b_1,b_2\in \mcA$, then  $b_1=b_2.$ In particular $\tT$ is cancellative.
   \end{enumerate}
In view of (ii) and (iii), when  $(\mcA,\mcA_0)$ is a  mon-pair we may localize at $\tT$ and assume that $(\mcA,\mcA_0)$  is a gp-pair.
\end{assump}

\subsection{Property~N and negation maps}

   In nontrivial pairs, $\mcA_0$ takes the place of $\zero$. The significance is that since $\tT$-bimodules need not have negation (for example, $\Net$), $\zero$ has no significant role except as a place marker in linear algebra. However, we need some weaker property than negation  to proceed.

\subsubsection{Property~N}\label{subsec-PN}$ $

For every $b\in \mcA$, we say that $b^\dag$ is a \textbf{quasi-negative} of $b$,
if $b +b^\dag\in \mcA_0$,
with $b^\dag\in \tTz$ when $b\in\tTz$.
Note that a quasi-negative need not be unique.
\begin{definition}\label{propN0}$ $
We say that
$(\mcA,\mcA_0)$ satisfies
 \textbf{weak Property~N}  if  $\one$ has a quasi-negative $\one^\dag$ with
$\one^\dag \tT\subseteq \tT,$ and $\one^\dag b = b \one^\dag$ for each $b\in \mcA$.  In this case, we fix $\one^\dag$, and define $e:= \one+\one^\dag\in \mcA_0.$
Then, we denote $b^\dag = b\one^\dag$,  and $b^\circ = b+ b^\dag $, for all $b\in \mcA.$  Let $\mcA^\circ = \{b^\circ : b\in \mcA\},$ and  $\tT^\circ = \{a^\circ : a\in \tT\}.$ Also define
$e^+ := e + \one.$ \end{definition}
\newcommand{\eplus}{\hyperref[propN0]{\ensuremath{e^{+}}}}
\newcommand{\e}{\hyperref[propN0]{\ensuremath{e}}}

\begin{lemma}\label{ep} $ $ For all $a\in \tT$ and $b, b_i\in \mcA$,
\begin{enumerate}
    \item  $(\sum b_i)^\circ = \sum b_i ^\circ $.
 \item
  $(a b)^\circ =a  b^\circ$ and $(ba )^\circ =  b^\circ  a$.
  \item  $a^\circ =a e=e a \in \mcA_0$.
  \item $\tT^\circ = \{a^\circ : a\in \tT\}=\tT e=e\tT.$
\end{enumerate}
\end{lemma}
\begin{proof} (i)
    $(\sum b_i)^\circ = \sum b_i  + (\sum b_i)\one^\dag  = \sum b_i  + \sum b_i \one^\dag =  \sum (b_i +b_i\one^\dag) =  \sum b_i ^\circ .$

(ii) $(ab)^\circ = ab+ (ab) \one^\dag = ab+ a (b \one^\dag)= a(b+ b \one^\dag)=ab^\circ$, and similarly $ (ba)^\circ = b^\circ a.$

(iii) Take $b=1$ in (ii).
\end{proof}
 According to Definition~\ref{propN0}, $\one^\dag$ and $e$ need not be uniquely defined.

 (iv) Clear.

 \begin{rem}\
 \begin{enumerate}
     \item For any pair $(\mcA,\mcA_0)$ of the first kind, we can always take $\one^\dag = \one$ to see that $(\mcA,\mcA_0)$ satisfies
 {weak Property~N}.
     \item In the above lemma, (i) implies that $\mcA^\circ$ is closed under addition. 
 \end{enumerate}
 \end{rem}

\begin{definition}\label{propN}$ $ A pair $(\mcA,\mcA_0)$   satisfying weak Property~N  satisfies
 \textbf{Property~N} when   $ a_1 + a_2 = a_1^\circ $ for each $a_1,a_2\in \tT$ such that  $a_1+a_2 \in \mcA_0$. In particular,  the element $e=\one+\one^\dag$ now is independent of the choice of $\one^\dag$, and so it is uniquely defined, and $\one^\dag +(\one^\dag)^\dag =e$.

Now let $(\mcA,\mcA_0)$ be a pair satisfying \PropertyN.
 \begin{enumerate}
  \item  $(\mathcal
A, \mcA_0)$  is
 \textbf{$e$-final} if
 $e^+=e$.

 \item  $(\mcA,\mcA_0)$  is $e$-\textbf{idempotent} if   $e+e=e.$

 \item \label{propN-iv} Define left and right actions of $\tT \cup \tT^\circ$ on $\mcA,$ by  defining $a^\circ b := (a b)^\circ$ and $b a^\circ  := (b a)^\circ$ for $a\in \tT,$ $b\in \mcA$.
\end{enumerate}
  \end{definition}

   Although
 classical pairs are not $e$-final, the pairs normally used in tropical
mathematics, such as in \Eref{modtr1} below, are
$e$-final.

\begin{lemma} \label{tt} Any $e$-final metatangible pair satisfies  $\mcA = \tT \cup \tT^\circ.$
\end{lemma}
\begin{proof}
    Using Lemma~\ref{bip1}, we never leave $\tT \cup \tT^\circ$ when taking sums.
\end{proof}

 \begin{example}  \label{tr0}$ $\begin{enumerate}
     \item
$\mcA$ is a semiring and $\mcA_0=\mcI\neq \mcA$
 is an ideal of $\mcA$.
 We could take instead  $\tT = \mcA \setminus \mcI$ to get an \admissible\ pair when $\mcI$ is a prime ideal.

When $\mcA$ is a ring,   $(\mcA,\mcA_0)$ is another way of describing
 $\mcA/ \mcI$.  
 But in general, the additive automorphism given by $b \mapsto b \one^\dag$   need not have order 2. For example, take $\mcA = (\Net,+),$ $\mcA_0 = p\Net,$    and $\one^\dag = p-1.$
 Property N fails, and this example often differs considerably from the situation that we are studying in this paper.

  \item When $\mcA_0 =\zero, $ we call $(\mcA,\zero)$ the \textbf{trivial pair}. In this case Property N yields classical negation, for if $a+a_i = \zero$ for $i=1,2$, then $a_2 = (a_1+a)+a_2 = a_1+(a+a_2) = a_1$; hence $\mcA^\natural$ is then a ring.

        \item  At the other extreme, \cite[Example~2.21]{AGR2}  provides the  pair $(\mcA ,\mcA_0 )$ with $\mcA = \tTz \cup \{e\}$, $\tT$~an arbitrary monoid, and $\mcA_0 = \{ \zero, e\},$
 satisfying $a_1+a_2= e$ for all $a_1,a_2 \in \tT.$ The pair $(\mcA ,\mcA_0 )$   satisfies \PropertyN, and we need to cope with this example when formulating theorems.
\end{enumerate}
  \end{example}

\begin{lemma}\label{lem-tcircmodule}
     If {\PropertyN} holds, then  $\mcA$ is a
     $(\tT \cup \tT^\circ)$-bimodule under the   action given in \ref{propN-iv} of Definition~\ref{propN}.
\end{lemma}
\begin{proof}
Let us first show that  $\mcA$ is a left and right
     $(\tT \cup \tT^\circ)$-module.
    Since $\mcA$ is already a $\tT$-bimodule, we need only to check distributivity with respect to elements of $\tT^\circ$.
 We appeal to \Cref{ep}(ii).
 Using the definition of the action of $\tT^\circ$, and the distributivity over $\tT$, we have, for all $a\in \tT$
 and $b,b'\in\mcA$,
 $$a^\circ (b+b') =(a (b+b'))^\circ=(a b+ a b')^\circ= (a b)^\circ+ (a b')^\circ=  a^\circ b+a^\circ b'.$$
 Next, let us show the bimodule associativity property.
 Given $a_1,a_2\in\tT$ and $b\in\mcA$,
 we have $(a_1 b) a_2=a_1 (b a_2)$
 since  $\mcA$ is a $\tT$-bimodule.
Using
 the definition of the action of $\tT^\circ$,
 we  obtain
 $$(a_1^\circ b) a_2=(a_1 b)^\circ a_2=((a_1 b) a_2)^\circ= (a_1 (b a_2))^\circ= a_1^\circ (b a_2).$$
 Moreover, using the previous equations, we also obtain
 $(a_1^\circ b) a_2^\circ=((a_1^\circ b) a_2)^\circ=(((a_1 b) a_2)^\circ)^\circ$, then, by symmetry and using that $\mcA$ is a $\tT$-bimodule, we obtain $(a_1^\circ b) a_2^\circ=((a_1 (b a_2))^\circ)^\circ= a_1^\circ (b a_2^\circ )$.
 This displays the bimodule associativity property
 with respect to any element of  $(\tT \cup \tT^\circ)$.
\end{proof}




\begin{lemma}\label{modu1}$ $  Suppose that $(\mcA,\mcA_0)$ satisfies \PropertyN.
\begin{enumerate}


\item $e^\dag = e \one^\dag = \one^\dag e= e.$

 \item   In the following, use the action given in~\ref{propN-iv} of Definition~\ref{propN}.
 \begin{enumerate}
     \item   $eb = be=b^\circ$ for all $b\in \mcA.$ (In particular $ee =e^\circ= e+e.$)
 \item $\mcA^\circ$ is a $(\tT\cup \tT^\circ)$-sub-bimodule of $\mcA$.
\item If $e\mcA  \subseteq \mcA_0$, 
then $\mcA^\circ \subseteq \mcA_0.$
 \end{enumerate}
\end{enumerate}
\end{lemma}
\begin{proof}

(i) By definition $e^\dag = e \one^\dag$; $  \one^\dag +(\one^\dag )^2= (\one + \one^\dag )\one^\dag  =  e \one^\dag \in \mcA_0 $, so \PropertyN\ implies this equals $  \one^\dag +\one = e.$


(ii) (a) Using the action given in~\ref{propN-iv} of Definition~\ref{propN},
we have $eb= (\one b )^\circ=b^\circ=be$.

(b)
Using the proof of \Cref{lem-tcircmodule}, we have that $b^\circ+(b')^\circ= (b+b')^\circ,$ $a b^\circ= (ab)^\circ$,  and $a^\circ b^\circ= (a b^\circ)^\circ$,  for all $a\in\tT$ and $b,b'\in \mcA$, which shows that $\mcA^\circ$ is
 a $(\tT\cup \tT^\circ)$-sub-bimodule of $\mcA$.

(iv) By (ii).
\end{proof}

\begin{remark}\label{finalplus}
If $(\mathcal A, \mcA_0)$  is $e$-final, then  also
 $e+\one^\dag=e$. Indeed, $e+\one^\dag=e \one^\dag+\one^\dag=( e+\one) \one^\dag=e \one^\dag=e$, by \Cref{modu1}.
    \end{remark}

\begin{lemma}\label{propN1} Suppose that $(\mcA,\mcA_0)$ is a mon-pair satisfying
 \PropertyN.
\begin{enumerate}

\item $\tTz^\circ$ is a monoid, denoted $\tTz^\odot,$ under the new operation $\odot$ given by
$a_1^\circ \odot a_2^\circ: = (a_1 a_2)^\circ.$ When $\tT$ is a group, $\tG := \tTz^\odot \setminus \{\zero\}$ is a group with multiplicative unit $e.$

  \item If $a\in \tT$ is invertible then, in $\tG,$
  the inverse of $a^\circ$ is $(a^{-1})^\circ$.

     \item  If $ a\in \tT$ is invertible and $\one^\dag$ is uniquely defined, then $a + b \in \mcA_0$ with $b\in \tT$ implies $b = a\one^\dag$.

\end{enumerate}
\end{lemma}
\begin{proof}
(i) Immediate from Lemma~\ref{ep}(iii).

(ii) $( a^\circ) ( (a ^{-1})^\circ)=  (a  a ^{-1})^\circ = e.$

   (iii)    $a +b \in \mcA_0$ implies $\one +a^{-1} b   = a^{-1} (a + b ) \in \mcA_0, $ so $a^{-1} b = \one^\dag,$ and
 $ b = a\one^\dag.$
\end{proof}

An {nd-semiring pair}  $(\mcA,\mcA_0)$ satisfying weak Property~N
   is $e$-\textbf{distributive} if  $e\sum a_i=\sum a_i^\circ$ in the multiplication of $\mcA.$
(In particular  $e^2= e (\one +  \one^\dag) = e+e$.)

\begin{lemma}\label{dista} If  an {nd-semiring pair}   $(\mcA,\mcA_0)$,
is $e$-distributive, then  the multiplication inside $\tT^\circ$ agrees with \begin{enumerate}
    \item the action given in~\ref{propN-iv} of Definition~\ref{propN}. 
    \item the semiring multiplication on $\mcA^\natural$.
    \end{enumerate}
\end{lemma}
\begin{proof}
(i)
    $a_1^\circ a_2^\circ = a_1^\circ (ea_2) = (a_1^\circ e)a_2 = (a_1 e^2)a_2 = a_1 (e^2) a_2 .$

    (ii) $a^\circ b  = (ae)b = (ab)e = (ab)^\circ.$
\end{proof}

 \subsubsection{Negation maps}$ $

 Here is the analog of a notion from \cite{Row21}.
  \begin{definition}\label{negmap}$ $
\begin{enumerate}
     \item
 A \textbf{negation map} on  $(\mathcal A,\mcA_0)$
is a semigroup automorphism of
$ (\mathcal A,+,\zero) $ of order $\le 2$,  written
$a\mapsto (-)a$, under which $\tT$ is invariant, with $\one +((-)\one)\in \mcA_0,$ which also
 respects the $\tT$-action in the sense that
\begin{equation}
    \label{negd} ((-)a)b = (-)(ab) = a ((-)b), \qquad ((-)b)a = (-)(ba) = b ((-)a),
 \qquad \forall a \in \tT,\ b \in \mathcal A.
\end{equation}

\item
 When $(\mathcal A,\mcA_0)$ is a pair with a negation map $(-)$, we also call $(\mathcal A, \mcA_0, (-))$  a \textbf{triple}. In that case, the pair $(\mathcal A,\mcA_0)$  is \textbf{uniquely negated} if $a+a'\in \mcA_0$ implies $a' = (-)a,$ for $a,a'\in \tT.$ In this case, we also say that $(\mathcal A,\mcA_0)$ has \textbf{unique negation}.
 \end{enumerate}
Note that $(-)\mcA_0 = ((-)\one) \mcA_0 = \mcA_0.$ We write $b(-)b'$ for $b+((-)b')$ for $b,b'\in \mcA$.
\end{definition}
\begin{example}\label{tru} A pair satisfying Property~N, but is not    uniquely negated.
    $\mcA =\{ 0,1,2,3\}$, $\tT=\{1,2\}$, and $\mcA_0 = \{0,3\},$ with addition truncated so that $n_1+n_2 =3$ whenever $n_1+n_2\ge 3$ with the usual addition. Then $2+2 = 2+1 = 3$. This sort of example is elaborated in \Cref{metex}, after we bring in more terminology.
\end{example}

\begin{lemma} \label{3.15} $ $
\begin{enumerate}
 \item   If $(\mathcal A,\mathcal A_0)$ has a  negation map $(-)$, then  $(\mathcal A,\mathcal A_0)$ has weak Property~N, with $ \one^\dag =(-)\one.$  If $(\mathcal A,\mathcal A_0)$ is uniquely  negated, then  $(\mathcal A,\mathcal A_0)$ has  Property~N (and $1+a = e$ implies $a = \one^\dag$).
 \item If  $(\mathcal A,\mathcal A_0)$ satisfies weak Property~N, then we can obtain a negation map by defining $(-)b = b^\dag =b \one^\dag=\one^\dag b$ for $b\in \mcA$, as soon as the   following properties hold:
 \begin{enumerate}
     \item $(\one^\dag)^2 = \one;$
     \item $(a\one^\dag) b= a (\one^\dag b)$ and  $b(\one^\dag a) = (b \one^\dag )a$, for all $a\in \tT$ and $b\in \mcA$
     (which holds when $\mcA$ is a mon-pair).
      \item $\mcA^\circ \subseteq \mcA_0$ (which holds when $\mcA$ is admissible).
 \end{enumerate}

  \item If  $(\mathcal A,\mathcal A_0)$ has weak Property~N and $\one$ has a unique quasi-negative $\one^\dag $,   then   for any invertible $a\in \tT$, if $a+a' \in \mcA_0$ for $a'\in \tT,$ then $a' =a^\dag.$
 \item If   $(\mathcal A,\mathcal A_0)$ is  a gp-pair with weak Property~N and $\one^\dag $ has a unique quasi-negative, then
 $(\mathcal A,\mathcal A_0)$ is uniquely negated.
\end{enumerate}\end{lemma}
\begin{proof} (i) Clear.

(ii) Set $(-)b := \one^\dag b= b\one^\dag$ for $b\in \mcA$.
  Then
      $ (-)((-)b) = \one^\dag(\one^\dag b ) = \one b = b,$ and
  $ b(-)b = b+b\one^\dag = b^\circ \in  \mcA^\circ \subseteq \mcA_0.$

Also for $a\in \tT$ and $b\in \mcA$, we have
$ (-a)b=(\one^\dag a) b=\one^\dag (a b)=(-) (ab)$ by assumption.
We also have $ (\one^\dag a) b= (a\one^\dag) b= a (\one^\dag b)=a ((-) b)$. The symmetric equations also hold, so
$(-)$ is a negation map.
 (iii) If $a+a' \in \mcA_0$ then   $a' = a \one^\dag = a^\dag$, by Lemma~\ref{propN1}(iii).

(iv)
$(\one^\dag)^2 +\one^\dag = \one^\dag(\one^\dag +\one) =e = \one +\one^\dag,$ implying  $(\one^\dag)^2 =(\one^\dag)^\dag =\one,$ so apply (ii).
\end{proof}

\begin{rem}
    \label{fk} If $(\mcA,\mcA_0)$ is of the first kind,
then the identity  is a negation map.
\end{rem}

 A negation map is  called a ``symmetry'' in
\cite[Definition~2.3]{AGG2}.
  Although negation maps play the key role in
\cite{AGG2,Row21,AGR2} we shall see in this paper that \PropertyN, sometimes in conjunction with $e$-idempotence, suffices to generalize much of the theory.

\begin{lemma}$ $ \begin{enumerate}
  \item If $(\mcA,\mcA_0)$ is a   pair of the first kind, then
  $a+a = a( \one + \one) \in \mcA_0$ for all $a\in \tT$.
  In particular,  $a +a \ne a.$

  \item   A    pair  $(\mcA,\mcA_0)$
is   of the second kind, if and only  if   $\one + \one \notin \mcA_0.$     \end{enumerate}
\end{lemma}
\begin{proof}    (i) $a+a = a( \one + \one) \in \mcA_0$ for all $a\in \tT$.

    In particular, $a + a \notin \tT,$ so  $a + a \ne a.$

(ii) $  (\Rightarrow)  $ Immediate.

$  (\Leftarrow)  $
    $a + a = a(\one + \one)\notin \mcA_0$ for all $a\in \tT.$ by~Assumption~\ref{Note1}(iii).
\end{proof}

\subsubsection{The $\mcA_0$-characteristic of a pair}$ $

Since $\mcA_0$ takes the role of $\zero$, we modify the characteristic. We consider the first kind, since we want $\mathbf{2} \in \mcA_0$.

\begin{definition}
    The $\mcA_0$-\textbf{characteristic} of a pair $(\mcA,\mcA_0)$ of the first kind is the smallest $k>0$ such that $\mathbf{2k+1} \in\mcA_0.$ (If there is no such $k$ then $(\mcA,\mcA_0)$ has $\mcA_0$-\textbf{characteristic} $0.$
\end{definition}

\begin{lemma}\label{A0char}
    If $(\mcA,\mcA_0)$ has $\mcA_0$-{characteristic} $k>0$ then $\mathbf{m} \in\mcA_0$ for all $m\ge 2k.$
\end{lemma}
\begin{proof}
    An obvious induction, since $\mathbf{2} \in \mcA_0.$
\end{proof}

\subsection{Balance relations and surpassing relations in pairs}$ $

 In the absence of a negation map, we need more structure in order to be able to build a theory.

\subsubsection{The balance relation}$ $

Lacking negation, we need to weaken equality.
One key relation in \cite{AGG2,AGR2,Row21} is the \textbf{balance relation}, which is presented
axiomatically here.

\begin{definition}\label{balr}$ $\begin{enumerate}
    \item
    A \textbf{balance relation}  $\nabla $  (different from the use of~$\nabla $ in~\cite{IzhakianRowen2008Matrices2}) is a symmetric and reflexive (but not necessarily transitive) relation
  also satisfying the properties:
\begin{enumerate}
    \item $\nabla $ is preserved under addition.
    \item $\nabla $ is preserved under multiplication by elements of $\tT$ (i.e., $b_1 \nabla b_2$ implies $ab_1 \nabla ab_2$ for $a\in \tT$).   \item $b \nabla \zero$ for all $b\in \mcA_0.$
\end{enumerate}

 \item A balance relation is $\tT$-\textbf{cancellative} if $ab_1 \nabla ab_2$ implies $b_1 \nabla b_2$, for all $a\in \tT.$

\item  An element $b\in \mcA$ \textbf{balances} a set $S = \{b_i: i \in I\},$ written $b\nabla S,$ if $b \nabla b_i$ for all $  i\in I. $
   \end{enumerate}

\end{definition}

\begin{example}\label{bal1}$ $
    \begin{enumerate}
         \item Equality is the \textbf{trivial} balance relation.

        \item  When $(\mathcal A, \mcA_0)$ satisfies weak Property~N, fix $\one^\dag$ and
    write   $b_1\, \nabladag \, b_2$, if
   $b_1 + b_2^\dag\in  \mcA_0$  and $b_2 + b_1^\dag\in  \mcA_0$,   which   is a balance relation because $\mcA_0$ is a sub-bimodule of $\mcA$.

 \item  When $(\mathcal A, \mcA_0)$ has a  negation map, (ii) becomes ``$b_1\, \nabladag \, b_2$ if
   $b_1 (-) b_2\in  \mcA_0$.''  In this case we write $\nablaminus$ for $\, \nabladag $.  \footnote{This holds in particular for $(\mathcal A, \mcA_0)$ of the first kind, cf.~\Cref{fk}. For pairs of the second kind, we use the more subtle definition of \Cref{bal1}(iv),  to compensate for the lack of a negation map.}

   \item  When $(\mathcal A, \mcA_0)$ is a pair, write $b_1\nablaT b_2$, $b_i\in \mcA$,
     if 
   we can write $$b_i = b_{i,0} +\sum _{j=1}^t a_{i,j},\qquad b_{i,0}\in \mcA_0,\ a_{i,j}\in \tT,\ i = 1,2,\ t\geq 0,$$ such that, for each $j\geq 1,$      there is $a_j\in \tT$  such that $a_{1,j} +a_j \in  \mcA_0$ and $a_{2,j} +a_j \in  \mcA_0$.

     If $(\mcA,\mcA_0)$ is a mon-pair satisfying weak property~N,     and $\mcA = \mcA_0 +\mcA^\natural$
     (in particular if $\mcA$ is admissible), then $\nablaT$ is a 
     balance relation. 
     Indeed,  $\nablaT$ is clearly symmetric, it is preserved under addition and multiplication because $\tT$ is a monoid and $\mcA_0$ is a semimodule. Moreover, $\nablaT$ is reflexive since writing $b = b_0 + \sum a_i$ for $b_0\in \mcA_0$ and $a_i\in \tT,$  we have $a_i+a_i^\dag\in \mcA_0$, hence $b\nablaT b$.
       \end{enumerate}
\end{example}

 \begin{rem}\label{nablagood} $ $
    \begin{enumerate}
        \item  $a\, \, \nabladag \,\, (b+c)$ if and only if $ (a+c^\dag) \, \, \nabladag \, \, b.$
  (Both sides are equivalent to $a+b^\dag + c^\dag \in \mcA_0.$)

     \item  The balance relations of \Cref{bal1}(i), (ii), and (iii) are cancellative in view of Assumption~\ref{Note1}. On the other hand, (iv) is cancellative when $\tT$ is a group.

       \item  The balance relations of \Cref{bal1}(i), (ii), and (iii) satisfy $b\nabla \zero$ iff $b\in \mcA_0$.

       \item      If $\mcA_0 = \{\zero\}$, then each of \Cref{bal1} is equality. Indeed, for $\nablaT$, if $a_{1,j} + a_j = \zero = a_{2,j}+a_j,$ then $a_{2,j} = (a_{1,j}+a_j)+a_{2,j} = a_{1,j} +(a_j+a_{2,j}) = a_{1,j}.$
    \end{enumerate}
\end{rem}

\subsubsection{Surpassing relations}$ $

 We  modify slightly the    {surpassing relation} of
 \cite[Definition~1.31]{Row21} and
\cite[Definition~2.4]{JMR1}.


\begin{definition}\label{precedeq07}$ $
 \begin{enumerate}  \item
 A \textbf{pre-surpassing relation} on a pair $(\mathcal A,\mcA_0)$, denoted
  $\preceq$, is a $\tT$-module pre-order satisfying    $\mcA_0 \subseteq \mathcal A _{\operatorname{Null}},$
   where $A _{\operatorname{Null}} : = \{ c\in \mcA: \zero \preceq c\}. $

 \item    A \textbf{surpassing relation}  is a {pre-surpassing relation} satisfying  the property:
 \begin{itemize}

\item  $a_1 \preceq a_2 $ for $a_1 ,a_2 \in   \tTz$
    implies $a_1 =a_2.$
\end{itemize}


 \item A $\preceq$-\textbf{morphism} $f:(\mcA,\mcA_0)\to (\mcA',\mcA'_0)$ is a multiplicative map $f:\mcA\to \mcA'$ satisfying $$f\left(\sum_{i=1}^n b_i\right) \preceq \sum _{i=1}^n f(b_i)$$    for each $n$, $b_i\in \mcA$.
   \end{enumerate}
\end{definition}

\begin{example}\label{Su1}$ $
    \begin{enumerate}
     \item  For any sub-$\tT$-module $\mathcal I$ of $\mathcal A$
containing $\mcA_0$, we write $a
\preceq_{\mathcal I}
 c$   if $a+b = c$ for some $b \in \mathcal I$.

The relation $(\preceq _{ \mathcal I} )$ is a pre-surpassing relation, and is a surpassing relation if it also
satisfies
\begin{itemize}\label{int18}
\item $(a + \mathcal I) \cap \tT = \{a\}, \
\forall a \in \tT.$
\end{itemize}
\noindent (This needed for $\preceq _{ \mathcal I} $ to satisfy \Dref{precedeq07}(ii).)
 In particular, taking $\mcI = \mcA_0,$     we denote   $\preceqzero$ for the pre-surpassing relation  $\preceq_{\mcA_0}$.

     \item     For a classical example, using Zorn's lemma, one could take $\mcA_0$ to be an ideal of a semiring $\mcA$ maximal with respect to being disjoint from $\tT, $  and we could take $\preceq$ to be equality.
 \end{enumerate}
\end{example}

\begin{INote}
 The point of \Cref{precedeq07} is to strengthen balancing when necessary, yielding  many  theorems generalizing classical algebra,
 noting that   surpassing restricts to equality on $\tT$.
 \end{INote}.

\begin{lemma}\label{impl}  $b_1\preceqzero b_2$ implies $b_1 \nabla b_2$ for any balancing relation $\nabla$.

\end{lemma}\label{symba}
\begin{proof}  Suppose $b_2 = b_1 +c$ where $c\in \mcA_0$. $b_1\ \nabla \ b_1$
by definition. Moreover  $c\nabla \zero$ by definition.
 $\nabla$ is additive, so we get the result by adding these two relations.
\end{proof}

\begin{lemma} In a semiring pair,
    If  $b_i \preceq_{\mcI} b_i'$ for $i =1,2,$ then $b_1b_2 \preceq_{\mcI} b_1'b_2'$.
\end{lemma}
\begin{proof}
    Write $b_i' = b_i +c_i$ for $c_i \in I$. Then $b_1'b_2' = b_1b_2 + (b_1 c_2 +c_2 b_1),$ and $b_1 c_2 +c_2 b_1\in \mcI.$
\end{proof}


\begin{lemma}\label{nae} If $(\mathcal A ,\mathcal A_0)$ is uniquely negated, then the pre-surpassing relation $\preceq_0$ is  a surpassing relation.
\end{lemma}
\begin{proof} If $a_1 + b = a_2$ for $a_i\in \tT$ and $b\in \mcA_0,$ then $a_2(-)a_1 =a_1 (-) a_1 +b \in \mcA_0$, implying $a_2=a_1.$
\end{proof}

\begin{lemma}
 Any $\preceq$-{morphism} $f:(\mathcal A ,\mathcal A_0)\to (\mathcal A ,\mathcal A_0)$ is a   weak morphism.
\end{lemma}
\begin{proof}
    If $\sum  b_i \in \mcA_0$ then $\sum  b_i \succeq \zero,$
    so $\sum f(b_i) \succeq f(\sum b_i) \succeq \zero.$
\end{proof}

\subsection{Metatangible pairs (with statements of \Cref{theoalphaA}, \Cref{theoalphaD}, and \Cref{theoalphaB})}$ $

Recall   that \cite[Definition~2.25]{Row21}  called a triple
 {\it metatangible} if $a_1+a_2 \in \tT $ for any $a_1,a_2 $
in~$\tT$ with $a_2 \ne (-) a_1$.
Metatangible triples are the mainstay of \cite{Row21}.

  Surprisingly, we retain many  results about metatangible from  \cite{Row21} when we weaken triples to pairs, replacing the negation map by \PropertyN.

\begin{definition}\label{metat}$ $
 \begin{enumerate}\item  $(\mcA,\mcA_0)$ is \textbf{weakly metatangible}
  if   $a_1+a_2 \in \tT \cup \mcA_0$ for any $a_1,a_2 $
in~$\tT.$

   \item  $(\mcA,\mcA_0)$ is  \textbf{metatangible} if it is  \admissible, weakly metatangible, and satisfies \PropertyN.
\end{enumerate}\end{definition}

Any metatangible pair is weakly metatangible, but the converse might fail. Metatangible pairs lie at the heart of this paper, in view of the following theorem.
\begin{theoalpha}[\protect{Generalizing \cite[Theorems 7.28, 7.32]{Row21}, which was proved for triples}]
\label{theoalphaA}For any metatangible pair $(\mathcal A, \mcA_0)$, every element $c$ of $ \mathcal A$
has a ``uniform presentation'', by which we mean there exists
$c_\tT \in \tT$ such that for the height $m_c$ of $c$, we have $ c_\tT \in \tT$,  with
either: \begin{enumerate} \item  The pair $(\mathcal A, \mcA_0)$ is  of the first kind, with  $c = m_c c_\tT,$ or
  \item  The pair $(\mathcal A, \mcA_0)$ is  of the second kind, and either
\begin{enumerate} \item $c =\zero$, and $m_c=0.$  \item   $c = c_\tT \in \tT$ and $m_c=1$, or
  \item   $c = c_\tT^\circ$ and $m_c=2$.
\end{enumerate}
 \end{enumerate}
In (i),  for some $m_0 \ge 0,$  $\mcA_0 = \{ m_c c_\tT : m_c \ge m_0 \text{ or $m$ is even}\}.$

In (ii), $\mcA = \{\zero\}\cup \tT\cup \tT^\circ$ and $\mcA_0 = \tTz^\circ$.
  \end{theoalpha}
The proof of \Cref{theoalphaA} is given in \Cref{up}.

\begin{lemma}\label{tanplu2}  Suppose $(\mcA,\mcA_0)$ is a metatangible pair.
\begin{enumerate}
    \item  $\tTz +\mcA_0 = \mcA$.
    \item  For any $b_1,b_2\in \mcA,$ there are $a_1,a_2\in \tT$ such that  $a_1b_1+a_2b_2\in\mcA_0$.
    \end{enumerate}
\end{lemma}
\begin{proof} (i) Take $b = \sum _{i=1}^t a_i \in \mcA,$ for $a_i \in \tT$. By induction
  on $t$,  $\sum _{i=1}^{t-1} a_i = a +b'$ for $a\in \tT,$ $b'\in \mcA_0,$ and thus $b = b' +(a+a_t),$ where $a+a_t\in \tT \cup \mcA_0.$

  (ii) Using (i), write $b_i = a_i' +b_i'$ for $a_i' \in \tT$ and $b_i'\in \mcA.$ Take $a_1 = a_2'$ and $a_2 = {a_1'}^\dag.$ Then $$a_1b_1+a_2b_2 = a_2'(a_1'+b_1') + {a_1'}^\dag(a_2'+b_2') = a_2'(a_1'+{a_1'}^\dag)+a_2'b_1+a_1'b_2 \in \mcA_0.$$
\end{proof}

Moreover, Theorem A  implies at once that  $\mcA = \tT \cup \mcA_0$ for any  metatangible pair of the second kind.
  \begin{example}  $\mcA = (\Net,+)$ and $\tT = \{  1\}.$ Assume that $(\mcA,\mcA_0)$ is a  pair is of the first kind.  Then $e= 2    \in \mcA_0.$
  \begin{enumerate}
  \item $\mcA_0 = \{ n : n \ne 1 \}.$ Then $\mcA = \tT \cup \mcA_0$, but $\mcA_0 \ne \mcA^\circ = 2 \Net  \ne \{0,2\} =\tTz^\circ. $

      \item   $\mcA_0 = 2\Net,$  so
 $\eplus  =\one + e  =\mathbf{3}\notin \tT \cup \mcA_0$.
  \end{enumerate}

  \end{example}

\begin{lemma}\label{ZSF} Suppose that $(\mcA,\mcA_0)$ is a weakly metatangible pair.
\begin{enumerate}
      \item   If $(\mcA,\mcA_0)$  satisfies weak Property~N, then for any $a_1,a_2\in \tT$ there is $a\in \tTz$ with $a_1 + a_2 + a \in \mcA_0.$

   \item   If $\one$  fails to have a quasi-negative in a gp-pair $(\mcA,\mcA_0)$, then  $\mcA_0 = \{0\}$ and $\mcA=\tTz,$ which
is~ZSF.
\end{enumerate}\end{lemma}

\begin{proof}
(i) If $a_1 + a_2  \in \mcA_0,$ then take $a_3 = \zero$.   If $a_1 + a_2  \in \tT,$ the assertion is by weak Property~N.

 (ii) We claim that $a_1+a_2\in \tT$ for all $a_1,a_2\in \tT.$ Otherwise take $a_1+a_2\in \mcA_0.$ Then $\one + a_1^{-1}a_2\in \mcA_0,$ contrary to hypothesis.   By induction on height, every nonzero element of $\mcA$ is then in $\tT.$ But any sum of tangible elements must then be tangible, and thus is not $\zero.$
\end{proof}

 So a weakly metatangible gp-pair  failing \PropertyN\ is a very special case, whose linear algebraic theory will be classical (and thus known). Conclusion (i) of the lemma, when applied to fuzzy rings, is called the ``quasi-field condition'' in   ~\cite{GJL}.

 \begin{lemma}\label{bip9}
For any  metatangible pair $(\mathcal A, \mcA_0)$ of the second kind,
$\tTz^\odot$   of  Lemma~\ref{propN1} is
a  semiring.
 \end{lemma}
\begin{proof}  If $a_1^\circ \ne a_2^\circ$ then $a_1+a_2\in \tT$  (since otherwise $a_1+a_2\in \mcA_0$ and so $a_1+a_2=a_1^\circ=a_2^\circ$ because $(\mcA,\mcA_0)$ satisfies property N).
Moreover by \Cref{ep}, $ (a_1+a_2) ^\circ= a_1^\circ+
a_2^\circ $, so $ a_1^\circ+ a_2^\circ \in\tTz^\circ $

If $ a_1^\circ = a_2^\circ$ then  $ a_1^\circ + a_2^\circ =  a_1^\circ + a_1^\circ = {(a_1+a_1)^\circ},$ noting that $a_1+a_1 \in \tT$ since $(\mathcal A, \mcA_0)$ is of the second kind. 

Associativity is seen via
$ a_3^\circ(a_1^\circ + a_2^\circ ) = (a_3 (a_1  + a_2))^\circ =  (a_3 a_1  + a_3a_2)^\circ =  (a_3 a_1)^\circ  +(a_3  a_2)^\circ.  $
\end{proof}

\begin{theoalpha} \label{theoalphaD}
 Suppose that  $(\mathcal A, \mcA_0)$ is a metatangible pair of the second kind, and  with a  negation map  $(-)$. If
 $b_1 \, \nablaminus \, b_2 $  then $b_1\nablaT b_2$.    Conversely, if $b_1\nablaT b_2$ and $(\mathcal A, \mcA_0,(-))$ is uniquely negated, then $b_1 \, \nablaminus \, b_2 $.
\end{theoalpha}

The proof of \Cref{theoalphaD} is given in \Cref{balel}.

 \subsubsection{null-bipotent pairs} $ $

 The next   property is stronger than metatangibility.

  \begin{definition}\label{bip3}
 An \textbf{null-bipotent} pair (called ``$\mcA_0$-bipotent'' in \cite{AGR2})
 is a metatangible pair $(\mcA,\mcA_0)$,
 with $a_1+a_2 \in \{a_1,a_2\} \cup \mcA_0$ for all $a_1,a_2$
in~$\tT.$
  \end{definition}

  \begin{lemma}\label{bipid0}
  Any null-bipotent pair $(\mcA,\mcA_0)$ of the second kind is   (additively) idempotent.
  \end{lemma}
  \begin{proof} We need to show that $b+b =b,$ for $b\in \mcA.$ Writing $b = \sum a_i$ for $a_i\in \tT,$ we may assume that $b\in \tT.$ Then
 $b+b \notin \mcA_0,$ so  $b+b =b.$
  \end{proof}

  \subsubsection{Moduli}\label{modtrp0}$ $

The next concept, inspired by \cite{AGG2}, is a way of partially ordering elements of a pair.

 \begin{definition}\label{mod1}$ $
 \begin{enumerate}

  \item A \textbf{homomorphism} from a weakly \admissible\  $\tT$-module  $\mathcal A$ to a weakly \admissible\ $\tT '$-module $\mcA'$ is an additive homomorphism
 $\mu : \mathcal A \to \mathcal A' $ such that $\mu(\one_\tT) = \one_{\tT'}$, $\mu(\tT) \subset \tT'$,
 $\mu(\zero_\mcA) = \zero_\mcA'$, and $\mu (ab) = \mu(a)\mu( b)$ for all $a\in \tT,$ $b\in \mcA.$


  \item A \textbf{weak modulus} from a pair $\mathcal A$  to an idempotent $\tT '$-module $\mcM$ is a nonzero  homomorphism
 $\mu : \mathcal A \to \mcM $ satisfying the property
     $b_1 + b_2 = b_1$ whenever  $\mu(b_1)\ne \mu(b_2)$ with $\mu(b_1)+\mu(b_2)= \mu(b_1).$

     \item A weak modulus is a \textbf{modulus} when $\mcM$ is  bipotent, and thus ordered via Remark~\ref{Gre}(i), and satisfying
     $$\mu (a_i)\le \mu(a_i')\;,  \quad i = 1,2 \;\Rightarrow \mu(a_1a_2)\le \mu (a_1'a_2'),\qquad \forall a_i, a'_i\in \tT.$$
  The modulus $\mu$ is a $\tT$-\textbf{modulus} when $\mu(\tT)= \mu(\mcA\setminus\{\zero\}).$
 \end{enumerate}

   \end{definition}

     This   definition resembles \cite[Proposition~2.10]{AGG2} (which takes $\nabla$ to be trivial),  in which  $\mu$ is onto (although the latter condition can be removed by restricting $\mcM$ to the image of $\mu$).

\begin{theoalpha}\label{theoalphaB}
   Suppose that  $(\mcA,\mcA_0)$ is a   metatangible mon-pair.

\begin{enumerate}
   \item
Viewing $(\tTz^\odot,\odot)$ as semiring via Lemma~\ref{bip9},
  define $\mu(a) = a^\circ$ for all $a\in \tT.$ Then
 $\mu: \tTz \to \tTz^\odot$ is a monoid homomorphism.

\item If  $(\mcA,\mcA_0)$ is $e$-final and uniquely negated, then there is a $\tT$-modulus $\mu: \mcA \to \tTz$ given by $\mu (c) = c_\tT^\circ.$.

 \item Suppose that  $(\mcA,\mcA_0)$ is  null-bipotent.  \begin{enumerate}
     \item $ \tTz^\odot$ becomes an idempotent semiring under the multiplication of (i) and the new addition~$\dplus$ given by
 $$a_1^\circ \dplus a_2^\circ =\begin{cases}
 (a_1+a_2) ^\circ= a_1^\circ+a_2^\circ \quad{ \text if }\quad a_1^\circ \ne a_2^\circ; \\ {a_1}^\circ \quad \text{ if }\quad   a_1^\circ = a_2^\circ.
\end{cases}$$

   \item There is a $\tT$-modulus $\mu: \mcA \to \tTz^\odot$ given by $\mu (c) = c_\tT^\circ.$

 \item When $(\mcA,\mcA_0)$ is of the second kind, $\dplus$ is the same addition as the given addition in $\mcA_0,$
so the map $\mu:\mcA\to\mcA_0$ given by $\mu(b)= b^\circ$ is a $\tT$-modulus.
 \end{enumerate}
    \item We have the following results with respect to with the original operations, when $(\mcA,\mcA_0)$ is null-bipotent of the first kind.
  \begin{enumerate}
     \item  When $(\mcA,\mcA_0)$ has \period\ $m$,   $m\tTz$ is a bipotent sub-semiring of $\mcA$, and there is a $\tT$-modulus $\mu:\mcA \to  m\tTz ,$ given by $\mu (ka) = ma$  for all $k\in \Net.$

  \item  When $\mcA$ has characteristic $(1,2)$,   $\mcA$ has  the  $\tT$-modulus $\mu: \mcA \to  \mcA ^\circ \subseteq \mcA_0$ given by  $\mu (b) =  b_\tT^\circ: b\in \mcA.  $
  In other words,  $\mu (ka) = a+a= a^\circ$ for all $a\in \tT,$ $1<k\in \Net.$ (This corresponds to the ``ghost map'' in supertropical algebra \cite{IR}).

    \end{enumerate}
\end{enumerate}
\end{theoalpha}

The proof of \Cref{theoalphaB} is given in \Cref{modtrp}. When $b \mapsto b^\circ$  defines a (weak)  modulus,
we call it  the \textbf{standard (weak) modulus} on $\mcA$.

\begin{rem} \label{nont}$ $\begin{enumerate}
    \item Any semigroup $\tG$ can be viewed as
 a pre-ordered $\tT$-module  $(\tG, \leq)$ via Green's  identification of \Cref{Gre}.

  \item  If one strengthens the definition to the    property that for $b\in \mcA,$ $\mu (b) = \zero$ iff $b = \zero,$ then $\mu$ is called a {\it valuation} in the literature.

    \item    If  $\mathcal A$ has a modulus $\mu$, and $a$ is a root of $\one$, then $\mu(a) = \mu(\one).$

      \item  Any ZSF semiring without zero divisors has the  \textbf{trivial modulus}  $\mu:\mcA \to \mathbf B = \{ \zero, \one\}$ given by $\mu(\zero) = \zero $ and $\mu(b) = \one $ for all $b\ne \zero.$
\end{enumerate}
\end{rem}
\section{Structure theory of   metatangible pairs and null-bipotent  pairs -- Proofs of \Cref{theoalphaA}, \Cref{theoalphaD}, and \Cref{theoalphaB}}\label{modtrp3}

In this subsection we recast results from \cite[\S 6]{Row21} in the more general context of pairs.
We start with our key result for pairs of the second kind.

\begin{lemma}\label{m11}  Suppose that $(\mcA,\mcA_0)$  is a   metatangible pair of the second kind.
\begin{enumerate}
    \item  Either $\eplus = e =e +\one^\dag$, or $\eplus\in \tT$.
          \item   $(\mcA,\mcA_0)$  is $e$-idempotent.
          \item  $(ae)e = ae$.
\end{enumerate}
        \end{lemma}
 \begin{proof} By hypothesis, $a: = \one  +\one \in \tT,$ and $\one ^\dag +\one ^\dag\in \tT.$ Also, $e+e\in \mcA_0$.

 (i)    $\eplus =e+\one  = \one ^\dag +a \in \mcA_0 \cup \tT.$
Assume that $\eplus\notin \tT.$ Then $a+\one^\dag = \eplus \in \mcA_0$, implying by \PropertyN\ that
$a+\one^\dag = (\one^\dag )^\circ = \one +\one^\dag =e ,$
and  hence $\eplus =a+\one^\dag  = e. $

(ii) Continuing the proof of (i), if $\eplus\notin \tT$ then $ \one +(\one^\dag+\one^\dag) =   e  +\one^\dag  = (a +\one^\dag)+\one^\dag = e+e.$ Hence by \PropertyN, $e+e = \one^\circ = e.$

Thus, by (i), we may assume that  $\eplus \in \tT.$ But then
$\eplus+\one ^\dag  =e+e\in \mcA_0,$ so $e+e = \one^\circ =e.$


(iii) $(ae)e = (ae)\one^\circ =  ae +ae \one^\dag = ae+ae = a(e+e) = ae.$
 \end{proof}

Let us elaborate.

\begin{lemma}\label{ws1}
    If $(\mcA,\mcA_0)$ satisfies \PropertyN\ and $a_1 \preceqzero a_2$ for $a_i \in \tT$, then $a_1 e= a_2 e.$
\end{lemma}
\begin{proof}
    Write $a_2 = a_1+b$ for $b\in \mcA_0.$  Then $a_1^\dag +a_2  = a_1e +b \in \mcA_0, $ so $a_2 e = a_1^\dag e = a_1e$.
\end{proof}

  \begin{lemma}\label{m1}  Suppose that $(\mcA,\mcA_0)$  is a   metatangible pair. 
  Let $a =a_1 +a_2 ^\circ$, where $a_i\in\tT$. One of the following hold:
  \begin{enumerate}
     \item  $a  \in \tT$, with $a^\circ = a_1^\circ.$
       \item $a  = a_2 ^\circ,$ with $a_1+a_2= a_2$.
     \item $a_1 ^\circ =   a_2 ^\circ,$ with $a  = a_2 e^+$.
       \end{enumerate}
  \end{lemma}
  \begin{proof}
If $a \in \tT,$ then $a +a_1^\dag\in \mcA_0,$ implying $a^\circ =  a_1^\circ$ by \PropertyN.
So assume that $a $ is not tangible. Then  we have one of the following, rewriting $a= (a_1+a_2)+a_2^\dag $:
\begin{itemize}
\item $a_1+a_2 \in \tT$, so $a\in \mcA_0$ and
  $a = (a_1+a_2) + a_2^\dag =(a_2^\dag)^\circ = a_2^\circ $.
    \item $a_1+a_2 = a_1 ^\circ =   a_2 ^\circ. $
    But then  $a = a_2^\dag +a_2^\circ  = a_2(\one+e)$.
\end{itemize}
  \end{proof}

 \begin{lemma}\label{bip1}  Suppose that $(\mcA,\mcA_0)$  is a   metatangible pair.  Let $a =a_ 1+a_2+a_3$, with  $a_i\in \tT$. One of the following holds:
 \begin{enumerate} \item $a $ is tangible, \item $a =a_1^\circ = a_2^\circ = a_3^\circ ,$ \item $a = e^+a_1 = e^+a_2 = e^+a_3 .$ \end{enumerate}
\end{lemma}
\begin{proof} We use repeatedly the fact that if $a_i + a_j \in \mcA_0 $ then $a_i + a_j = a_i^\circ = a_j^\circ .$

 First assume that $a_1+a_2\in \mcA_0.$
Then $a = a_1^\circ + a_3 = a_2^\circ + a_3$, so we are done by  \lemref{m1}.

So we may assume that $a_i+a_j\in \tT $ for all $i\ne j.$
If (i) fails then  $a_1+a_2+a_3 = (a_1+a_2)+a_3= a_3^\circ.$ By symmetry of indices, we have (ii).
\end{proof}

\begin{lemma} \label{theoalphaA0} Suppose $(\mathcal A, \mcA_0)$ is a metatangible pair, with $a_1,a_2\in \tT$.\begin{enumerate}
    \item  If $(\mathcal A, \mcA_0)$ is  of the first kind,
then for $m_1, m_2 \ge 1,$ $$m_1 a_1 + m_2 a_2 = \begin{cases}
m_3a_3, \text{ when } a_3:=a_1+a_2  \in \tT,\text{ with } m_3 \le \min\{m_1,m_2\}  ,\\ ( m_1 + m_2) a_1 = ( m_1 + m_2)a_2, \text{ when }  a_1+a_2\in \mcA_0.
\end{cases}$$

  \item
If $(\mathcal A, \mcA_0)$ is  of the second kind then $$a_1^\circ +a_2^\circ = \begin{cases} a_3^\circ, \text{ when }  a_3= a_1+a_2\in \tT,
\\  a_1^\circ, \text{ when }  a_1+a_2= a_1^\circ = a_2^\circ.
         \end{cases}$$
\end{enumerate}
\end{lemma}
\begin{proof}
 (i) Assume that $m_1\le m_2.$ Let $a = m_1 a_1 + m_2 a_2 = m_1(a_1+a_2) + (m_2-m_1)a_2 .$
 If $a_3 = a_1+a_2\in \tT$ then $a= m_a a_3 + (m_2-m_1)a_2,$ so  we are done by induction on $m_2.$ Otherwise $a_1 + a_2 = a_1^\circ = a_1 +a_1 = a_2 + a_2,$ so we can replace $a_2$ by $a_1,$ and visa versa.

 (ii)  $(\mathcal A, \mcA_0)$ is $e$-idempotent by \lemref{m11}(ii), so we apply \lemref{m1}(ii).
\end{proof}

\subsection{The uniform presentation (Proof of \Cref{theoalphaA})}\label{up}$ $

Let us proceed with details of the proofs of the first three main theorems.
The uniform presentation of \cite[Theorem 6.25]{Row21} carries over for metatangible pairs.

\begin{proof}[Proof of \Cref{theoalphaA}]
 For  $(\mathcal A, \mcA_0)$   of the first kind, write $c = \sum _{i=1}^{m_c} a_i,$ for $ a_i,\in \tT$. The proof is by induction on $m_c.$ If $m_c = 1$, the assertion is obvious. For $m_c \ge 2$ we have $c = (a_1+a_2) + b$ where $b= \sum _{i=3}^{m_c} a_i.$
 If $a_1+a_2\in \tT$ then $c$ has smaller height, a contradiction, so
  $a_1+a_2\in \mcA_0,$ i.e., $a_1+a_2= a_1^\circ .$ Now apply induction to $b,$ and conclude by adding it to $a_1^\circ$ using  \lemref{theoalphaA0}.


 For  $(\mathcal A, \mcA_0)$  of the second kind, we  induct on the height, noting that we are done by Lemma~\ref{bip1} unless  $m_c \ge 3$ and $a_1+ a_2+a_3 = e^+a_1$; then $a_1+a_1\in \tT$ enables us to reduce the height. The last assertion is clear.
 \end{proof}

In contrast to \cite[Theorems 7.28, 7.32]{Row21}, the element  $c_\tT$ need not be uniquely defined, although we do have the next result.

 \begin{lemma}\label{bip11}$ $
 \begin{enumerate}
     \item  Suppose that  $(\mathcal A, \mcA_0)$  is a   null-bipotent  pair of the first kind, and $c\ne \zero$. Either
     \begin{enumerate}
         \item  $c_\tT^\circ $ of \Cref{theoalphaA} is   uniquely defined, or         \item  $ \mathcal A $  has some \period\ $m>0$ and
$mc_\tT$ is   uniquely defined, and  $m\tT$ is an idempotent semiring.
     \end{enumerate}

  \item Suppose $(\mathcal A, \mcA_0)$ is a metatangible  pair  of the second kind.
  \begin{enumerate}
      \item   $c_\tT^\circ $ is   uniquely defined.   \item  If $c_\tT^\circ +c_\tT \in \mcA_0$, then $c_\tT^\circ +c_\tT = c_\tT^\circ .$
  \end{enumerate}
 \end{enumerate}
 \end{lemma}
\begin{proof} (i) Suppose   that  $c_\tT^\circ $ is not  uniquely defined. Then $c = mc_\tT = m'c'_\tT  $, with $c_\tT^\circ \ne {c'_\tT}^\circ ,$ and thus $c_\tT + c'_\tT \notin \mcA_0,$
so we may assume that  $c_\tT + c'_\tT =c_\tT .$ Then by iteration $c + c'_\tT =  mc_\tT+ c'_\tT = c,$ and then $2m c_\tT  =  m c_\tT+ m'c'_\tT =
 c =m c_\tT ,$ implying $2m\one = m\one,$ so $(\mathcal A, \mcA_0)$  has  \period\ $m.$ But then $m' c'_\tT = 2m' c'_\tT = c+c= 2m c_\tT =m c_\tT , $ and for any $a\in \tT,$ $ma+ma = 2ma = (2m)\one a = m\one a =ma.$

 (ii) (a) is obvious by definition, applied to \Tref{theoalphaA}(ii)(c).

 (b) $c_\tT^\dag+c_\tT^\dag \in \tT,$ so $c_\tT^\circ +c_\tT^\dag = c_\tT+ (c_\tT^\dag+c_\tT^\dag) \in \mcA_0$ implies $c_\tT^\circ +c_\tT^\dag =
 c_\tT^\circ$.

 \end{proof}

\begin{corollary}
    \label{cia} In the notation of  Theorem~\ref{theoalphaA}, for $a\in \tT$ and $c\in \mcA,$ either $a+c \in \tT $ or $a+c = c_\tT^\circ $ or $a+c = c+c_\tT^\dag. $
\end{corollary}
\begin{proof} For $m_c=1$ the assertion is obvious, so we proceed by induction on $m_c.$
    $a+c = (a+c_\tT)+(m_c-1)c_\tT.$ If $a+c_\tT \in \tT,$ then, by induction on $\tT,$ letting $a' = a+c_\tT,$ either $a' +(m_c-1)c_\tT \in \tT,$ in which case $a+c = a' +(m_c-1)c_\tT \in \tT$ and we are done, or
    $a'+(m_c-1)c_\tT = c_\tT^\circ ,$ and we are done, or $a'+(m_c-1) c_\tT = (m_c-1) c_\tT + c_\tT^\dag,$   so $a + m_c c_\tT = a + (m_c -1)c_\tT + c_\tT = m c_\tT + c_\tT^\dag.$

    On the other hand, if $a+c_\tT \in \mcA_0,$ then $a+c_\tT = c_\tT^\circ,$
and $a+c = a +c_\tT +(m_c-1)c_\tT = c_\tT^\circ  +(m_c-1)c_\tT  = c_\tT^\dag +c.$ \end{proof}

  \begin{corollary}\label{ci}
    In the notation of  Theorem~\ref{theoalphaA}, if  $a+c \in \mcA_0$ for $a\in \tT,$ then
   either $a+c = c_\tT^\circ$, or $(\mcA,\mcA_0)$ is of the first  kind with $a+c =(m_c+1)c_\tT$ for pairs.
  \end{corollary}
  \begin{proof}  For pairs of the second kind, we are done unless  $c = c_\tT^\circ = c_\tT + c_\tT^\dag.$ If $a+ c_\tT^\dag \in \mcA_0$, then $a+ c_\tT^\dag = c_\tT^\circ,$ by Property N,
    so $a+c = c_\tT+  c_\tT^\circ = c$ by \Cref{bip11}(ii). So assume $a+ c_\tT^\dag \in \tT.$ Then $(a+ c_\tT^\dag) +c_\tT = c_\tT^\circ$ by Property N.

The rest is by \Cref{cia}.\end{proof}

\begin{lemma}\label{parti}
    In any  metatangible pair $(\mcA,\mcA_0),$ if $b = m a$ and $c = m_c c_\tT$ with $b+c\in \mcA_0,$ then one of the following holds:
    \begin{enumerate}
        \item $m  =1,$  and either $b+c = {c_\tT}^\circ= a^\circ$ with $(\mcA,\mcA_0)$ of the second kind, or $b+c = (m_c+1)c_\tT$ with $(\mcA,\mcA_0)$ of the first kind.

              \item $b = a^\circ$, and either $b+c =b$, or $b+c = (m_c+2)c_\tT = c+c_\tT^\circ,$ with $(\mcA,\mcA_0)$ of the first kind.

              \item $b = m_b a,$ and $b+c \in \{ b, c+ 2k c_\tT, s\in \Net\}.  $

    \end{enumerate}
\end{lemma}
\begin{proof} (i) The case $m=1$ is \Cref{ci}.

(ii) Next suppose $b = a^\circ.$ First assume $a+c\in \mcA_0.$  By (i), we have several possibilities.
\begin{itemize}
    \item $a+c = {c_\tT}^\circ$ with $(\mcA,\mcA_0)$ of the second kind. If  $a^\dag + {c_\tT}^\dag \in \mcA_0 $, then $a^\dag + {c_\tT}^\dag = (a^\dag)^\circ = a^\circ,$ so $a^\circ +c = a + a^\circ.$ But $a+a \notin A_0$ implies $a+a\in \tT,$ so $b+c = a^\circ +c = (a+a) +a^\dag = a^\circ=b.$
      \item If $a+c =(m_c+1)c_\tT = c+c_\tT,$ for $(\mcA,\mcA_0)$ of the first  kind, then $b +c = a+(a+c) = a +(c+c_\tT) = c+c_\tT + c_\tT = (m_c+2)c_\tT.$
\end{itemize}

So we may assume that $a+c \in \tT.$ Then
$b+c = a  +(a+c)\in \mcA_0,$ implying $b+c= a_\tT = b.$

(iii)  $(\mcA,\mcA_0)$ is of the first kind, and $b = m a$ for $m\ge 3.$ Take $m'$ maximal such that $a+m'c_\tT\in \tT.$  If $m' = m_c$ then $b+c =(m-1)a +(a+m'c_\tT)$, and reversing the roles of $b$ and $c$ and applying (i), we have $b+c = ((m-1)+1) a =b.$ So we may assume that $m'<m.$ Then $a+(m'+1)c_\tT \in \mcA_0$, so $a+(m'+1)c_\tT = a^\circ = a+a.$ Hence $b +c = (m-1)a +a +c = (m-1)a +a+a +(m-(m'+1))c. $
    Then $a+m'c_\tT +c_\tT \in \mcA_0,$ so by Property N, $a+m'c_\tT +c_\tT = a^\circ = 2a.$
     \end{proof}

 In the presence of a negation map, we   improve \cite[Theorem~7.28]{Row21}.

\begin{lemma}\label{met1} Suppose $(\mathcal A, \mcA_0,(-))$ is uniquely negated metatangible with all $a_i \in \tT$. \begin{enumerate}
    \item
     If $a_1+a_i \in \mcA_0$ for $i = 2,3,$ and  $a_1+a_2+a_3\in \mcA_0$ then $a_3 = a_2 = (-)a_1$, and  $a_1+a_1 \in \{a_1, a_1^\circ\}.$

     In particular,  $\mcA$ is either of the first kind or is idempotent.

       \item Suppose $b:=\sum _{i=1}^t a_i \in \mcA_0$.\begin{enumerate}
\item  If no subsum of length $t-1$ is tangible, then there is a subsum $a=\sum_{j\in J} a_j\in \tT$ with $J\subset \{1,\ldots, t\}$, such that  $|J|=k\leq t-2$, $a+a\in \{a, a^\circ\}$, $a_i = (-)a$ for all $i\notin J$. 
           Furthermore $b= a^\circ$ when $a+a=a$, and $\sum _{i=1}^t a_i=  (t+1-k)a$  when $a+a\in a^\circ.$

              \item If $\sum _{i=1}^{t-1} a_i $ is tangible, then $\sum _{i=1}^{t-1} a_i = (-)a_t$.




\end{enumerate}
\end{enumerate}
\end{lemma}
\begin{proof}
(i)    $a_i = (-)a_1 $ for $i=2,3,$ by unique negation.
Furthermore either $a_1+a_1 = (-)(a_2+a_3) \in \mcA_0,$ proving $\mcA$ is of the first kind, or $a_2+a_3 \in\tT$, then $a_2+a_3 = (-)a_1 $ by unique negation, so $(-)a_1 (-)a_1 = (-)a_1$, implying $a_1+a_1 =a_1.$

(ii)(a)  Take a  tangible subsum of a maximal length, i.e., maximal number of $a_i$, which we may assume is $a=\sum _{i=1}^k a_i,$ for $k\le t-2.$ Then $a + a_j \in \mcA_0$ for each $j>k,$ implying $a_j = (-)a.$ If $a+a \in \mcA_0$, then $a+a =a^\circ$, otherwise, $a+a \in \tT$, and since $\sum _{i=1}^{k+2} a_i\notin \tT$,  (i) shows $a+a \in \{a,a^\circ\},$ and indeed $a+a=a$.
If $a+a=a$ then the sum is $a(-)(t-k)a = a(-)a =a^\circ.$ If $a+a =a^\circ,$ then $a = (-)a,$ so the sum is $(t+1-k)a.$

(b) By unique negation.
\end{proof}

\subsection{Results about balancing (Proof of \Cref{theoalphaD})}\label{balel}$ $

We shall prove more general results, which yield  \Cref{theoalphaD} as a consequence.
\begin{lemma}\label{pl} Suppose that $(\mcA,\mcA_0)$ has weak Property~N and $\mcA=\mcA_0+\mcA^\natural$.
\begin{enumerate}
    \item
 $a \nablaT (a+b)$ for any $a\in \tT$ and $b\in \mcA_0.$

 \item   
 $b \nablaT (b+b')$ for any $b\in \mcA$ and $b'\in \mcA_0.$
\end{enumerate}
\end{lemma}
 \begin{proof}
We already checked that $\nablaT$ is a balance relation in \Cref{bal1}.
 Since $\nablaT$ is preserved by addition and reflexive, we only need to check that $b\nablaT \zero$, which follows from the definition by taking $t=0$.
%
 \end{proof}

  $\nablaT$   acts as a congruence on $\tT$ in the following situation.

\begin{definition}\label{Ntran}
%
A pair $(\mathcal A,\mathcal A_0)$  is \textbf{N-transitive} if
$a_i + a_{i+1}\in \mcA_0$ for $a_i\in \tT,$ $1\le i\le 3$, implies $a_1+a_4 \in \mcA_0.$
\end{definition}

 \begin{example}$ $ \begin{enumerate}
     \item   The pairs of Example~\ref{tr0}(iii)and Example~\ref{tr0}(iv)(a) are N-transitive.

      \item    The pair of Example~\ref{tr0}(iv)(b) has weak Property~N but is not N-transitive, by construction. To obtain a counterexample satisfying \PropertyN, one could mod out the congruence generated by $\{ (\la _i + \la_{i+1},  \la_{i+1} +  \la_{i+2}),\ i = 1,2 \}.$
 \end{enumerate}
 \end{example}

\begin{lemma}
  Any uniquely negated pair $(\mathcal
  A, \mcA_0)$  is  N-transitive (as well as satisfying \PropertyN).
\end{lemma}
\begin{proof}
     $a_i + a_{i+1}\in \mcA_0,$ so $a_1 = (-)a_2 = a_3 = (-)a_4,$ implying
 $a_1 + a_{4}\in \mcA_0.$
\end{proof}
\begin{lemma}\label{bal4} Suppose that $(\mathcal A, \mcA_0)$ has weak property N.

    \begin{enumerate}
        \item $a_1 \nabladag a_2$ implies $a_1 \nablaT a_2$, for $ a_1,a_2 \in \tT$.
         \item When $(\mcA,\mcA_0)$ is N-transitive, $b_1 \nablaT b_2$ implies $b_1 \nabladag b_2$, for $b_i\in \mcA$.
    \end{enumerate}
\end{lemma}
\begin{proof}
    (i) Suppose $a_1 +a_2^\dag \in \mcA_0.$ By definition $a_2 +a_2^\dag \in \mcA_0.$ Hence, taking $a= a_2^\dag$ we have  $a_1 \nablaT a_2$.

(ii) Write $b_i=b_{i,0}+\sum_{j=1}^t a_{i,j}$ with $b_{i,0}\in \mcA_0$ and
$a_{i,j}\in \tT$ such that $a_{i,j} +a_j\in \mcA_0$ for $a_j\in \tT,$ $i=1,2.$  Then for all $j\geq 1$, $a_{1,j}+ a_j\in \mcA_0$,   $a_j+ a_{2,j}\in \mcA_0$ and obviously $a_{2,j}+ a_{2,j}^\dag\in \mcA_0$ so, by N-transitivity, $a_{1,j} +a_{2,j}^\dag\in \mcA_0$. We deduce that $b_1 +b_2^\dag\in \mcA_0$. By symmetry, we obtain $b_1 \nabladag b_2$.
\end{proof}

\begin{proof}[Proof of \Cref{theoalphaD}]
Assume $(\mcA,\mcA_0)$ is metatangible with a negation map of the second kind.
Since $\mcA$ has a negation map, we write $\nablaminus$ for $\nabladag$.

Let $b_1,b_2\in \mcA$.
Since $(\mcA, \mcA_0)$ is uniquely negated, it satisfies \PropertyN\ together with N-transitivity;
then, by \Cref{bal4} (ii), we have that $b_1\nablaT b_2$ implies $b_1 \, \nablaminus \, b_2 $.

Conversely, assume $b_1 \, \nablaminus \, b_2 $.
Then $b_1  (-) b_2 \in \mcA_0$.
Since $\mcA$ is metatangible of the second kind, we have $\mcA = \tT \cup \mcA_0$.
If $b_1,b_2\in \mcA_0$, then $b_1\nablaT b_2$ by definition of $\nablaT$.
If $b_1,b_2\in \tT$, taking $a = (-)b_1,$ we have  $a+b_1 \in \mcA_0$ and $a+b_2 = (-)(b_1(-)b_2)\in \mcA_0.$
So $b_1\nablaT b_2$.

If $b_1\in \tT$ and $b_2\in \mcA_0$, then write $b_2=a^\circ$, by \Cref{theoalphaA}. By \Cref{ci} (replacing $b_1$ by $(-)a$ and $c$  by $a^\circ$), $a^\circ = (a^\circ + b_1) \ \nablaT\ b_1$ by \Cref{pl}.

The case $b_2\in \tT$ and $b_1\in \mcA_0$ is symmetrical.
\end{proof}

 \begin{lemma}\label{balNpropN}   Suppose that $(\mathcal A, \mcA_0)$ is  N-transitive and has property N.
 Then for any choice of $\one^\dag$, $\nabladag$~restricted to
$\tT$ 
 is transitive, and thus is an equivalence relation, and it is independent of the choice of~$\one^\dag.$
 \end{lemma}
 \begin{proof} 
 Assume $a_1,a_2,a_3\in \tT$ are such that $a_1\nabladag a_2$ and $a_2\nabladag a_3$.
Then $a_1+a_1\one^\dag\in \mcA_0$ (for instance by \Cref{ep}). 
Moreover,
 $a_1\one^\dag+a_2\in \mcA_0$, and $a_2+a_3\one^\dag\in \mcA_0$, so  $a_1+a_3\one^\dag\in \mcA_0$ by~N-transitivity.
 Similarly $a_3+ a_1\one^\dag\in \mcA_0$. 
 Hence $a_1\nabladag a_3$.

 Furthermore, let $a_1,a_2\in \tT$ and $\one^{\dag}$ and $\one^{\dag'}$ be two
quasi-negatives of $\one$.
If $a_1+a_2\one^\dag \in \mcA_0$, noting $ \one + \one^{\dag'}\in \mcA_0,$ and
 $\one^\dag+ \one\in \mcA_0,$ we have $ a_2 +a_2\one^{\dag'}\in \mcA_0,$  and $a_2 \one^\dag + a_2\in \mcA_0,$ so $a_1 + a_2\one^{\dag'}\in \mcA_0$ by~N-transitivity.
Exchanging $a_1$ and $a_2$, and then $\one^{\dag}$ and $\one^{\dag'}$,
we obtain that $a_1\nabladag a_2$ is equivalent to $a_1\nabla_{\dag'} a_2$ for
$a_1,a_2\in \tT$.
 \end{proof}

 \begin{corollary}
    If $(\mcA,\mcA_0)$ has a negation map $(-),$ then $(\mcA,\mcA_0)$ is N-transitive if and only if $\nablaminus$ is transitive.
\end{corollary}




   \subsection{Uniquely negated metatangible pairs}$ $

 \begin{rem}   Suppose that  a pair $(\mathcal
A, \mcA_0)$   satisfies \PropertyN.
     If $a_1+a_2 =a_1^\circ  $, then clearly $a_1^\circ =a_2^\circ $, leading us to ask when the converse holds. (This need not be the case, cf.~\Eref{metex}(ii).) In this subsection, we show that the converse holds in any uniquely negated  metatangible pair with $e+\one \in \{\one, e\}$.
 \end{rem}

\begin{definition}
  Suppose $(\mathcal
A, \mcA_0,(-))$ is a  uniquely negated  metatangible triple. $(\mathcal
A, \mcA_0,(-))$  has \textbf{tropical type } if $(\mathcal
A, \mcA_0,(-))$ is $e$-final;  $(\mathcal
A, \mcA_0,(-))$ has \textbf{pseudo-classical type} if $e^+ = \one. $
\end{definition}

 Any semiring triple of pseudo-classical type is ``dominated'' by a ring, in the following sense.

\begin{theoalpha}\label{theoalphaV} Suppose the semiring triple $(\mathcal
A, \mcA_0,(-))$ has pseudo-classical type. Let $\tT' = \tT\cup \{e\},$ $\tT_1 = \{ a\in \tT': a+ e = a\},$ and $\tT_2 = \{a\in \tT' : a+\one = \one\}.$
Then $\tT_1\cap \tT_2 = \{e\}.$ $\tT_1$ is a ring whose zero element is~$e$, and $\tT_1$ additively absorbs $\tT_2$ in the sense that $a_1+a_2 = a_1$ for all $a_i\in \tT_i$. $\tT_2$ is an antiring  whose zero element is~$\zero$, and whose   element $e$ is~additively absorbing.
\end{theoalpha}
\begin{proof}  $\one + e = \one$ is given, and $(-)\one + e = (-)(\one + e) = (-)\one.$ Hence $e(-)e = e.$ Clearly $e\in \tT_1\cap \tT_2.$ If $a\in \tT_1\cap \tT_2$ then $e = e+a = a.$

Also $e^2 = (\one(-)\one))e = e(-)e = e$.

If $a_1+e \in \tT$ with $a_1 \ne (\pm)\one,$ then $(a_1+e)(-)a_1 = a_1^\circ +e \in \mcA_0, $ so $a_1+e =a_1$ by unique negation, and $a\in \tT_1$.

If $a_2+e \in \mcA_0$ then $a_2+\one = \one$ by unique negation, so $a_2\in \tT_2$.

If also $a_1'\in \tT_1,$ then $a_1+a_1' +e= a_1 +a_1',$ and
 $a_1a_1' +e = a_1(a_1'+e)+e = a_1a_1' +(a_1+e)e =  a_1a_1' +a_1e = a_1 (a_1'+e) = a_1a_1' , $ so $a_1+ a_1', a_1a_1' \in \tT_1.$ Furthermore, $(-)$

Next, $a_1+a_2 = (a_1+e) +a_2 = a_1+ (e+a_2) = a_1+e = a_1.$

If also $a_2'\in \tT_2,$ then $a_2+a_2' +\one = a_2 +\one =  \one,$
and
  $a_2a_2' +\one = a_2a_2'+ (a_2'+\one) =  (a_2+\one) a_2'  +\one = a_2'+\one = \one,$ so $a_2+a_2', a_2a_2'\in \tT_2'.$
\end{proof}

\subsubsection{Tropical type}$ $

 In tropical mathematics, the main interest lies in tropical type, which we shall show is quite restrictive.

   \begin{lemma}
       If  $(\mathcal
A, \mcA_0,(-))$ is of tropical type, then $\one + \one \in \{ \one, e\}$.
   \end{lemma}
   \begin{proof}
      Let $ a=\one + \one.$ If $a\ne e$ then $a\in \tT,$ and $a(-)\one = \one + e = e,$ so $a=\one$ by unique negation. (This also follows from the uniform presentation.)
   \end{proof}

\begin{theoalpha}\label{theoalphaW} Suppose the triple $(\mathcal
A, \mcA_0,(-))$ has tropical type.
\begin{enumerate}
\item If $a_1^\circ = a_2^\circ$ for $a_i\in \tT,$ then  $a_1 = (\pm)a_2$.
    \item $\tT^\circ$ has a $\tT$-module order given by $a_1^\circ \le a_2^\circ$ if and only if $a_1^\circ + a_2^\circ =a_2^\circ.$

  \item  $(\mathcal
A, \mcA_0,(-))$ is null-bipotent, $\mcA = \tTz \cup \tTz^\circ$ and $\mcA_0=\tTz^\circ$.

\item Suppose $b:= \sum_{i=1}^{t} a_i\in \mcA_0.$ Then \begin{enumerate}
    \item $b= b(\pm)a_i$ for each $1\le i \le t.$

    \item There are $i_1,i_2$ such that $a_{i_1} = (-)a_{i_2}$, and $b=a_{i_1} +a_{i_2} =a_{i_1}^\circ =a_{i_2}^\circ$.

      \item $b = a_i^\circ$, for all $a_i\in \tT$ such that $a_i^\circ$ is maximal under the order of (ii).
\end{enumerate}
\end{enumerate}
\end{theoalpha}
 \begin{proof}
 (i) $(a_1+a_2)(-)a_1 = a_1^\circ +a_2 = a_2e^+ = a_2e.$ If $a_1+a_2\in \mcA_0$ then $a_1 =(-)a_2.$ Otherwise $a_1+a_2\in \tT,$ so $a_1+a_2 = a_1$ by unique negation. But symmetrically $a_1+a_2=a_2,$ so $a_1=a_2.$

 (ii) Suppose $a_1^\circ \ne a_2^\circ.$ Then $a_1 + a_2\in \tT.$ If $a_1^\circ + a_2 \in \tT,$ then $a_1^\circ + a_2 =a_2$ by unique negation, implying $a_1^\circ + a_2^\circ = a_2^\circ$. Otherwise $a_1^\circ + a_2 \in \mcA_0$ by Lemma~\ref{m1},   implying $a_1^\circ + a_2 =a_1^\circ.$ Likewise $a_1^\circ (-)a_2 =a_1^\circ,$ so $$a_1^\circ + a_2^\circ = (a_1^\circ + a_2)(-)a_2 = a_1^\circ (-) a_2 = a_1^\circ.$$

 (iii) Suppose $a,b\in \tT$, with $b\ne (-)a.$ Then $c = a+b\in \tT$. By (ii) we may assume that $b^\circ =a^\circ + b^\circ = c^\circ $.  Thus $b = (\pm) c$ by (i). We are done unless $b = (-)c.$ Thus
 $b+b = (-)a + b^\circ.$ If $b+b \in \tT$ then $b+b=(-)a,$ since $b+b = a^\circ + b^\circ\in \mcA_0$ and by unique negation, so $a^\circ = b^\circ +b^\circ = b^\circ,$ implying $a=b.$
Thus $a+a = c = (-)b = (-)a,$ implying $\one +\one = (-)\one,$ so $\one+\one +\one = \one^\circ =e.$ But now $\one =((-)\one)^2=   (\one +\one)(\one +\one) = \one + \one +\one +\one =e+\one = e,$ a contradiction.

We prove by that any element is in $\tTz \cup \tTz^\circ$. Indeed we need to show that $\tTz \cup \tTz^\circ$ is closed under addition.
$a^\circ (\pm)a = a^\circ$ by $e$-finality. If $a_1, a_2\in \tT$ with $a_1\ne (-)\a_2$, we may assume by null-bipotence that $a_1+a_2 = a_1$. Then: $a_1^\circ + a_2 = a_1^\circ$ and $a_1 + a_2^\circ = a_1(-)a_2 \in \tT \cup \tT^\circ.$

   (iv)(a) Since $b\in \mcA_0$, we have $b\in \tT^\circ$ and so $b=(-)b$.
   We shall prove that $b+a_t=b,$ which implies $b(-)a_t = (-)(b+a_t) =(-)b = b.$ Let $b' = \sum_{i=1}^{t-1} a_i.$
If $b' \in \tT,$ then by (b), $b'= (-)a_t,$ so $b+a_t = a_t^\circ + a_t = a^\circ = b.$ If $b'\in \tT^\circ$ then, by Lemma~\ref{m1},  $b'+a_t = a_t^\circ$ (since $b'+a_t \notin \tT$), so again $b +a_t = b$.

 \indent\indent(b) The assertion is obvious unless $t\ge 2.$ By induction we may assume that every subsum of $\sum _{i=1}^t a_i$ is in $\tT.$ By (a), $\sum _{i=1}^t a_i =a_t^\circ,$ and so
$\sum _{i=1}^t a_i= \sum _{i=1}^t a_i +a_t = (a_1+a_t) +\sum_{i=2}^t a_i,$ so by unique negation, $a_1= a_1+a_t,$ and $\sum _{i=1}^t a_i = \sum _{i=1}^{t-1} a_i,$ and we conclude by induction.

\indent\indent(c) By (b), we may assume that $b= a_1^\circ,$ and then $a_i^\circ = a_i^\circ +a_1^\circ =b+a_1^\circ= (b+a_i)(-)a_i = b.$
 \end{proof}

\subsection{Interaction of moduli and pairs (Proof of \Cref{theoalphaB})}\label{modtrp}$ $

Let us see how moduli,   \Dref{mod1}, enter the theory of pairs.

\begin{definition}
     The \textbf{natural pre-order} $\le$ on a subset $S$ of a $\tT$-module is
given by $b_1\le b_2 $ if $b_1+b = b_2$ for some $b\in S.$
\end{definition}

 \begin{lemma}\label{bip8a}(Inspired by \lemref{Gre}) $ $
\begin{enumerate}
    \item  Any $\tT$-module has the {natural pre-order}, which respects multiplication by elements of $\tT$.

\item When $\mcA$ is uniquely negated, the natural preorder on $\mcA$ restricts to a natural preorder on $\tT$.

 \item  When $(\mathcal A, \mcA_0)$ is null-bipotent,
 \begin{enumerate}
  \item $a_1+a_2 = a_2$ implies $a_1^\circ + a_2^\circ =a_2^\circ$.

        \item  There is an  order on $\tT^\circ$ given by $a_1^\circ < a_2^\circ$ for  $a_1^\circ \ne a_2^\circ$,   iff $a_1^\circ + a_2^\circ =a_2^\circ$ (Green's identification).

        \item  $a_1+a = a_2 $ for $a_i,a\in \tT$ iff $a_1+a_2 = a_2$. Thus, when $\mcA$ also is uniquely negated, the natural preorder on $\mcA$ is compatible with Green's identification.
 \end{enumerate}
\end{enumerate}
\end{lemma}
\begin{proof}
(i) The  assertion is clear.

(ii)   Suppose $a_1+b = a_2$ for $a_i\in \tT$. Write $b= a + b'$ for $a\in \tT,$ $b'\in \mcA_0.$ If $a_1 +a \in \mcA_0$ then $a_2\in \tT \cup \mcA_0,$ a contradiction. Hence $a_1+a\in \tT,$ and $a_1+a (-) a_2\in \mcA_0,$ implying $a_1+a = a_2.$

(iii)(a)  If $a_1^\circ \ne a_2^\circ$ then $a_1+a_2\in \tT$ by the contrapositive of \PropertyN, implying $a_1+a_2\in \{a_1,a_2\}.$ Likewise $a_1^\dag +a_2^\dag \in \tT,$
and $ (a_1+a_2) + (a_1^\dag +a_2^\dag ) =a_1^\circ+a_2^\circ \in \mcA_0$ implies also
$(a_1+a_2) + (a_1^\dag +a_2^\dag ) =(a_1+a_2)^\circ.$

(iii)(b)  By $\mcA_0$-bipotence and (a).

(iii)(c)
 Suppose $a_1+a = a_2$. We know $a_1+a \in \{ a_1, a,  a_1^\circ\}.$
If $a_1+a = a_1^\circ$ then
$a_2 =  a_1^\circ \in \tT \cap \mcA_0,$ a contradiction. If $a_1+a = a$ then $a_2 =a,  $ so  $a_1+a_2 = a_2$. If $a_1+a = a_1$ then $a_2=a_1.$
The last assertion follows from (ii).  \end{proof}

\begin{proof}[Proof of \Cref{theoalphaB}]
(i) $\nu (a_1 a_2) = (a_1 a_2)^\circ,$  which is the product of $a_1^\circ$ and $a_2^\circ$ in $\tTz^\odot$.

(ii)(a),(b) We have (b) by Lemma~\ref{bip9}, and thus (a) for the second kind. In (a) for the first kind, one easily checks associativity since $a_1 + a_2 + a_3 \in \{a_1,a_2,a_3\}$ unless two of them are equal, in which case we use Lemma~\ref{bip8a}(iii)(a).

(ii)(c) $\mu (c_1) +\mu(c_2) = \mu(c_1+c_2)$ unless ${c_1}_\tT ={c_2}_\tT,$ in which case $\mu (c_1) +\mu(c_2) =  {{c_1}_\tT}^\circ = \mu (c_1 +c_2)$, in view of Lemma~\ref{bip8a}(iii)(a).

(iii)(a) follows by the definition of \period.
Take an element $   na\in \mcA_0$ for some $n\in \Net$ and $a\in \tT$. By~ Assumption~\ref{Note1}, $\mathbf{n} \in \mcA_0.$ Hence $\mathbf m = \mathbf{mn} \in \mcA_0, $ and clearly $mb  = mb + mb$ for all $b\in \mcA.$

 (iii)(b) The amount of uniqueness satisfied by the uniform presentation (Lemma~\ref{bip11}) shows that $\mu$ is well-defined.

\end{proof}



\section{Examples of pairs}\label{pairex}

Let us present some of the main examples.

\subsection{Metatangible pairs which lack unique negation}$ $

 \begin{example}\label{metex}$ $
\begin{enumerate}
\item  (The \textbf{truncated semiring pair}, see
    \cite[Example~4.2(v)]{Row21}.
  \todo{SG: previous reference incorrect.
    \cite[Example~3.8(vii)]{Row21}.}
  Fix  $m\ge 3$ in the ordered monoid  $(\Net,+),$  and define $$a_1``+"a_2= \min\{a_1+a_2,m\}, \qquad a_1``\cdot"a_2= \min\{a_1a_2,m\}.$$

      \begin{enumerate}

       \item Take  $\mcA = \{0,\ldots, m\},$  $\tT = \{ 1\}$, and $\mcA_0 =\{0,m\}.$ But this pair lacks weak \PropertyN.
      \item Take  $\mcA = \{0,\ldots, m\},$  $\tT = \{ 1\}$,  and $\mcA_0 =(\mcA\cap 2\Net ) \cup\{m\}.$ This pair is of the first kind, and  has   unique negation.
   \end{enumerate}

\item      Here are some  pairs  lacking unique negation.
 $\tT$ is an arbitrary cancellative monoid, $\mcA = \tTz \cup \{\infty\},$ $\mcA_0 =  \{\zero,\infty\},$ $\infty+\infty =\infty,$ and $b_1+b_2= \infty$ for all $b_1 \ne b_2$ in $\tT \cup \{\infty\}$.
There are two kinds:
\begin{itemize}
  \item First kind. Then $a+a = \infty$ for all $a\in \tT$. This pair is {N-transitive}.
  \item Second  kind. Then $a+a = a$ for all $a\in \tT$. This pair is not N-transitive.
\end{itemize}


\end{enumerate}
 \end{example}

 \begin{definition}
A $\tT$-module     $\mcA$
is \textbf{nonarchimedean} if there is $a\in \tT$ such that $\one + a = a.$
 \end{definition}

 \begin{lemma}\label{bA2}
  Every null-bipotent pair  $(\mcA,\mcA_0)$ over a group $\tT$ which is not the $\mcA_0$-minimal pair of Example~\ref{metex}, is nonarchimedean.
\end{lemma}
\begin{proof}
  By hypothesis there is $a \in \tT$ such that $a +\one \in \tT$, and so $a +\one \in \{\one, a\}.$ If $a + \one  =\one$ then multiply by $a^{-1}.$
\end{proof}

  \subsection{Supertropical pairs}\label{sut}$ $

\begin{example}\label{modtr1}$ $\begin{enumerate}
    \item  Suppose $\tGz$ is an ordered monoid with absorbing minimal element $\zero_{\tG}$, and $\tTz$ is a monoid with absorbing element $\zero_{\tT}$, together with an onto homomorphism $\mu: \tTz\to \tGz$.
Take the action $\tTz\times\tGz\to \tGz$ defined by $a\cdot g = \mu(a)g.$
Then $\mcA $, defined as the disjoint union $ \tTz\cup  \tG$, with  $\zero_{\tG}$ and $\zero_{\tT}$ identified, is a multiplicative monoid  when we extend the given multiplications on $\tTz$ and on $\tG,$ also using the given $\tT$ action.

Setting $\mu(g)=g$ for all $g\in \tG,$
we  define addition on $\mcA$ by $$b_1+b_2 = \begin{cases}
 b_1 \text{ if } \mu(b_1)>\mu(b_2),\\
b_2 \text{ if } \mu(b_1)<\mu(b_2),\\
\mu(b_1) \text{ if } \mu(b_1)=\mu(b_2).
\end{cases}.$$

  We call $(\mcA,\tGz)$ the \textbf{supertropical pair arising from} $\mu.$
 $(\mcA,\tGz)$ is  a nonarchimedean pair of the first kind, of tropical type, and $\mcA$ is of characteristic $(1,2)$.
\begin{enumerate}
    \item  When $\mu$ is a monoid isomorphism, this is a slightly more general way of defining the supertropical semiring of \cite{IR}.

    \item Here is the initial supertropical pair.

For $\tT = \{\one\},$
we modify the  semifield  $\tTz= \{ \zero, \one\} $  to
the \textbf{super-Boolean pair}, defined as  $(\mcA,\tGz)$ where $\mcA= \{ \zero,\one,e\}$ with
$e$  additively absorbing, $\one+\one=e,$ and $\mcA_0 = \{ \zero, e\}.$

The super-Boolean    pair  is isomorphic to the  sub-pair  generated by $\one$ of each   supertropical pair.

 \item At the other extreme, taking $\tG = \{\one_\tG\} =\{ e\}$ yields
 the trivial pair (\Cref{tr0}(ii)).

 \item Non-isomorphisms $\mu$ which are not trivial give other variants such as $\absl{\phantom{w}}: \C^\times \to \R^\times,$ which we do not explore here.
\end{enumerate}
\item One can modify the supertropical pair of (i), by declaring $b_1 +\mu (b_1) = b_1.$ In particular, $\one + e =\one.$ $\mcA$   now is of characteristic $(2,1)$.
\end{enumerate}
  \end{example}

\subsection{Hyperpairs and hyperring  pairs}\label{hyps0}$ $

 Hyperpairs provide an important class of pairs which may fail to be metatangible, and provide a fine source of counterexamples.
Some of the finer points, in the relationship of ``systems'' and ``hypersystems,''  were studied in \cite{AGR2}, which
also provides other examples concerning ``fuzzy systems,'' tracts,
matroids,    and ``geometric'' systems.

As noted in \cite[\S 2.4]{JMR1}, the same constructions of
Krasner \cite{krasner} can be carried out quite generally for semirings without a
negation map, and yield a  pair, which may have \PropertyN.
We briefly consider an even more general situation.

\begin{definition}\label{hy7}  Let  $\mathcal H$ be a set, and $\tT$ a subset of $\mathcal H$.  We denote by $\nsets (\mathcal H)$ the set of non-empty subsets of $\mathcal H$.   \begin{enumerate}
    \item $\mathcal H$  is called a $\tT$-\textbf{hypermodule}, when we are given
  \begin{enumerate}

     \item A commutative multivalued \textbf{hyperaddition} $\boxplus : \mathcal{H}\times \mathcal{H}\to \nsets (\mathcal H),$ which is \textbf{associative }   in the sense that if we
 define
\[ a \boxplus S = S\boxplus a =\bigcup _{s \in S} \ a \boxplus s,
\]
 then $(a_1
\boxplus a_2) \boxplus a_3 = a_1 \boxplus (a_2\boxplus a_3)$ for all
$a_i$ in $\mathcal H .$

\item A distinguished element $\one$ of $\tT$, and binary operations $\tT \times \mcH \to \mcH$ and $ \mcH \times \tT \to \mcH$ denoted by concatenation, distributing over $\boxplus$ and such that
$\one b=b \one= b$ and
  $(a_1b)a_2=a_1(ba_2)$ for $a_i\in \tT$ and $b\in \mcH$.
Then we define the action of $\tT$ on $ \nsets (\mathcal H )$ by $aS = \{as : s\in S\}$ and  $Sa = \{sa : s\in S\}$.

  \item An absorbing element $\Hzero\in \mathcal H$ (for the action of $\tT$) called \textbf{the hyperzero}, satisfying $\Hzero \boxplus a = a \boxplus \Hzero = a$ for all $a\in \nsets (\mathcal H )$. We assume that $\Hzero \notin \tT$,
  and denote $\tTz = \tT \cup \{\Hzero\}$.
\end{enumerate}
    \item  We view $\mcH$ as the set of singletons in $\nsets (\mathcal H ),$ identifying $a\in \mathcal{H}$ with~$\{a\}$.  Defining $  S_1 \boxplus S_2 := \bigcup_{s _i\in S_i}\left(s_1 \boxplus s_2 \right)$ makes $\nsets (\mathcal H )$ a  weakly \admissible\ $\tT$-bimodule.
 Then, take $\mcA$ to be the $\tT$-sub-bimodule of $ \nsets (\mathcal H )$ that is $\boxplus $-spanned by $\mcH$, which we call the $\tT$-\textbf{bimodule} associated to~$\mcH$.
 \item    $\mcH$ is called a \textbf{hypersemiring} if it is an $\mcH$-hypermodule, where $\mcH = \tT$ is a monoid for concatenation.

 \item  $\mcA$ is called the \textbf{nd-semiring} associated to a hypersemiring $\mcH$ if $\mcA$ is equipped with multiplication taken elementwise from $\mcH$; i.e., $S_1S_2 = \{s_1s_2: s_i\in \mcH\}$ 
 ($\mcA$ is a $\mcH$-semiring  if multiplication distributes over hyperaddition of $\mcH$.)
\footnote{In \cite{AGR1} we also defined multiplication in a way that multiplication distributes over hyperaddition of $\mcH$. But often we want to keep the original multiplication}

    \item\label{hyperpairsviii}   Take a $\tT$-sub-bimodule $\mathcal S_0\subseteq \mcA$ such that  $\mathcal S_0 \cap \tTz= \{\Hzero\}.$ Then we get a \textbf{$\tT$-bimodule pair} $(\mcA,\mcA_0)$ for
   $\mcA_0 = \{S \in \mcA : S_0 \cap S \ne \emptyset     \text{ for some } S_0 \in \mathcal S_0\}. $

  \item  $(\mcA,\mcA_0)$ is the \textbf{hyperpair} (over $\tT)$ associated to $\mathcal H $, when $\mathcal S_0 = \{\Hzero\}$ in \ref{hyperpairsviii}.

   \item A \textbf{hyperring} is a hypersemiring
  for  which every element $ a \in \mathcal H  $ has a unique \textbf{hypernegative} $-a \in \mathcal H  $ \footnote{One could generalize this to a hypernegative on $\tT$-hypermodules,  stipulating that $(-)\one \in \tT$.}, in the sense that
$\mathcal \Hzero \in a \boxplus  (-a).$

    \item A hyperring~$\mathcal H $ is a \textbf{hyperfield} if $\mathcal H \setminus \{\Hzero\} $ is a  multiplicative group.

  \item A \textbf{hyperring pair} is a  hyperpair $(\mcA,\mcA_0)$ associated to a  hyperring $\mcH$ (with $\tTz = \mcH$). Then $(\mcA,\mcA_0)$ is uniquely negated.
  A \textbf{hyperfield pair} is a  hyperpair $(\mcA,\mcA_0)$ associated to a  hyperfield $\mcH$.
   \end{enumerate}
  \end{definition}

\begin{rem}$ $
   \begin{enumerate}
       \item We   saw in \cite[Lemma~5.3]{AGR2} that    in a hyperpair $(\mcA,\mcA_0)$,  one has:

$\quad  a_1 \in a_2 \boxplus a_3 \quad \text{iff} \quad a_3 \in
a_1  \boxplus (-a_2).$

\item A hypermodule is \textbf{stringent} in the sense of \cite{BSu} if and only if its hyperpair is metatangible.
   \end{enumerate}
\end{rem}




 \begin{lemma}\label{tanplu3}
  If  $(\mcA,\mcA_0)$ is the hyperring pair of a hyperring $\mcH$, then    for any $S_1,S_2 \in \mcA \subset \nsets(\mcH),$ there are $a_1,a_2\in \mcH$ for which $a_1 S_1 \boxplus a_2 S_2 \in \mcA_0.$
 \end{lemma}
 \begin{proof}
     Take $a_2\in S_1$ and $a_1 = (-)a$ for some $a\in S_2.$ Then $\zero \in a_1 a_2 \boxplus (-)a_2a_1 \subseteq a_1S_1 \boxplus a_2S_2.$
 \end{proof}

\begin{rem} $ $
 \begin{enumerate}
     \item  For  a  hyperpair, we have the surpassing relation $S_1 \preceq_\subseteq S_2$ when $S_1 \subseteq S_2$.

      \item
      When $\mathcal H$ is a hyperring, its hyperpair $(\mcA,\mcA_0)$  is a semiring if multiplication in $\nsets (\mathcal H )$   distributes over addition.\footnote{In general $\nsets (\mathcal H  )$  always satisfies $(\boxplus _i S_i)(\boxplus _j S_j')\subseteq \boxplus _{i,j} (S_iS_j')$, cf.~\cite[Proposition 1.1]{Mas}.}.   Otherwise, $\mcA$ need not even be closed under multiplication.

      \item In a hyperpair one defines the negation map $(-)S := \{ (-)s: s\in S\}. $ Then, as in \cite{AGR2}, obviously $S_1 \,\nablaminus \, S_2$ if and only if there exists $b\in \mcH$ such that  $S_1 \, \nablaminus \, b$ and $b\, \nablaminus \,S_2$. This property is similar to the following one.
 \end{enumerate}
\end{rem}

\begin{lemma}\label{hypr}
   For  a hyperring pair,  $S_1 \nablaminus S_2$ for sets $S_1,S_2,$ iff $S_1 \cap S_2 \neq \emptyset.$
\end{lemma}
\begin{proof} If $a\in S_1 \cap S_2,$ then $(-) a \in S_2$ so
$\zero\in a\boxplus (-a)\subseteq S_1\boxplus (-) S_2$, so $S_1 \nablaminus S_2$.  Conversely,
  if $S_1 \nablaminus S_2$ then $\zero\in S_1 \boxplus (-) S_2$,
  so there exists $a\in S_1$ and $b\in (-) S_2$ such that $0\in a\boxplus b$,
  so by unique negation $b=(-a)$, and since then $(-)b \in S_2$,
  we get that $a\in S_1\cap S_2$.
\end{proof}

Here are the appropriate morphisms for the hypertheory \cite{Vi}, special cases of weak morphisms (\Cref{symsyst}) and $\subseteq$-morphisms (\Cref{pmor} below) of pairs.

\begin{definition}\label{hmorph} We consider maps $f: \mathcal H_1 \to \mathcal H_2  $, where $H_i$ is a $T_i$-hypermodule, for $i=1,2$, satisfying $f(ab) = f(a)f(b)$ for all $a\in \tT_1\subset H_1,$ $b\in \mcH_1$, with $f(a_1)\in \tT_2\subset H_2$. \begin{enumerate}
    \item
A \textbf{weak hypermorphism} of hypermodules is a map $f: \mathcal H_1 \to \mathcal H _2  $ for which  $\Hzero \in \boxplus a_i $ implies $0 \in \boxplus f(a_i)$, for $a_i \in \mathcal H_i. $

A $\subseteq$-\textbf{hypermorphism} of hypermodules is a map $f: \mathcal H_1 \to \mathcal H _2  $ for which  $f(\boxplus a_i )\subseteq \boxplus f(a_i)$,   $a_i \in \mathcal H_i. $
\end{enumerate}
\end{definition}

\subsubsection{The  hyperfield of signs}\label{signsem}$ $

There is a hyperfield of special relevance
to this paper.
\begin{example}$ $
 \begin{enumerate}
\item  The \textbf{hyperfield of signs} $L := \{0, 1, -1 \}$  has the intuitive multiplication
law, and hyperaddition defined by $1 \boxplus  1 = 1 ,$\ $-1
\boxplus  -1 = -1  , $ $\ x \boxplus  0 = 0 \boxplus  x = x $ for
all $x ,$ and $1 \boxplus  -1 = -1 \boxplus  1 = \{ 0, 1,-1\} .$

\item Alternatively, we define the \textbf{sign
semiring} $L = \{0, 1, -1,\infty \},$  with the usual multiplication
law, where $\infty$ denotes an additively absorbing element, and
addition is defined by $$1 +1 = 1 ,\qquad -1 + (-1) = -1  ,\qquad x + 0 = 0+
x = x,\  \forall x ,\qquad 1 + ( -1) = \infty .$$ We have the natural
negation map satisfying $(-)1 = -1$. Hence $\infty = 1 ^\circ$  and $(-)\infty
= \infty$. We take $L_0 = \{\infty\}.$
\end{enumerate}
\end{example}

 \begin{rem}
 Here are some instances of isomorphisms involving hyperpairs.
 \begin{itemize}
   \item
 There is an isomorphism from the sign semiring pair
to the    hyperpair of signs, given by $$ -1 \mapsto
-1, \qquad  0 \mapsto 0, \qquad  +1 \mapsto +1, \qquad  \infty
\mapsto \{ 0, 1,-1\} .$$  It also is isomorphic to the doubled Boolean pair, seen by sending $\infty \mapsto \{ \one,\one\}.$

\item  (As in \cite[Example 3.17]{AGR2})  The hyperpair of the tropical hyperfield of \cite{Mit} and \cite[\S 5.3]{Vi} is isomorphic to the corresponding supertropical pair of \Eref{modtr1}(i), sending $a\mapsto a$ and $[-\infty,a]\mapsto a^\mu$.

\item  Viro \cite[\S 4.7]{Vi} has another hyperfield,
isomorphic to the modified supertropical pair of \Eref{modtr1}(ii), sending $a\mapsto a$ and $[-\infty,a)\mapsto a^\mu$.

\end{itemize}
\end{rem}

 \subsubsection{Quotient hyperpairs}$ $

 Krasner \cite{krasner} discovered a construction which ties in
 beautifully to classical field theory and arithmetic.
 We present it more generally in the context of semigroups; also see \cite{Row25}.

 \begin{theoalpha}\label{theoalphaE}$ $
 \begin{enumerate}
     \item Suppose $(\mathcal S,+)$  is a semigroup and $f:\mathcal S \to \bar {\mathcal S}$ is any  set map onto a set $\bar {\mathcal S}$. Define   hyperaddition $\boxplus: \bar {\mathcal S} \to \mathcal{P}(\bar {\mathcal S})$
  by $\bar a \boxplus \overline{a'} = \{ f(a+a'): f(a) = \bar{a}, f(a') = \overline{a'}\},$ for $a,a'\in {\mathcal S}.$ Then hyperaddition is associative. 
  \item Suppose a monoid $G$ acts on a $\tT$-semiring $(R,+)$, such that the orbits are multiplicative, i.e., $(g_1 a)( g_2 a' )= g_1g_2(aa').$ Write $\overline { b} = Gb,$ for $b\in R,$ and consider the map $f:R\to \bar R, $ given by $ b\mapsto \overline{b}$. Denote $\bar \tT = \{f(a): a\in \tT\}$ and  $ \bar R = \{f(b): b\in R\}$.  Then $\bar R$ is a $\bar \tT$-\hypersemiring, with hyperaddition of (i), with $\tT$ acting naturally on the cosets, i.e. $\overline{a} \overline{b}:= \overline{ab}.$

    \item  Suppose that $\tTz$ is a monoid. For any $\tT$-semiring $R$, and any subgroup $\tG$ of $\tT$ which is a normal subgroup of $R$ (which is the case if $R$ is commutative), the set of multiplicative cosets $R/\tG = \{ b\tG: b \in R\}$ is a $(\tT/G)$-\hypersemiring,  as in (ii).
Let $\mcA = \nsets(R/\tG)$. Then $(\mcA,\mcA_0)$ is a $\tT/\tG$-hyperpair,
 where:
\begin{enumerate}
    \item   $\mcA_0 = \{S \in \mcA: \zero \in S\}$.

   \item (generalizing (a)) given   a multiplicative ideal $M$ of $R$,     $\mcA_0 = \{S \in \mcA:  S\cap M\ne \emptyset\}.$
\end{enumerate}
 \end{enumerate}
\end{theoalpha}

 \begin{proof} (i) As in  \cite{krasner}.
  The verification is the same as given in \cite[Proposition~2.13]{JMR1}. Namely $$(\bar a \boxplus \overline{a'})+\boxplus \overline{a''} = \{ \overline{a+a'+a''}: f(a) = \bar{a}, f(a') = \overline{a'}, f(a'')=\overline{a''}\} =\bar a \boxplus (\overline{a'} \boxplus \overline{a''}).$$

     (ii) First note by induction that $   \boxplus b_i  \tG = \{ \sum a_i \tG : \quad a_i   \in b_i\tG \}.$ Then  $$ a\tG(\boxplus_{i=1}^{m} b_i  \tG )= \left\{ a\tG  \sum a_i'\tG:\ a_i' \in b_i \tG\right\}    =\left\{ (\sum a  a_i')\tG: \,a_i' \in b_i \tG\right\} =  \boxplus_{i=1}^{m} a\tG b_i  \tG. $$
     Hence $\mcA$ is a $\bar R$-bimodule, implying $\bar R$ is a $tT$-\hypersemiring.

     (iii)(a),(b)  The  hypotheses of (ii) are satisfied, and $\mcA_0$ is obviously a multiplicative ideal.
   \end{proof}

 \begin{definition}\label{Krs}
 We call $R\to \mcA = R/\tG$ the \textbf{Krasner map} and say that the  pair $(\mcA,\mcA_0)$ of \Cref{theoalphaE} is a \textbf{residue  hyperpair}, or \textbf{quotient hyperpair}.
 \end{definition}
 \begin{example}\label{KrT}
 Suppose that $R$ is a $\tT$-semiring with a multiplicative homomorphism $\mu: R \to \mcM,$ where $\mcM$ is an ordered abelian group. For a subgroup $\tG\subseteq \tT$, we  form the quotient $(\tT/G)$-hyperpair  $\mcA: = \nsets (R/\tG)$. 
 If $\gamma =\max \{ \mu ( g) : g\in \tG\}\le \one$ is bounded on $\tG,$ for example, if $\tG$ is comprised of roots of~$\one$, we define  $$\mu (b\tG )=   \gamma \mu(b ). $$
Normalizing, we can take $\gamma=1$ and  $\mu (b\tG )= \mu(b).$ Then $\mu$ is a modulus on $\mcA$.
 \end{example}

Many hyperpairs are   quotient hyperpairs, cf.~\cite{MasM};
 in particular, the    hyperpair of signs   is isomorphic to the quotient hyperpair arising from $\R/\R^+.$
  As noted in \cite{MasM}, other examples   are not so obvious, such as the ``phase hyperfield'' which is isomorphic to $\C/\R^+$,  and the Krasner hyperfield $F/F^*$. Nevertheless, many hyperpairs are not quotient , hyperpairs, cf.~\cite{HoJ,Ho}.



\subsection{Doubling}\label{Making}$ $

There is a general way, inspired by \cite{Ga,AGG2,GK}), to embed any   pair   into a   pair with a negation map  of the second kind.

\subsubsection{Doubling of a module}

\begin{definition}\label{Sym2}(See~\cite{Ga}, \cite[Theorem~4.2]{AGR2}; analogous to symmetrization in \cite{Row21}). For
any $\tT$-bimodule $(\mcA,+,\zero)$. We define $\widehat{\mcA} = \mcA\times \mcA,$   and $$\widehat{\tT } = ({\tT }\times \zero )\cup (\zero
\times {\tT }) \subset \widehat{\mcA} ,$$
$$\widehat{\tT}_\zero = ({\tTz }\times \zero )\cup (\zero
\times {\tTz }) =  \widehat{\tT } \cup \{\zero,\zero\} \subset \widehat{\mcA} .$$

 The \textbf{twist action} of $\widehat {\tT}_\zero $
on $\widehat{\mathcal A}$
 is defined as follows:

 \begin{equation}\label{twis} (a_0,a_1)\ctw (b_0,b_1) =
 (a_0b_0 + a_1 b_1, a_0 b_1 + a_1 b_0),\ (a_0,a_1) \in \widehat {\tT}_\zero, (b_0,b_1) \in \widehat{\mathcal A}.  \end{equation}

When $\mcA$ is a nd-semiring, one can define multiplication in $\widehat{\mcA} $
by \eqref{twis} for $a_i\in \mcA.$
\end{definition}

\begin{example}
\label{Boo1} The   \textbf{doubled Boolean semifield}, which is $$\{ (\zero,\zero), (\one  ,\zero), (\zero, \one),  ( \one, \one) \}.$$

This can be rewritten with a negation map as $\{\zero, \one, (-)\one, \infty\}$ where $\one (-) \one = \infty.$
\end{example}

In \Cref{theoalphaF} we shall see that the
doubling procedure   defines a  functor of  the relevant categories, providing
  a negation map of the second kind.

\begin{lemma}\label{sw0}
     Suppose that $\mcA$ is a weakly \admissible\
$\tT$-module (resp.~mon-module, resp.~gp-module).
Then \begin{enumerate}
    \item $\widehat{\mcA}$ is a weakly \admissible\ $\widehat{\tT }$-module (resp.~mon-module, resp.~gp-module).

      \item   There is an embedding ${ A }\to
\widehat {\mathcal A }$ given by $b \mapsto (b,\zero)$.

\item  $\widehat {\mathcal A }$ has the same characteristic as
 $\mathcal A$.
\end{enumerate}
\end{lemma}
\begin{proof}
      $\one_{\widehat{\tT}} =
       (\one,\zero)$ is the unit element since $(\one,\zero)(b_1,b_2) = (b_1+\zero,b_2+\zero).$
          When $\tT$ is a group, $\widehat{\tT}$ is a
group since $(a,\zero)^{-1} = (a^{-1} ,\zero)$ and $(\zero,a)^{-1} =
(\zero,a^{-1} ,\zero)$.
The rest is clear.
\end{proof}
\begin{lemma}\label{sw} Notation as in \Cref{sw0}, define
$\widehat{\mcA}_0 = \{(b,b): b \in A \}.$ Then \begin{enumerate}
    \item  $(\widehat{\mcA},\widehat{\mcA}_0 )$ is a weakly \admissible\ $\widehat{\tT }$-pair  (resp.~mon-pair, resp.~gp-pair).

    \item The ``switch'' $(-)(b_0,b_1)  := (b_1,b_0)$ is a negation map  of the second kind on $\widehat {\mathcal A }$,  under which
$(\widehat {\mathcal A },\widehat {\mathcal A }_\zero)$ is uniquely negated.
\item The surpassing relation $\hat{\preceqzero}$ on $\widehat{\mcA}$ is given by $(b_1,b_2)\hat{\preceqzero} (b_1+b,b_2+b),$ for $b_1,b_2,b\in \mcA.$

     \item $\widehat{\mcA}$ is an nd-semiring   (resp.~semiring) when $\mcA $ is  an nd-semiring   (resp.~semiring).

\end{enumerate}
\end{lemma}

\begin{proof}
 $(-)(\one,\zero)= (\zero,\one)$, and the
switch is a negation map of the second kind. Clearly $\widehat{\mcA}^\circ =\widehat{\mcA}_0 $.
If $(a,\zero)\in \widehat{\tT}, $ then its only negation is $(\zero,a).$

The other assertions are immediate.
 \end{proof}

 The Krasner map  is functorial with respect to doubling.
 \begin{lemma} Suppose $\tT$ is a group.
 If $\mcA$ is a $\tT$-bimodule then
 $(\widehat{\mcA/\tT};\subseteq) \cong  (\widehat{\mcA}/\widehat \tT;\subseteq).$
 \end{lemma}
 \begin{proof}
  One can first double and then mod out the group $\widehat \tT$, or first  mod out the group $\tT$ and then double, obtaining the same structure.
 \end{proof}

Doubling also provides a useful ``norm.''

\begin{lemma}\label{abs1} For any triple $(\mcA,\mcA_0,(-)),$
there is a  multiplicative map
$\widehat {\mathcal A } \to \mcA$
sending $$(b_1,b_2)\mapsto \absl{(b_1,b_2)}: = b_1 (-) b_2.$$
\end{lemma}
\begin{proof}\begin{equation}
    \begin{aligned}
     \absl{(b_1,b_2)\ctw(c_1,c_2)}& = \absl{b_1c_1+b_2c_2, b_1c_2+b_2c_1}\\ & = b_1c_1+b_2c_2 \, (-)   b_1c_2 \, (-) b_2c_1 \\ & = (b_1 (-)b_2)(c_1 (-)c_2) = \absl{(b_1,b_2)}\, \absl{(c_1,c_2)}.
    \end{aligned}
\end{equation}
\end{proof}



\subsubsection{$\nabla$-Doubling of a pair and of a triple}\label{dbl1}$ $

Now, given instead a pair $(\mcA,\mcA_0)$  with a balance relation $\nabla$, we modify the doubling construction to   embed $(\mcA,\mcA_0)$  into a pair with a negation map. This method will be  relevant to our
treatment of matrix theory.

\begin{lemma}[Generalizing \Cref{Sym2} in view of \Cref{tr1a}]\label{dbl}  Given an
\admissible\ pair $(\mcA,\mcA_0)$ with a balance relation $\nabla$,
consider the $\widehat{\tT }$-bimodule
$\hat {\mcA}$ with $\hat {\mcA} = \mcA\times \mcA $ and $\widehat{\tT } = ({\tT }\times \zero )\cup (\zero
\times {\tT })$ as in \Cref{Sym2}.
Now define ${\widehat {\mcA}_{0}} = \{ (b_1,b_2) \in  \widehat {\mcA}: b_1 \nabla b_2\}.$
  Then   $\widehat{(\mcA,\mcA_0)}_\nabla:= (\hat {\mcA},{\widehat {\mcA}_{0}})$ is a $\widehat {\tT}$-pair, which will be called the
$\nabla$-\textbf{doubled pair}.  The switch is a negation map.
If $(\mcA,\mcA_0)$ is a mon-pair, then so is~$\widehat{(\mcA,\mcA_0)}_\nabla$.
\end{lemma}

 \begin{proof}
 $(b_1,b_2)\in {\widehat {\mcA}_{0}}$ and $a\in \tTz$ implies
 $(a,\zero)\ctw (b_1,b_2) = (a b_1 , a  b_2 )  \in {\widehat {\mcA}_{0}}, $ by definition of a balance relation.
 Similarly, $(b_1,b_2)\ctw (a,\zero)\in {\widehat {\mcA}_{0}}$, $(\zero,a)\ctw (b_1,b_2)\in {\widehat {\mcA}_{0}}$.
 The switch clearly preserves $\widehat{\tT }$. It also preserves
 ${\widehat {\mcA}_{0}},$ since $\nabla$ is symmetric.
 By definition $\nabla$ is additive, so the addition preserves ${\widehat {\mcA}_{0}}$. So ${\widehat {\mcA}_{0}}$ is a sub-bimodule of ${\widehat {\mcA}}$.
The remaining verifications are clear.
 \end{proof}

If $(\mcA,\mcA_0,(-))$ is a triple we use $\nablaminus$ for $\nabla.$
\begin{lemma}
  If
  $(\mcA,\mcA_0,(-))$ is an nd-semiring triple, then
 $\widehat {\mcA}_0 $ is an ideal of $ \hat {\mcA} $  when ${\mcA}_0 $ is an ideal of $  {\mcA}$.
\end{lemma}
\begin{proof} For $(c_1,c_2)\in \widehat {\mcA}_0 ,$
$$(b_1c_1+b_2c_2)(-)(b_1c_2+b_2c_1) = (b_1(-)b_2)(c_1(-) c_2) \in \mcA_0.$$
\end{proof}

\begin{example}
    The $\nabla$-doubling of the trivial pair $(\mcA,\zero)$ with the trivial balance relation is just the doubling of \Cref{Sym2}.
\end{example}

 $\widehat{(\mcA,\mcA_0)}_\nabla$  in general fails property~N, since $(\one,\zero)$ has quasi-negatives $(\one^\dag,\zero)$ and $(\zero,\one).$


\section{The tools of linear algebra over pairs}\label{lina}$ $

Finally we are ready for linear algebra over pairs. For the remainder of this paper, we assume furthermore:
 \begin{assump}\label{Note2}$ $
 \begin{enumerate}
     \item  $(\mcA,\mcA_0)$ is a $\tT$-mon-pair satisfying \PropertyN. \footnote{When $(\mcA,\mcA_0)$ has no given negation map  $(-)$,
we get one by  embedding it inside the doubled pair
$(\hat{\mcA},\hat{\mcA_0})$ of~\Cref{sw}.
 In this manner, one could reformulate the results of this paper for pairs in general.}

        \item    $( k \one )( k'\one) =(kk')\one$ for all $k,k'\in \Net^+.$

        \item   $(\tTz,\cdot,\one)$ is a commutative  monoid.

    \end{enumerate}
 \end{assump}
Let us now bring in the standard concepts of linear algebra.

\subsection{Matrices}\label{matra0}$ $

Some of the main notions of rank require matrices.      $A$~denotes an $n\times n$ matrix   over  $\mcA$.
Proceeding further requires a version of the determinant.
  \begin{definition}
A \textbf{tangible matrix} is a matrix all
of whose entries are in $\tTz$.
\end{definition}

    \begin{INote}
    We only deal with $\mcA^\natural,$ so we may assume that $\mcA =\mcA^\natural$  is admissible.

   One must take care that if the original $\mcA$ is already an nd-semiring,   $\mcA^\natural$ need not be an nd-semiring since its multiplication may not be closed,
but we still can define matrix multiplication over $\mcA^\natural$.
   The product of  any matrix in $\mcA^\natural$ times a tangible matrix  is defined, and in fact $A(BA') = (AB)A'$ for any tangible matrices $A,A'$ and any matrix $B$ over $\mcA$, but $  A^2 A^2$ and $AA^2A$ may differ for a tangible matrix~$A$.

  In short, our theory  is valid in either of the following situations: \begin{itemize}
      \item   $A$ is a tangible matrix or   \item $\mcA$ is a commutative semiring.
  \end{itemize}

  \emph{We make this assumption throughout the paper.}
  \end{INote}

\subsubsection{Singularity}\label{det1}
$ $


A \textbf{track} of an $n \times n$ matrix
$ A = (a_{i,j})  $ is a product $a  _\pi := a_{ \pi(1),1 } \cdots a_{ \pi(n),n }$ for $\pi\in S_n.$
When $A$ is tangible, the tracks clearly exist and are tangible or $\zero$, so  the following formulas 
make sense.

\begin{equation}\label{eq:tropicalDetsignsyma}
 {\Det A} _+ = \sum _  {\pi  \in S_n \text{ even }} a_{\pi}, \qquad {\Det A} _- =
 \sum _ {\pi  \in S_n \text{ odd }}   a_{\pi}.
\end{equation}

This allows us to define singularity of matrices with respect to a balance relation $\nabla$ as follows:

  \begin{definition}

The matrix $A$ is  $\nabla$-\textbf{singular} if ${\Det A} _+ \nabla {\Det A} _-$. The matrix $A$ is  \textbf{singular} if ${\Det A} _+ = {\Det A} _-$. The matrix $A$ is $\zero$-\textbf{singular} if ${\Det A} _+ = {\Det A} _- =\zero$.
  \end{definition}


\begin{rem}
   The matrix $A$ is  {singular} iff $(A,\zero)$ is  $\nabla$-{singular} in the doubled pair.
\end{rem}


  \begin{lemma}\label{dsum}
    If $\mathbf{v_n} = \sum_{i=1}^{n-1}\mathbf{v_i},$ then the matrix $A$ whose rows are $\mathbf{v_1},\dots, \mathbf{v_n}$ is $\nabla$-singular for any balance relation $\nabla$.
\end{lemma}
\begin{proof}
  The $+$ and
$-$ parts in \eqref{eq:tropicalDetsignsyma} match.
\end{proof}

\subsubsection{Special case: $(-)$-Determinants of matrices.}$ $

In this subsection we
assume for convenience that  $(\mcA,\mcA_0,(-))$ is a triple, noting that we can  obtain  a negation map by doubling a pair $(\mcA,\mcA_0)$, and embedding $(\mcA,\mcA_0)$ into $(\widehat \mcA, \widehat{\mcA_0})$ if necessary.
This allows us to define the $(-)$-determinant  as follows.

 \begin{definition}\label{mdet} For a triple $(\mcA,\mcA_0,(-)),$
the $(-)$-\textbf{determinant} $\Det A$ is ${\Det A} _+ (-)  {\Det A} _-.$
In other words,  for $\pi \in  S_n$, write $(-)^{\pi}$  for $(-)^{\sgn\pi}$.
Then $\Det A = \sum_{\pi \in  S_n}  (-)^{\pi} a_\pi$,
taken over the tracks $a_\pi$.

Thus a matrix $A$ is $\nablaminus$-singular if and only if $\Det A \in \mcA_0$.
\end{definition}

For pairs of the first kind, the $(-)$-determinant is just the permanent.

\begin{definition}\label{adjo} Write $(-)^0$ for $+$, $(-)^1$ for~$(-),$ and, inductively, $(-)^k$ for
$(-)(-)^{k-1}.$  Write $a_{i,j}'$ for
the $(-)$-determinant of the submatrix obtained by removing the $i$ row and $j$ column. The
\textbf{$(-)$-adjoint} matrix $\adj A$ is $( (-)^{i+j}a_{j,i}')$.
\end{definition}

Zeilberger \cite{zeilberger85} showed that a number of determinantal
identities of matrices admit bijective proofs, in the sense that the tracks pair off. (This also enables one to prove Lemma~\ref{singp}  without resorting to a negation map.)
Reutenauer and Straubing
\cite{ReS} observed  as a consequence that these   identities have
semiring analogs. They also showed that some of these analogues can
be proved by algebraic arguments.

A general version of the argument of Reutenauer and Straubing was formalized
as a ``transfer principle'' in~\cite{AGG1}. This principle
stated in particular
that a classical polynomial identity of the form $P=Q$
in which $P,Q$ are multivariate polynomial expressed as sums of distinct monomials has a valid version over any semiring with negation map. Observe that all the monomials
appearing in the expansion of the left-hand side of~\eqref{e-Lap1} are distinct, and so do all the monomials appearing at the right-hand side of~\eqref{e-Lap1}. Then, the result
follows by an immediate application of this principle.

This provides a machinery for lifting results   from matrix theory, often using $\nablaminus,$   illustrated as follows:

\begin{lemma}\label{deta} Assume that  $(\mcA,\mcA_0,(-))$ is a commutative semiring triple.  For any $A, A_i \in M_n(\mcA),$
\begin{enumerate}
    \item \label{aadj0}    $ \Det{A_1}\Det{A_2} \preceq_0 \Det{A_1A_2},$
 \item\label{aadj}  ${|A|}I  \preceq_0 {\adj A}A$
and ${|A|}I  \preceq_0 A \, {\adj A}$.
\end{enumerate}

Moreover, \ref{aadj} holds without the semiring hypothesis if $A$ is a tangible matrix. \end{lemma}

\begin{proof}
  By Reutenauer and H.~Straubing \cite[Lemma 3]{ReS}, we have $\adj{A}A=|A|I + B^\circ$ for some matrix $B\in M_n(\mcA)$. This implies~\ref{aadj}. Moreover,
  the strong transfer principle (Theorem~3.4 of \cite{AGG1}) implies that $|A_1A_2|=|A_1||A_2|+ c^\circ$ for some $c\in \mcA$. This implies~\ref{aadj0}.
\end{proof}
The following is an immediate corollary.
\begin{corollary} Assumption as in in \Cref{deta}.  For any $A, A_i \in M_n(\mcA),$ \begin{enumerate}

    \item   When $(\mcA,\mcA_0)$  is given a pre-surpassing relation $\preceq,$
\begin{enumerate}
    \item     $\Det{A_1A_2}\succeq \Det{A_1}\Det{A_2},$
 \item  ${|A|}I  \preceq {\adj A}A$
and ${|A|}I  \preceq A \, {\adj A}$.
\end{enumerate}
    \item  \begin{enumerate}
     \item     $\Det{A_1A_2}\, \nablaminus \, \Det{A_1}\Det{A_2},$
 \item  ${|A|}I \, \, \nablaminus \, \,  {\adj A}A$
and ${|A|}I  \, \, \nablaminus \, A \, {\adj A}$.
\end{enumerate}
\end{enumerate}
Moreover, (b) holds in (i) and (ii) without the semiring hypothesis if $A$ is a tangible matrix. \end{corollary}

Analogs of famous results, for  a pair  $(\mcA,\mcA_0)$ with a negation map:

\begin{corollary}\label{singp}
       The product of a $\nablaminus$-singular matrix $A$ with another matrix $B$ is singular.
\end{corollary}

\begin{proof}
    $|AB| \succeq_0 |A| |B| \in \mcA_0.  $
\end{proof}

\begin{theoalpha}[Cayley-Hamilton theorem]\label{theoalphaG}
    For $A \in M_n(\mcA),$ let $f(\la) =|\la I (-) A|.$ Then $f(A)\in \mcA_0.$
\end{theoalpha}
\begin{proof} (also cf.~\cite{IzhakianRowen2008Matrices})
Essentially a rewording of \cite{St}, since the extra terms come in pairs with opposite signs; see \cite{zeilberger85} for a clear graph-theoretic argument.
\end{proof}

\begin{theoalpha} [Generalized Laplace identity and Cauchy-Binet formula]
\label{theoalphaH}
Laplace's well-known identity $$|A| = \sum _{j=1}^n (-)^{i+j}
a'_{i,j}a_{i,j},$$ for any  $i$, holds  over a  pair with a negation map, where $A=(a_{ij})$
 is a square matrix.

More generally, fix $I = \{ i_1, \dots, i_m \}\subset \{1, 2, \dots ,
n\},$ and for any set $J = \{ j_1, \dots, j_m\}\subset \{1, 2, \dots ,
n\}$,    write    $(-)^J$ for $(-)^{ j_1 +\dots + j_m},$ $a_{I,J}$
for the $(-)$-determinant of the $m\times m$ matrix $(a_{i,j}: i\in I, j \in J)$, and
$a_{I,J}'$ for the $(-)$-determinant of the $(n-m)\times (n-m)$
matrix obtained by deleting all rows from $I$ and all columns from~$J$. Then \begin{align}\label{e-Lap1}|A| = \sum _{J : |J|= m} (-)^{I}(-)^{J} a'_{I,J}|a_{I,
J}|.\end{align}
\end{theoalpha}
\begin{proof}
This could be proved by a bijective argument, following
the idea of \cite[Exercise~3]{zeilberger85}.
Alternatively, we could rely on the approach of Reutenauer
and Straubing, who proved in \cite[Lemma 3]{ReS} the important special case where $m=1$, in which the positive and
negative parts of determinants are handled separately.
\end{proof}

(This result extends \cite[Example~3.8]{AGG1} and
\cite[Lemma~8.2]{PH}, the latter formulating an analogous result by
means of $|A|_+$ and $|A|_-$.)

The ideas of Zeilberger's bijective proofs, and the transfer principle, can take us quite far, but
as we shall see, one must be careful since some subtler conjectures have counterexamples.

 \subsubsection{$(-)$-Determinants over $\tT$-hyperrings   of Krasner type}$ $

Let us interpret $(-)$-determinants   over commutative
hyperrings of Krasner type.
 \begin{lemma} \label{krd} $ $\begin{enumerate}
\item
Let $R$ be a commutative ring, and $\tG$ a multiplicative subgroup, and as in \Cref{theoalphaE}, $\bar A = \big\{ (a_{i,j} g_{i,j}): a_{i,j} \in R, \, g_{i,j}\in \tG\big\}. $
Then
 \begin{equation}
     \label{Kd}
     \Det{\bar A} = \bigg\{\sum _{\pi  \in S_n} (-1)^\pi a_\pi g_\pi : g_\pi \in \tG\bigg\}.
 \end{equation}

 \item
 $\Det{\bar A} \in \mcA_0,$ if and only if $\sum _{\pi  \in S_n }(-1)^\pi a_\pi g_\pi = \zero$ for suitable $g_\pi \in \tG$.

 \item
 Vectors $(a_{i,1}\tG,\dots, a_{i,n} \tG), $
 $1\le i \le n$, are $\overline{\mcA_0}$-dependent if  $(a_{i,1} g_{i,1},\dots, a_{i,n} g_{i,n}), $ $1\le i \le n$, are dependent for suitable $g_{i,j}\in\tG.$
 \end{enumerate}
\end{lemma}
\begin{proof}
(i) By definition of the Krasner map.

(ii) One of the terms in the right side of \eqref{Kd} must be $\zero.$

(iii) Also by the Krasner map.
\end{proof}

But the converse to (iii)  could fail, since different elements of $\tG$ might be used in the sums.

\subsection{Vector pairs, dependence and bases}$ $

Rather than defining arbitrary module pairs over a pair $(\mcA,\mcA_0)$,   we only use:

\begin{definition}\label{vs}
   A \textbf{vector pair} over a pair $(\mcA, \mathcal A _0)$  is a pair $  (\mcV,\mcV_0),$ where $\mathcal V :=\mathcal A ^{(J)}$ denotes the set of $J$-tuples over $\mathcal A$ with only finitely many nonzero coefficients,  and $\mcV_0 =\mcA_0^{(J)}$.   The $\tT$-bimodule operations are defined componentwise.

\end{definition}

\begin{definition}\label{dep1}
$ $ \begin{enumerate}
    \item A set of
  vectors $\{\mathbf v_i \in
\mcV : i \in I\}$ is $\mathcal V_0$-\textbf{dependent} (or \textbf{dependent} for short), if $\sum _{i\in I'} a_i
\mathbf v_i \in \mathcal V_0$ for some nonempty finite subset $I' \subseteq I$ and $a_i \in \tT$.

   \item  A $\mathcal V_0$-\textbf{base} (written \textbf{base} for short) is an independent set of   vectors  $\{\mathbf v_i \in
\mcV : i \in I\}$  such that for any $\mathbf v\in \mcV$ there are $a,a_i\in \tT$   with  $( \sum_{i\in I} a_i {\mathbf v}_i + a\mathbf{v})\in \mcV_0.$

   \item  A \textbf{tangible vector} is a nonzero vector all of whose
elements are in~$\tTz .$

\end{enumerate}
\end{definition}

 \begin{lemma}
      Any independent set can be enlarged to a base.
 \end{lemma}
\begin{proof}
    The union of a chain of independent sets is independent, so apply Zorn's Lemma to get a maximal independent set, and then any vector is dependent on it.
\end{proof}

\begin{definition}
The 
\textbf{row rank} of a matrix is the maximal number of
independent rows.

The notion of \textbf{column rank}  is
defined analogously.

The \textbf{submatrix
rank} of $ A $ (with respect to a given balance relation $\nabla$) is the largest size of a nonsingular square submatrix
 of $ A $, with respect to a given balance relation $\nabla$.
\end{definition}


\section{Dependence in terms of submatrix rank: Conditions \A{1} thru \A{6}} $ $

From now on we take $J = \{1,\dots,n\}$,    $\mcV = \mcA^{(n)}$, and $\mcV_0 = \mcA_0^{(n)}$. When the pair $(\mcA,\mcA_0)$ has a negation map, we take $\nabla= \nablaminus;$  more generally, when the pair $(\mcA,\mcA_0)$  satisfies \PropertyN, we take $\nabla = \nablaT,$ or $\nabla = \nabladag$ which coincide on $\tT$ when $(\mcA,\mcA_0)$ is also N-transitive, by \Cref{bal4}.
 ``Singular'' means $\nabla$-singular.

 The main subject for the remainder of this paper is to
explore the following conditions for pairs  $(\mcA, \mathcal A _0)$ satisfying \PropertyN, which all hold in linear algebra over a field.


\begin{description}\label{Cond}
\item[\textbf{Condition \Al{1}}]
The submatrix rank is less than or equal to the row rank and the
column rank.

  \item[\textbf{Condition \Al{2}}]
  We say that Condition \A{2} is true for rank $r$, if
  the row rank and the
column rank  are less than or equal to the submatrix rank,
for any matrix with submatrix rank equal to $r$.
If Condition \A{2} is true for all $r\geq 1$, then we simply say that Condition \A{2} is true.

\item[\textbf{Condition \Al{3}}] We say that Condition \A{3} is true for rank $n$ if for any $n\times  n$ matrix $A$,
if the row rank or the
column rank of $A$ is $n$, then $A$ is nonsingular.
If Condition \A{3} is true for all $n\geq 1$, then we simply say that Condition \A{3} is true.

\item[\textbf{Condition \Al{4}}]  cf.~\cite[Question~5.13]{BZ}: The rows of an $m\times n$ matrix are dependent if $m>n.$

\item[\textbf{Condition \Al{5}}] For $m>n,$ either the rows of an $m\times n$ matrix $A$ are dependent or    $A$   has a singular $n\times n$ submatrix.

 \item[\textbf{Condition \Al{6}}] Any base of $\mcA^{(n)}$ has at most $n$ vectors.
 \end{description}

 Conditions \A{1} and \A{2} were taken from  \cite[\S9.2]{Row21}.
  Matrices over supertropical semifields   were seen  in
\cite[Theorem~3.4]{IzhakianRowen2009TropicalRank} to satisfy
Conditions \A{1} and~\A{2}, also cf.~\cite{IzhakianRowen2008Matrices}, and
the motivation for this paper comes from an attempt to generalize
these results and related results from \cite{Pl,Ga,AGG1,AGG2}.

\subsubsection{Implications among the Conditions}

\begin{rem}\label{imps}$ $
  \begin{enumerate}
      \item   Condition \A{2} implies \A{3}.
       \item  Condition~\A{3} implies \A{4}, since adding on   $m-n$ columns of zeroes gives a singular matrix, whose rows must thus be dependent.

     \item   Obviously Condition \A{4} implies  \A{5}, which is included because of \Cref{theoalphaR} below.

       \item   Condition \A{4} implies  \A{6}, since a base of   $m>n$ vectors would be dependent.
  \end{enumerate}
 \end{rem}

\subsubsection{Easy preliminary results}$ $

    Throughout,  $\bfv_i$ denotes the $i$ row of a matrix $A$.
 First we quickly dispose of an easy case.

\begin{definition}\label{rdef} Rows $\mathbf v_1, \dots, \mathbf v_m$
have \textbf{rank defect} $k$ if there are $k$ columns, which we
denote as columns ${j_1}, \dots, {j_k}$, such that $a_{i,j_u} = \zero$
for all $1 \le i \le m$ and $1\le u \le k$.
\end{definition}

\begin{proposition}\label{ssing1} Suppose that $(\mcA,\mcA_0)$ is a ZSF metatangible pair, cf.~\Dref{ZSF0}.
 Every $m\times m$ submatrix of a tangible $m\times n$ matrix~$A$ is   $\zero$-singular if, for some $1 \le k<m$, $A$ has $k$ rows
having rank defect $n+1-k.$\end{proposition}
\begin{proof}
This essentially is a theorem of Frobenius~\cite{Fro}, as explained
in \cite[p.141]{Sch}, since the proof is combinatoric, keeping
track of the    entries that are $\zero$.
 \end{proof}

\begin{proposition}\label{firstc37a}  Condition \A{5} holds for every  triple $(\mcA, \mathcal A _0,(-))$  with $\mcA = \tT \cup \mcA_0$.
\end{proposition}
\begin{proof}
   Suppose $\bfv_1, \dots, \bfv_{n+1}$ are row vectors with entries in $\mcA$.  Let $A$ be the $n\times n$ matrix with rows  $\bfv_1, \dots, \bfv_{n}$.
   Assume that $A$ is not singular, otherwise we are done.
 First assume that the first entry of $\bfv_{n+1}\adj A$ is in $\mcA_0$. Then the $n\times n$ matrix obtained by replacing the first row of $A$ by $\bfv_{n+1}$ is singular. Applying this to each entry  of $\bfv_{n+1}\adj A$ in turn,   we are done unless  $\bfv:= \bfv_{n+1}\adj A$ is tangible. But then since
    $$ \Det A \bfv_{n+1} \preceq \bfv_{n+1} \adj A A=   \bfv A,$$
    we get that  $(-) \Det A \bfv_{n+1} + \bfv A\in (\mcA_0)^n$.
    Since  $A$ is not singular,  so $\Det A$ is tangible, and all the entries of $\bfv$ are tangible, we get
    a tangible linear combination of  the vectors $\bfv_1,\dots, \bfv_{n+1}$ which is balanced, so the vectors are dependent.
\end{proof}


\Cref{firstc37a} will be improved in \Cref{firstc37b}, when we strengthen the notion of dependence.



%


Condition \A{4}, being weaker than Condition \A{2}, holds in many situations,  including hyperfield examples given in \S\ref{hy1}.

\subsection{Counterexamples to Conditions \A{2} and \A{3}}\label{neg}$ $

We  finally are ready for a detailed investigation of   dependence of vectors over pairs, relating to Conditions \A{1}, thru~\A{6}.
 Our
initial hope was that, in analogy to the supertropical
situation~\cite{IzhakianRowen2009TropicalRank}, Condition~\A{2} would
hold in general for   pairs satisfying \PropertyN\ (and then the column
rank would equal the row rank).
 However, the second author
observed that a (nonsquare) counterexample to Condition~\A{2} already
had been found in~\cite{AGG1}, even when the underlying pair  is
  null-bipotent. The essence of the example exists in the ``sign
semiring'' of Example~\ref{signsem}.

\begin{example}\label{firstc}
Recall the sign semiring pair  $(\mcA,\mcA_0)= ( \{ 1 , 0, -1, \infty\},\{ \infty\})$, and write $+$
for $+1$ and $-$ for $-1$. The matrix
\begin{equation}\label{impma1} \left(\begin{matrix}
                + & + & - & +\\
                + & - & + & +\\ - & + &  + & +
               \end{matrix}\right) \end{equation}
has row rank 3. Indeed, call the rows $\mathbf v_1,\mathbf v_2,\mathbf
v_3$, and suppose that $\mathbf v_1 +\a _2 \mathbf v_2 +\a _3 \mathbf v_3
\in \mathcal A _0,$ with $\a_i \in \tTz$ not all $\zero$. It is easy to see that
both $\a_2, \a _3  \ne 0$. From the first column we cannot have $\a
_2 = +, \ \a _3 = -.$ From the second column we cannot have $\a _2 =
-, \ \a _3 = +.$ From the third column we cannot have $\a _2 = -, \
\a _3 = -.$
 From the fourth column
we cannot have $\a _2 = +, \ \a _3 = +.$ Thus we have eliminated all
possibilities.

On the other hand, each $3\times 3$ minor is singular. Indeed, up to symmetry there are two
options:

\begin{equation}\label{comp}\begin{cases}\left|\(\begin{matrix}
                + & + & -  \\
                + & - & + \\ - & + &  +
               \end{matrix}\)\right| \quad  \text{expanded along the first row is} \;  (-)1(-)1  + \infty= \infty  ,\\ \\
             \left|\(\begin{matrix}
                + & + & +  \\
                + & - & + \\ - & + &  +
               \end{matrix}\)\right|     \quad  \text{expanded along the first row  is} \;(-)1(-)1  (-) \infty= \infty  .\end{cases} \end{equation}  \end{example}
Either way the $(-)$-determinant is $ \infty.$


 \subsubsection{Almost regular pairs}$ $

We  formulate this example quite generally for certain semiring pairs of the second kind.

\begin{definition}\label{strsec}
A  triple $(\mathcal A, \mcA_0,(-))$  is of the \textbf{tangible second
kind} if $a \nablaT a' $ for $a,a'\in \tT$ implies $ a+a'\notin \mcA_0$.
\end{definition}

\begin{lemma}\label{hyps}$ $
\begin{enumerate}
    \item Any  pair  of the tangible second
kind is of the  second
kind.
\item Conversely, any uniquely negated metatangible   pair  of the  second
kind is of the tangible second
kind.
\end{enumerate}
\end{lemma}
\begin{proof}
(i)  $a \nablaT a $, so  $ a+a\notin \mcA_0$.

(ii) If $a \nablaT a'$ for $a,a'\in \tT$ then $a(-)a' \in \mcA_0$ by \Cref{theoalphaD},
so   $a' =a,$ implying $a+a'\notin \mcA_0.$
\end{proof}

\begin{definition}\label{alr}
A pair $(\mathcal A, \mcA_0)$ satisfying \PropertyN\ is \textbf{almost regular} if $ a_1 + a_2 + a_3 \in \mcA_0$  and $ a_1 ^\dag+ a_2 + a_3 \in \mcA_0$ for $a_i\in \tTz$ imply that  $a_2 + a_3 \in \mcA_0$.
\end{definition}.

\begin{lemma}\label{alr1}  Any  null-bipotent pair is of the tangible second
kind  iff it is almost regular.
\end{lemma}
\begin{proof} $(\Rightarrow) $
 Suppose that $a_2 + a_3 = a_2$. Then $a_1 +   a_2 = a_1 + a_2 + a_3 \in \mcA_0, $ and likewise $a_1^\dag +a_2\in \mcA_0$,
so $a_1\nablaT a_1^\dag,$ contrary to hypothesis.

$(\Leftarrow) $ Take $a_3 = \zero$ and $a_1=a_2$ to get the contrapositive.
\end{proof}

\begin{example}
    The sign semiring pair is null-bipotent of the tangible second kind, and thus almost regular.
\end{example}

\begin{lemma}\label{rowrank3} Over any almost regular semiring pair $(\mcA,\mcA_0)$  of the tangible second kind, the matrix
\eqref{impma1} of
  Example~\ref{firstc} (writing $\one$ for $+$ and $\one^\dag$ for $-$)  has row rank~3.
\end{lemma}
\begin{proof} Otherwise there would be elements $a _i \in \tT$ such that
$\sum _{i=1}^3 a _i \mathbf v_i \in \mathcal A_0.$  The first and
fourth columns show that $$ a _1 + a _2 +a_3 \in \mathcal
A^\circ,\qquad  a _1 + a _2 +   a_3\one^\dag \in \mathcal A^\circ;$$ hence, by
hypothesis, $a_1  + a _2 \in \mathcal A_0 $. Likewise the  second and fourth columns
show that $a_1  + a _3 \in \mathcal A_0 $, so $a_2 \nablaT a_3.$ But  the  first and second  columns
show that $a_1  + a _2 \in \mathcal A_0, $
contrary to $(\mcA,\mcA_0)$ being of the tangible second kind.%
%
%
\end{proof}

\begin{theoalpha}\label{theoalphaI} Condition \A{2}  fails for the matrix   \eqref{impma1}, writing $\one$ for $+$ and $\one^\dag$ for $-$, taken
over any almost regular pair of the tangible second kind.
\end{theoalpha}
\begin{proof} As in  the computation in  \eqref{comp},  $
A$ is singular, so we conclude using~Lemma~\ref{rowrank3}.
\end{proof}%

\begin{corollary}\label{theoalphaI1} Condition \A{2}  fails for an $n\times n$
matrix, $n \ge 4$, over any almost regular pair of the tangible second kind.
\end{corollary}

 \begin{proof} We start with the $3\times 4$
 matrix $A$ of \eqref{impma1}, and
 form a $4\times 4$ matrix $B$ by taking its fourth row to be
the  repetition of the first row. Clearly, the row rank and
submatrix rank of $B$ are the same as for $A$. For higher $n$, we
just tack  an identity matrix of size $n-4$ onto $B$, along the
diagonal.
 \end{proof}

 \subsubsection{A counterexample to Condition \A{2}  over  quotient hyperfields}$ $

 \begin{example}\label{countq}
 Condition  \A{2} can fail for quotient hyperfield pairs of the second kind,  when $m= n\ge 3$.

 Indeed, take the rational field $F$  over~$\Q$ in $9$ indeterminates $\la_{i,j}$,  $1 \le i,j \le 3$,  and $6$ indeterminates $\mu_{\pi}$ for $\pi \in S_3$, and the generic $3\times 3$ matrix $A = (\la_{i,j}),$
 letting $\tG$ be the subgroup generated by the $\mu_\pi.$
 We impose the  relation
 $\sum_\pi a_\pi \mu_\pi = 0,$ thus giving us a matrix $\bar A$ of $(-)$-determinant $\bar 0$.  The projective space of solutions has dimension 5 .  But the projective solution space for independence of the rows has dimension 3, so we have solutions where the determinant over the Krasner residue field is 0, whereas the rows are independent.
 \end{example}

\subsection{Condition \A{1}}  $ $

Having battered Condition \A{2}, we turn to Condition \A{1}, which is much more agreeable.

\begin{rem}
Condition  \A{1} is equivalent to the assertion:\medskip

    If $m=n$ and the vectors $v_1,\dots,v_m$ are  dependent,
    then the matrix $A$ is singular.
\medskip

    (Indeed, one cuts down to $m\times m$ submatrices of $A$.)
  Thus, Condition  \A{1} is related to the conclusion of
\cite[Theorem~4.18]{AGG2}, as will be explained in \S\ref{Cr}.
\end{rem}

 We recall:

\begin{lemma}[\cite{BZ}]
Condition \A{1}     holds over   hyperfield pairs of Krasner type.
\end{lemma}
\begin{proof}
Taking  vectors $(a_{i,1} g_{i,1},\dots, a_{i,n} g_{i,n}), $ $1\le i \le m$ which are dependent,   their matrix has determinant $\zero.$

Thus the rows of $\bar A$ being dependent implies that its $(-)$-determinant  in $R/\tG$ contains $\bar \zero$.
\end{proof}

\subsubsection{A reformulation of Condition \A{1}}$ $

Applying the definition of $\nabla$ componentwise to the vector pair $(\mathcal V, \mcV_0)$, we obtain the  balance relation on vectors, i.e., $\mathbf{v} \, \nabla  \,\mathbf{v'}$ if    $ {v_i} \nabla \, {v'_i}$, for all $i$.
We  have the following reformulation of~Condition~\A{1} for triples:

\begin{lemma}\label{A1alt}
Consider the following properties over a triple $(\mcA,\mcA_0,(-))$,
 when $\nabla=\nablaminus$:
\begin{enumerate}
    \item\label{A1alt1}
For all $n\geq 1$, and $n\times n$ square matrices $A$ over ${\mathcal A}$, we have
\begin{equation} A \mathbf x\, \nabla \,\zero
\; \text{and}\;  \mathbf x\in \tTz^{(n)}\setminus\{\zero\} \implies
{|A|} \, \nabla \,\zero\enspace .
\label{nulldet}
\end{equation}
\item \label{A1alt2}Denote by $\one_{(n)}$ the $n$-dimensional column vector with all entries equal to $\one$, we have
for all $n\geq 1$, and $n\times n$ square matrices $A$ over ${\mathcal A}$, we have
\begin{equation} A \one_{(n)}\, \nabla \,\zero
 \implies
{|A|} \, \nabla \,\zero\enspace .
\label{nulldet1}
\end{equation}
\item \label{A1alt3}For all $n\geq 1$, $n\times n$ square matrices $A$ over ${\mathcal A}$, and vectors $v\in {\mathcal A}^{(n)}$, we have
\begin{equation}A \mathbf x\, \nabla \,\mathbf v
\; \text{and}\;  \mathbf x\in \tTz^{(n)} \implies
{|A|} \mathbf x\, \nabla \,{\adj A\mathbf v }\enspace .\label{cramer1}
\end{equation}
\end{enumerate}
Then, Condition \A{1} is equivalent to each of the properties
{\rm \ref{A1alt1}, \ref{A1alt2}} and {\rm \ref{A1alt3}}.
\end{lemma}
\begin{proof}
Condition \A{1} implies \ref{A1alt1}:
If $A \mathbf x\, \nabla \,\zero$, with $\mathbf x\in \tTz^{(n)}\setminus\{\zero\}$,
then the columns of $A$ are dependent,
  implying that the rank of $A$ is at most $n-1$, so by Condition \A{1}, this implies that the submatrix rank of $A$ is at most $n-1$, so $A$ is singular, that is $|A|\, \nabla \,\zero$.

\ref{A1alt1} implies Condition \A{1}:
 Let $r$ be the column rank of an $m\times n$ matrix $A$.  If $r=n$, then the submatrix rank of $A$ is at most
 $n$. Otherwise, any collection of
 $r+1\leq m$
 columns of $A$ is dependent. So the columns of any $(r+1)\times (r+1)$ submatrix $B$ of $A$ are dependent, that is there exists $x\in\tTz^{(r+1)}\setminus \{\zero\}$ such that $B\mathbf x\, \nabla \,\zero$.
 Hence $B$ is singular by (i).
 Since this holds for any $(r+1)\times (r+1)$ submatrix $B$ of~$A$, this shows that the submatrix rank of $A$ is less than $r+1,$ so is at most $r$.

 \ref{A1alt1} implies \ref{A1alt2}: take $\mathbf x=\one_{(n)}$.

\ref{A1alt2} implies \ref{A1alt1}: Let $\mathbf x\in \tTz^{(n)}\setminus\{\zero\}$ such that
$A \mathbf x\, \nabla \,\zero$.
Consider the square matrix $B$
with entries $B_{ij}=A_{ij} x_j$,
then we have $B\one_{(n)} \nabla \,\zero$.
This implies by \ref{A1alt2} that $x_1\cdots x_n |A|=|B|\nabla \,\zero$. If all entries of $x$ are in $\tT$ then, by \Cref{Note1}(ii) this implies that $|A|\nabla \,\zero$.
Otherwise, assume that the first entries of $x$, from $1$ to $m<n$ are in $\tT$ and the last ones are~$\zero$. Then the last $n-m$ columns of $B$ are~$\zero$. Let $C$ be any   $m\times m$ submatrix of $B$ obtained by taking the first $m$ columns of $B$
and $m$ arbitrary rows of $B$. Then
$C\one_{(m)}\nabla \,\zero$.
Applying \ref{A1alt2} to $C$, we deduce that
$|C|\nabla \,\zero$. Again,  \Cref{Note1}(ii)
  implies  that $|D|\nabla \,\zero$
for any $m\times m$ submatrix of $A$ obtained by taking the first $m$ columns of $A$
and $m$ arbitrary rows of $A$. Then by the block expansion, $|A|\nabla \,\zero$.

\ref{A1alt1} implies \ref{A1alt3}:
Suppose $A \mathbf x\, \nabla \,\mathbf v$. For each $j\in [n]$, consider  the matrix $B$ such that
$B_{ij'}=A_{ij'}$ for~$j'\neq j$ and $B_{ij}=A_{ij} x_j (-) v_i$,
and the vector $\mathbf x'$ with $x'_{j}=\one$ and $x'_{j'}=x_{j'}$
for $j'\neq j$; then
$B \mathbf x'\, \nabla \,\zero$,
so by \eqref{nulldet1} we get $|B|\, \nabla \,\zero$,
which implies
$$|A|x_j (-) ({\adj A} \mathbf v)_j\, \nabla\, \zero.$$
Since this holds for all $j$, we deduce that
$|A|\mathbf x \, \nabla \, {\adj A} \mathbf v$.

\ref{A1alt3} implies \ref{A1alt1}: take $v=\zero$ and using that one entry of $\mathbf x$ is not $\zero$, with
 \Cref{Note1}(ii). 
\end{proof}

Some of the above equivalences can be shown using only \PropertyN.

\begin{lemma}\label{nab-A1alt}
Assume that $(\mcA,\mcA_0)$ satisfies \PropertyN, and
take $\nabla = \nabladag$. Then
$b\nabla \zero$ is equivalent to $b\in\mcA_0$, and
we have the following properties:
\begin{enumerate}
    \item\label{nab-A1alt1}
For all $n\geq 1$, and $n\times n$ square matrices $A$ over ${\mathcal A}$, we have
\begin{equation} A \mathbf x\, \nabla \,\zero
\; \text{and}\;  \mathbf x\in \tTz^{(n)}\setminus\{\zero\} \implies
{|A|_+} \, \nabla \, {|A|_-}\enspace .
\label{nab-nulldet}
\end{equation}
\item \label{nab-A1alt2}
For all $n\geq 1$, and $n\times n$ square matrices $A$ over ${\mathcal A}$, we have
\begin{equation} A \one_{(n)}\, \nabla \,\zero
 \implies
{|A|_+} \, \nabla \, {|A|_-}\enspace .
\label{nab-nulldet1}
\end{equation}
\end{enumerate}
Then, Condition \A{1} is equivalent to each of property
{\rm \ref{nab-A1alt1}} and {\rm \ref{nab-A1alt2}}.
\end{lemma}
 \begin{proof}
The conclusions of {\rm \ref{nab-A1alt1}} and {\rm \ref{nab-A1alt2}} mean that
$A$ is singular with respect to $\nabla$, then the proof of the equivalence of
 Condition \A{1} and {\rm \ref{nab-A1alt1}} is the same as above. For the
proof of the equivalence of Condition \A{1} and {\rm \ref{nab-A1alt2}},
one remarks that $ab\nabla ab'$ with $b,b'\in \mcA$ and $a\in \tT$ is
equivalent to $b\nabla b'$ by using \Cref{Note1}(ii),
which solves the case when all the entries of $x$ are in $\tT$, and then
for the general case, we need to develop the positive and negative determinants by blocks.
\end{proof}
\begin{remark}
If we take $\nablaT$ instead of $\nabladag$ in the above statement, then
Condition \A{1} is equivalent to {\rm \ref{nab-A1alt1}} where the condition
$A \mathbf x\, \nabla \,\zero$ is replaced by $A \mathbf x\, \in \mcV_0$.
Moreover, for the equivalence with  {\rm \ref{nab-A1alt2}}, one needs   $\nablaT$ to be cancellative,
which holds for instance when $\tT$ is a group.
\end{remark}


  \subsubsection{Condition \A{1} for matrices   over a metatangible, uniquely negated  pair}$ $

Here is a step towards obtaining Condition \A{1}.
\begin{itemize}
    \item
Let us call a tangible column vector {\it satisfactory} if one of the following holds:  \begin{enumerate}
    \item Precisely one entry is some $a\in \tT,$ one entry is $(-)a,$ and all other entries are $\zero.$
    \item $(\mcA,\mcA_0)$ is of first kind of $\mcA_0$-{characteristic} $k$,    precisely $2k+1$   entries are $a\in \tT,$ and all other entries are $\zero.$
\end{enumerate}

 \item More generally, let us call a column vector {\it satisfactory} if either
   \begin{enumerate}
     \renewcommand{\theenumi}{(\roman{enumi}$'$)}
\renewcommand{\labelenumi}{\theenumi}
\item (i) above holds,
\item Precisely one entry is some $a^\circ$ with $a\in \tT,$ and all other entries are $\zero,$
    \item $(\mcA,\mcA_0)$ is of first kind of $\mcA_0$-{characteristic} $k$, and all the entries of the vector are either~$\zero$ or of the form $m_ia$, and  $\sum_i m_i =2k+1$.
\end{enumerate}
 Recall that $k$ depends uniquely on $\mcA.$ When $\mcA$ is $e$-final   then $k=1.$
 \item
    A satisfactory column is {\it  normalized} if $a=1.$
  \item A matrix $A$ is (normalized)  {\it satisfactory} if each column is (normalized)  {satisfactory}.
\end{itemize}

 Note that the sum of the entries in each column of a   satisfactory matrix is in $\mcA_0.$
 \begin{theoalpha}\label{A1partA} Suppose $(\mathcal A,\mcA_0)$ is a metatangible, uniquely negated  pair.
 \begin{enumerate}
 \item Any (tangible) column vector whose entries have a sum belonging to $\mcA_0$ is the sum of (tangible) satisfactory column vectors, or equivalently the tangible linear combination of normalized (tangible) satisfactory column vectors.
   \item The determinant of any (tangible) matrix all of whose  column sums are in $\mcA_0$ is a tangible linear combination of determinants of normalized satisfactory matrices.
 \end{enumerate}
 \end{theoalpha}

 \begin{proof} (i) Consider  a tangible vector $\bfv=(v_1,\ldots , v_n)^T\in \tT^{(n)}$
   such that 
 $\sum_{i=1}^n v_{i} \in \mcA_0$.
We call the {\it breadth} of $\bfv$  
to be the maximal number of $v_{i}$ whose sum is tangible.
If~$\bfv$ has breadth $n-1$, then by \Cref{met1} (ii,b),
   there is $i^*$ such that $\sum_{i\ne i^*} v_{i} = (-)v_{i^*}$.
Then, $\bfv=\sum_{i\ne i^*} \bfv_i$, where $\bfv_i$ is the column vector
with an $i$th-entry $v_{i}$ and an $i^*$-entry $(-) v_{i}$ and all other entries
equal to $\zero$.

So we may assume $v$ 
is not satisfactory and has breadth $\ne n-1.$
Applying \lemref{met1}(ii,a), we may
 normalize $v$ 
so that 
$a$, as defined above, is $(-)\one$.
Rearranging, we may assume that 
 $\bfv =(\one,\dots,\one, v_{i'+1},\dots, v_{n})$,  where $v_{i'+1}+\dots+ v_{n}=(-)\one$
and either $\one+\one=\one$, or  $\one+\one=\one^\circ$
 with $\mathbf i'+{\mathbf 1}\in \mcA_0$.
We get easily that
$\bfv = \sum_{l=i'+1}^{n} (-)v_{l}{\bfv}_{l}$, where
$\bfv_l=(\one,\dots,\one, \dots, (-) \one,\dots ,\zero )$,
where the   $(-)\one$ is in position $l$.

If $\one+\one=\one$, we get that
$\bfv_l = \sum_{i=1}^{i'}{\mathbf u}_{i,l}$, where
$ {\mathbf u}_{i,l}$ is the vector having $i$th entry $\one$,
$l$th entry $(-)\one$, and all other entries   $\zero$.
So $\bfv$ is the sum of satisfactory vectors of type (i).

Assume now that  $\one+\one=\one^\circ$ and so $\mathbf {i'}+1\in \mcA_0$.
If $i'+1$ is even, then we can write all vectors $\bfv_l$ as the sum of
satisfactory vectors of type (i) with only two entries equal to $\one$.

If $i'+1$ is odd, and the $\mcA_0$-characteristic of $\mcA$ is $k$,
then $i'+1=2k+1+2l$, and then $\bfv_l$ is the sum of $l$
satisfactory vectors of type (i) with two entries equal to $\one$,
and one satisfactory vector of type (ii).

For general entries $v_{i} \in \mcA$ we write $v_{i}= m_{i}u_{i}$ for $u_{i}\in \tT$ and apply the previous result to a column vector $\bfv'$ of larger dimension
$\sum_{i} m_i$ and repeated entries $u_{i}$, and then project each summand
of the decomposition of $\bfv'$
onto a vector of dimension $n$, which will be satisfactory.

(ii) Apply (i) to each column and use the linearity of determinant with respect to each column.
 \end{proof}

We shall need the following combinatorial results.
\begin{lemma}\label{lem-combin0}
Let $A$ be an $n\times n$ matrix with nonnegative integer entries, where for each column the sum of the entries  is  equal to $2$.
Then $\Per A$   is even.
\end{lemma}
\begin{proof}
   In classical algebra the sum of the columns is~$0$ modulo 2, and so the classical determinant also is~$0$ modulo 2, implying the permanent is even.
\end{proof}
\begin{lemma}\label{lem-combin1}
Let $A= (a_{ij})$ be an $n \times n$ matrix with nonnegative integer entries, where for each column the sum of the entries either is  equal to $2$, or
is at least  $L \in \Net$.
Then, either $\Per A\geq L$ or $\Per A$ is even.
\end{lemma}
\begin{proof}
     We call a column {\it $L$-ample} if   its sum
is at least equal to $L $. Let $\mathcal{S}_{m,L}$ be the set of $n \times n$  matrices with $m$ columns that are $L$-ample and $n-m$ columns   for which  the sum of the entries is   $2$. In other words, $A\in \mathcal{S}_{m,L}.$ We proceed by  double induction, on   $m$   and on $L.$  The case $m=0$ follows from \Cref{lem-combin0}. The case $L=0$ or $1$ is obvious. So we assume that $L\ge 2$ and $m\ge 1$.
Thus we may assume that the first column is $L$-ample.

     Permuting rows, we may assume that all the entries on the main diagonal are nonzero.  If $a_{1,1} = L$ then we are done, so we may assume that   $a_{1,1}<L,$ and thus $a_{i,1} \ge 1$ for some $i\ne 1.$
     Let $A'$ be the matrix where we replace the first column $\bfv_1$ by  the vector with $1$ in the $1$ and $i$ positions and $0$ elsewhere, and let $A''$ be the matrix where we replace   $\bfv_1$ by  the vector with $a_{1,1}-1$ in     position $1$, $a_{i,1}-1$ in the $i$ position,  and  $a_{i',1}$ elsewhere.
So $A' \in \mathcal{S}_{m-1,L} $,
and $A''\in \mathcal{S}_{m,L-2}$.

By induction on $m$, $\Per A'$ is even or $\ge L$ and is clearly nonzero, and by induction on $L$, $\Per A''$ either is even or is $\ge L-2.$
Hence $\Per A = \Per A'+\Per A''$ either  is even or is $\ge 2 + (L-2)   = L.$
\end{proof}

 \begin{example}
     $\left(\begin{matrix} \one & \one & \zero \\
\one & \zero & \one   \\ \one & \one
 &\one
\end{matrix}\right)$
  has permanent 3.
     $\left(\begin{matrix} \one & \one & \zero \\
\one  & \one & L \\ \zero & \zero
 &\one
\end{matrix}\right)$ has permanent $2$.
 \end{example}

 \begin{theoalpha}[Condition \A{1} via metatangibility]\label{A1partB} Condition \A{1} holds  if   $(\mathcal A,\mcA_0)$ is a uniquely negated metatangible pair.
 \end{theoalpha}
 \begin{proof}
 By \Cref{A1alt}, Condition \A{1}
 is equivalent to \Cref{A1alt}\ref{A1alt2}.
 So let $A$ be a $n\times n$ matrix with entries in
 $\mcA$ such that $A\one_{(n)}\nablaminus \zero$.
 Then, the transpose of $A$ has column sums in~$\mcA_0$.
Applying  \Cref{A1partA} to the transpose of $A$, we deduce that the determinant of $A$ is a  tangible linear combination of determinants of  normalized satisfactory matrices.
 It   thus suffices to show that the determinant of
 any normalized satisfactory matrix is in $\mcA_0$.


 Let $A$ be a $n\times n$ normalized satisfactory matrix.
 Assume first that $\mcA$ is of the second kind.
 Then, all entries of $A$ are $\zero$, $\one$, $(-)\one$ or $\one^\circ$ and so all tracks are equal to $\zero$, $\one$, $(-)\one$ or $\one^\circ$. Since the sum in each column is in $\mcA_0,$ then by substituting in the matrix $A$, $\zero$ and $\one^\circ$ by $0$, $\one$ by $1$ and $(-)\one$ by $-1$, we obtain a matrix $B$
 in classical algebra  whose column sums are all equal to~$0$.
 So the classical determinant of $B$ is~$0$, implying that the number $m$ of tracks equal to $1$ is the same as the number of tracks equal to $-1$ and the others are $0$.
 Then $A$ has $m$ of tracks equal to $\one$, the same for $(-)\one$, and the other tracks are either $\one^\circ$ or $\zero$.  So
   $|A|= m' e\in  \mcA_0$ with $m'\ge m$.

 Assume now that $(\mcA,\mcA_0)$ is of the first kind, i.e., $(-)\one = \one,$ and let $k$ be the $\mcA_0$-characteristic of~$\mcA$.
 Then the entries of $A$ are such that $a_{ij}= b_{ij} \one$ where $B=(b_{ij})$ is a matrix with nonnegative entries
 satisfying the assumptions of \Cref{lem-combin1}
 with $L=2k+1$.
Since $|A|=(\perm B) \one$ and
 $\perm B\geq 2k+1$ or is even.
By \Cref{lem-combin1},
 using \Cref{A0char}, we get that $|A|\in \mcA_0$.
     \end{proof}

 \subsubsection{Condition \A{1} via a modulus}$ $

What can we say about Condition \A{1} without unique negation?

\begin{lemma} Suppose $(\mcA,\mcA_0)$ satisfies \PropertyN.
    \label{moduu} $ $\begin{enumerate}

  \item \label{moduu-1} If $a +a_i \in \mcA_0$ for $ a,a_i\in \tT,$  $1\le i \le m$, then $a+ \sum _{i=1}^m a_i = a  +ma^\dag$, for any $m\ge 1$ (independently of the choice of $\one^\dag$).

 \item \label{moduu-2}     $a'e+a^\dag  = a'e^+$ for each $a,a'\in\tT$ such that $a'+a \in \mcA_0$.

   \item \label{moduu-3} If in (i), $(\mcA,\mcA_0)$ is of the first kind, then $ a+\sum _{i=1}^m a_i =  (m+1) a$ for any $m\ge 1.$
    \end{enumerate}
\end{lemma}
 \begin{proof}
\ref{moduu-1} For $m=1$ this is by assumption. For $m=2,$   $a + a_1 +a_2 = (a+a^\dag)+a_2 = (a+a_2) +a^\dag = a+ a^\dag+a^\dag.$ By induction on $m$; assume $a+ \sum _{i=1}^{m-1}a_i= a + (m-1) a^\dag,$
then $$a+ \sum _{i=1}^{m} a_i =  a + (m-1)a^\dag + a_m = (a+a_m) +  (m-1)a^\dag = (a+a^\dag ) +  (m-1)a^\dag  = a  +ma^\dag .$$


\ref{moduu-2} $a'+a \in \mcA_0$ and \PropertyN\ imply $a'+a=a'e=ae$, so
$a'e+a^\dag = a'+{a'}^\dag  +a^\dag  = a'+a'e ^\dag  = a'+a'e = a'e^+.$

\ref{moduu-3} Special case of \ref{moduu-1}.
 \end{proof}

 \begin{example}
     Suppose $(\mcA,\mcA_0)$ satisfies \PropertyN\ and $\one + a_1 = \one + a_2 = e.$ Then the rows of the matrix $\left(\begin{matrix} \one & a_1   \\ a_2 & \one \end{matrix}\right)$ are dependent, but $A$ could be  nonsingular. For example, we could take the semiring $\mcA =\Net [\lambda_1, \lambda_2, \lambda_3]$ and $\mcA_0 = \mcA(\lambda_1+1) +\mcA(\lambda_2 +1)+\mcA(\lambda_3+1).$ Then we could take $\one^\dag =\lambda_1, $ $a_1 = \lambda_2,$ and $a_2=\lambda_3$.
 \end{example}

 To avoid such an example, we need   N-transitivity, defined in \Cref{Ntran}.

\begin{lemma}\label{dom2} Suppose $(\mcA,\mcA_0)$ satisfies \PropertyN\ and N-transitivity and let  $a_i\in\tT$ for all $i\geq 1$.
$ $ \begin{enumerate}
  \item  \label{dom2-2} Assume each $a_i$ is a quasi-negative of $\one^\dag.$ Then $a_1    \dots a_{m} +\one ^\dag=e,$   for all $m\in \Net.$

    \item \label{dom2-1} Assume each $a_i$ is a quasi-negative of $\one.$ Then
  \begin{enumerate}
     \item  \label{dom2-1a} $a_1    \dots a_{2m} +\one^\dag =e,$ for all $m\in \Net$;
   \item  \label{dom2-1b} $a_1    \dots a_{2m-1} +\one =e,$  for all $m\in \Net.$
      \end{enumerate}
\end{enumerate}
\end{lemma}
\begin{proof}
  \ref{dom2-2} By \PropertyN, $a_i+\one^\dag=e$ for all $i\geq 1$,
hence the property is true for $m=1$.
This is also true for $m=2$ since
    $ a_1a_2+ a_2^\dag  = (a_1+ \one^\dag)a_2 \in \mcA_0,$
$a_2^\dag+a_2\in \mcA_0$ and $a_2+\one^\dag\in \mcA_0$ with
$a_1a_2,a_2^\dag, a_2,\one^\dag\in \tT$, imply
$ a_1a_2+\one^\dag\in \mcA_0$ by N-transitivity, and then
$ a_1a_2+\one^\dag=e$ by \PropertyN.
Then, by induction, if the property holds for $m$, applying the property
for $m=2$ to the tangible elements $a_1\cdots a_m$ and $a_{m+1}$,
we deduce the property for $m+1$.

  \ref{dom2-1} By  \PropertyN, we have $a_i+\one = e$ for all $i\geq 1$.

\ref{dom2-1a} Assume first $m=1.$
We have $a_1+a_1a_2 = a_1(\one + a_2)\in \mcA_0$, and since
$a_1+\one \in  \mcA_0$, and $\one+\one^\dag\in \mcA_0$,
with $a_1a_2,a_1,\one\in \tT$, we deduce that $a_1a_2 + \one^\dag \in \mcA_0$
by N-transitivity. Then,  $a_1a_2 + \one^\dag=\one+\one^\dag=e$, by \PropertyN.

Applying this result to any pair of elements $a_{2i+1},a_{2i+2}$,
we deduce that $a_{2i+1}a_{2i+2}$ is a quasi-negative of $\one^\dag$ for all $i$.
We then obtain the result by applying \ref{dom2-2}.

\ref{dom2-1b} This is true for $m=1$.
Now assume $m\geq 2$. By \ref{dom2-1a}, we have
$a_1    \dots a_{2m-2 } +\one^\dag=e\in \mcA_0$.
Then, $a_1    \dots a_{2m-1 } + a_{2m -1}\one^\dag  =a_{2m-1}(a_1    \dots a_{2m-2 } + \one^\dag)\in \mcA_0$.
We also have $a_{2m -1}\one^\dag+\one^\dag\in\mcA_0$ and $\one^\dag+\one \in \mcA_0$, with  $a_1    \dots a_{2m-1 }, a_{2m -1}\one^\dag, \one^\dag, \one \in \tT$.
By N-transitivity we deduce that $a_1    \dots a_{2m-1 }+\one \in \mcA_0$,
and by Property N, $a_1    \dots a_{2m-1 }+\one=e$.
\end{proof}

At times we can   use  a modulus to compensate for lacking unique negation.

\begin{definition}  Suppose the  semiring pair $(\mcA,\mcA_0)$  is equipped with a modulus $\mu$.
\begin{enumerate}
\item      An element $b\in \mcA$ \textbf{dominates} $b'$ if     $\mu(b) \ge \mu(b').$
\item      An   element $b$ is  \textbf{dominant} in a set $S$  if     $\mu(b) \ge \mu(b')$ for all $b'\in S.$

\item (See \cite[\S3.2]{IR}.)
We say that a
track $a _\pi$ of a  matrix  $A$ is
\textbf{dominant} if
 $\mu(a_\pi)
\ge \mu(a_\sg)$
for every permutation $\sg$.

A matrix $A = (a_{i,j})$ has a \textbf{dominant diagonal} if the track
$a_{\id}= a_{1,1}\cdots a_{n,n}$  is dominant.
\end{enumerate}
 \end{definition}

 This relates to the optimal assignment question, as
explained in \cite{AGG2}.

\begin{lemma}
\label{dom9}
If a sequence  $\{ a_1, \dots, a_m\} $ has a single dominant element $a_j$, then $ \sum _{i=1}^m a_i = a_j.$
\end{lemma}
\begin{proof}
 $\mu(a_j + a_i) = \mu( a_j)$ for each $i\ne j,$ since otherwise $a_i$ would also be dominant, implying $a_i+a_j = a_j$; but then we can remove each $a_i$ one by one.
\end{proof}
\begin{rem} $ $
When only one track $\pi$ is dominant, $\mu({\Det{A}}) = \mu((-)^\pi a_\pi)$, and if all the entries of $A$ are not in $\mcA_0$, then $A$  is nonsingular.
 \end{rem}

In order for the theory to have substance, we need some condition on $\mu,$ in view of Remark~\ref{nont}(iv).

\begin{lemma}\label{l1} Assume that $(\mathcal A,\mcA_0)$ is an \admissible, $e$-final   (mon-)pair with \PropertyN, having a modulus $\mu$ satisfying the hypotheses  for $a_i \in \tT$:
 \begin{enumerate}
     \item[(H1)]  $\mu(a_1) =\mu(a_2)$ implies either $a_1+a_2 \in \mcA_0$ or $a_1+a_2^\dag \in \mcA_0.$
     \item[(H2)] $\sum_{i=1}^t a_i \in \mcA_0$ implies $a_i+a_j\in \mcA_0$ for some $i\ne j.$
 \end{enumerate}
  Then,  $\mu$ also satisfies:
    \begin{enumerate}  \item\label{l1-0}  If $b=\sum_{i} a_i $ with $a_i\in \mcA$ then $b = \sum_{i} \{a_i : \mu(a_i) = \mu (b)\}. $

    \item\label{l1-1}   $\mu(b) = \mu(a)$, $b \in \mcA$ and $a\in \tT$ implies  $b+a=ae \in \mcA_0$ or $b+a^\dag=ae \in \mcA_0.$ Hence $b+ae = ae.$
    \item\label{l1-3} $\mu(b_1) = \mu(b_2)$ for $b_1,b_2 \in \mcA$ implies either $b_1+b_2 =a e\in \mcA_0$ or $b_1+b_2^\dag =a e\in \mcA_0,$ for some $a\in \tTz$.


      \item\label{l1-4a}     $ \mcA_0=\tT e$.

    \item\label{l1-4}   If   $b_1 = ae \in \mcA_0$ and $b_2\in\mcA$ with $\mu(b_2)\le \mu(b_1)$ then $b_1+b_2=b_1 = ae\in \mcA_0$.

\item \label{l1-5}If $b=\sum_{i} b_i\in \mcA_0$ with $b_i\in \mcA$, then there exists a subset $B$ of one or two dominant terms $b_i$, such that $b=\sum_{b'\in B} b'\in \mcA_0$.
\item \label{l1-6} If $b_j=\sum_{i} a_{ij}\notin \mcA_0$ with $a_{ij}\in \tT$, for $j=1,2$,  and $b=b_1+b_2\in \mcA_0$, then there exists $i_j$, $j=1,2$, such that $b=a_{i_1,1}+a_{j_2,2}$ and $\mu(a_{i_j,j})=\mu(b_j)=\mu(b)$.
    \end{enumerate}
\end{lemma}
\begin{proof}
\ref{l1-0} All the non-dominant tangible elements $a_i$ get absorbed into the dominant $a_j$, cf.~\Cref{mod1}(ii).

   \ref{l1-1}
Since $\mu(a)\neq \zero$ for $a\in \tT$, we have $b\neq \zero$.
Then since $\mcA$ is \admissible, we can write $b = \sum_i a_i $ for $a_i\in\tT$.
For each dominating $a_i$, we have $\mu(a_i)=\mu(b) = \mu(a),$ so
   by hypothesis,  $a_i+a \in \mcA_0$ or $a_i+a^\dag \in \mcA_0.$ By \PropertyN\ this implies $a_i+a =ae$ or $a_i+a^\dag=ae.$
Hence each $a_i$ can be replaced by $a$   or $a^\dag$.
   But  then we conclude by $e$-finality.
   Then by \Cref{finalplus}, we get that $b+ae=a e$.

    \ref{l1-3} The case $\mu(b_1) = \mu(b_2)=\mu(\zero)$ is trivial
since this implies $b_1=b_2=\zero$.
If $b_2\in \tT$, this follows from \ref{l1-1} with $a=b_2$.
Otherwise, $b_2\in\mcA\setminus\{\zero\}$, then
writing  $b_2 = \sum a_i $ for $a_i\in\tT,$ such that $\mu(a_i)=\mu(b_2)$, by \ref{l1-0},
and applying \ref{l1-1} to $b_1$ and $a_1$,
we get that $b_1+a_1=a_1e$ or $b_1+a_1^\dag=a_1e$.
For instance, assume  that  $b_1+a_1=a_1e$, then applying \ref{l1-1}
to $a_1 e$ and successively each $a_i$,
we deduce $b_1+b_2=a_1 e$.
The same holds if $b_1+a_1^\dag=a_1e$, by applying \ref{l1-1}
to $a_1 e$ and successively each $a_i^\dag$.
This shows \ref{l1-3} for any $b_2$ with $a=a_1$.


  \ref{l1-4a} Clearly   $ \mcA_0\supseteq \tT e$.  For inclusion, if   $b  \in \mcA_0$, write $b = \sum_{i=1}^t a_i$ with $\mu(a_i) = \mu(b)$ for each $i$.  By (H2) we may assume that $a_1 +a_2 \in \mcA_0,$ so $a_1+a_2 = a_1e.$ By the same arguments as for \ref{l1-3}, we get $a_1e + a_3 = a_1e, $ and continuing, we get
  $b = (a_1e + a_3) + \dots + a_t = a_1 e.$

    \ref{l1-4} Clearly $\mu(b_1) = \mu(a) .$ If $\mu(b_2)<\mu(b_1)$, this comes from the definition of a modulus since then $b_1+b_2=b_1$.
Otherwise, this comes from \ref{l1-1}.





\ref{l1-5} Using \ref{l1-0}, we  write $b$ as the sum of the dominant elements $b_i$. Then, decomposing each $b_i$ as sums of tangible dominant elements, using (H2),
there exist $a,a'\in \tT$ such that $a+a'\in \mcA_0$,
where $a$ is a summand of some $b_i$, $a'$ is a summand of some $b_j$,
and if $i=j$, then $a,a'$ are not the same summands of~$b_i$.
By \ref{l1-4} and \Cref{l1}, \Cref{l1-4}, we deduce that $b=ae$. Also, if $i\neq j$ then $a+a'=b_i+b_j$, and if $i=j$, then $a+a'=b_i\in \mcA_0$.

\ref{l1-6} This follows from \ref{l1-5} applied to $b$ and the decomposition in the sum of all $a_{ij}$, since there are necessarily two dominant terms because $a_{ij}\notin \mcA_0$ for all $i,j$, and the two dominant terms $a_{ij}$ obtained from \ref{l1-5} cannot both be a summand of $b_1$, since otherwise $b_1\in \mcA_0$, and the same for $b_2$.
\end{proof}
\begin{remark}
    $ $
    \begin{enumerate}
        \item
Without (H2), we may have that all $a_k+a_l\not\in \mcA_0$ so $a_k+a_l^\dag\in \mcA_0$ and one can only deduce   $\mcA= \tT\cup  \tT e \cup \{\sum  a_i  \mid a_i\in \tT, a_i +a_j^\dag = a_i e, \text{ for all}\; i\ne j\}.$

 For instance one may consider $\mcA=\Z\cup\{e\}$, $\tT=\{-1,1\}$, with $\one =1,$ and $m\in \mcA_0$ if either $m=e$ or $m\in\Z$ is multiple of some prime number $p$.
Assume $m(-)n=\one(-)\one$ for $m,n>0$ and take $\mu$ equal to $\one$ on $\mcA\setminus\{0\}$.
Then $p\in \mcA_0$ and $p(-)\one=e$ but $p+\one\not\in\mcA_0$.
Note however that this set $\mcA$ does not seem to satisfy the N-transitivity.

        \item    If $(\mcA,\mcA_0)$ is null-bipotent then (H1) implies (H2).
 Indeed,
    write $b= \sum_{i=1}^{t-1} a_i.$ Then $b\in \mcA_0$ and we are done by induction, unless $b= a_i$ for some $i.$ But then $a_i+a_t\in \mcA_0.$
    \end{enumerate}

\end{remark}
\begin{proposition}\label{A1part} Condition \A{1} holds 
if   $(\mathcal A,\mcA_0)$ is an \admissible, $e$-final  , N-transitive pair with \PropertyN, having a modulus $\mu$ and satisfying
 the hypotheses of \Cref{l1}, where
 $\nabla=\nabladag$.
 \end{proposition}
\begin{proof} (Paralleling
\cite[Theorem~8.11]{Row21}.) 
 Localizing $\tT,$ we may assume that $\tT$ is a group.
Let us show Property \ref{nab-A1alt2} of \Cref{nab-A1alt}.
 It is enough to consider a square matrix $A$ such that
 $A \one_{(n)}\, \nabla \,\zero$ and prove that
${|A|_+} \, \nabla \,{|A|_-}$.


If all tracks have modulus zero, then the modulus of $|A|_{\pm}$ is zero, which implies that  $|A|_{\pm}=\zero$ and the assertion is immediate. Assume now that at least one track has a nonzero modulus.
By \Cref{dom9},  ${|A|_++|A|_-^\dag}$ is the sum of all dominant even tracks   and $\one^\dag$ times all dominant odd tracks.
Then, if some dominant track belongs to $\mcA_0$, we have ${|A|_++|A|_-^\dag} \in \mcA_0$.

Thus assume that all dominant tracks are not in $\mcA_0$.
Permute the rows of $A$ so that the diagonal is a dominant track.
Thus $a_{i,i}\not \in \mcA_0$,
for all $i$. Since $\sum_j a_{i,j}\in \mcA_0$,  there is
at least one dominant entry $a_{i,j}$ in this sum with $j\neq i$, so outside the diagonal, and possibly another dominant
entry $a_{i,k}$ such that the sum of these (one or two) dominant entries is in $\mcA_0$ and is equal to $\sum_j a_{i,j}$.
Consider the graph $\mathcal G$ obtained by taking $[n]$ as the set of nodes and
such that $(i,j)$ is an arc if $a_{i,j}$ is one of the entries
involved in the sum of row $i$ as above.
For each $i$,  since $a_{i,j}$ is dominant we have
necessarily $\mu(a_{i,j})\geq \mu(a_{i,i})$.
Since for all $i$, there is at least one arc $(i,j)$ with $j\neq i$,
there exists at least one non-trivial cycle
$c=(i_1,\ldots, i_k,i_{k+1}=i_1)$ in the graph.
Consider the permutation $\sigma$
such that $\sigma(i_{l+1})=i_l$ for all $l=1,\ldots, k$ and $\sigma(i)=i$
for $i\not\in c$.
Since  $\mu(a_{i,j})\geq \mu(a_{i,i})$, for all arcs $(i,j)$,
we have $\mu(a_\sigma)\geq\mu(a_{\id})$, for the track $a_\sigma$ associated to $\sigma$, hence the track $a_\sigma$
is dominant and $\mu(a_{\sigma(j),j})=\mu(a_{j,j})$ for all $j$.
Applying the permutation $\sigma$ to the rows of $A$,
we obtain a new matrix $A'$ which satisfies the same properties
as $A$ (sum of rows in $\mcA_0$, dominant diagonal track,  and thus all
diagonal entries not in $\mcA_0$), however
we know that, for all $i$ in the cycle $c$, the loop $(i,i)$ is in the new
graph $\mathcal G'$ constructed using the matrix $A'$.
Since the loops $(i,i)$ with $i$ outside~$c$ are unchanged, the number
of loops of $\mathcal G'$ is greater or equal to the number of loops of
the initial graph~$\mathcal G$, with equality if and only
$\mathcal G$ already contained all the loops $(i,i)$ with $i\in c$.

Applying such an operation successively to the matrix $A$, we obtain
a sequence $(A_p)_{p\geq 0}$ of matrices which are all obtained by a permutation of rows of the initial matrix $A$. Denoting by $k_p\leq n$, the number of loops of the graph $\mathcal G_p$ associated to $A_p$ as above, we have that $(k_p)_{p\geq 0}$ is nondecreasing.
The sequence is necessarily stationary after some step $p$:
meaning that $k_{p+1}=k_p$, hence $A_p$ is
such that there exists a cycle $c=(i_1,\ldots, i_k,i_{k+1}=i_1)$ with for all
$l=1,\ldots , k$, both $(i_l,i_l)$ and $(i_l,i_{l+1})$ are arcs in $\mathcal G_p$, that is $a_{i_l,i_l}+a_{i_l,i_{l+1}}=\sum_{j} a_{i_l,j}\in \mcA_0$
and $\mu(a_{i_l,i_l})=\mu(a_{i_l,i_{l+1}})\geq \mu(a_{i_l,j})$ for
all $j$.
Since the positive and negative determinants of $A$ and $A_p$ are either the same or exchanged, we may assume from now on that $A_p=A$ satisfies the above conditions.

Let $\sigma$ be the permutation constructed as above: $\sigma(i_{l+1})=i_l$ for all $l=1,\ldots, k$ and $\sigma(i)=i$ for $i\not\in c$.
Assume first that $k-1$ is even, as is then  the permutation $\sigma$.
Then, $|A|_+=a_{\id}+a_{\sigma}+b$ for some $b\in \mcA$, such that
$\mu(b)\leq \mu(a_{\id})=\mu(a_{\sigma})$ and $\mu(|A|_-)\leq \mu(|A|_+)$.
We also have $a_{\id}+a_{\sigma}=b_1b_2$
with $b_1\in \mcA$ and $b_2=\prod_{l=1}^k a_{i_l,i_l} +\prod_{l=1}^{k}  a_{i_l, i_{l+1}}$.
For each $l=1,\ldots, k$, $a_{i_l,i_l}$ and $a_{i_l, i_{l+1}}$ have same modulus and are not in $\mcA_0$. So,
 we can decompose them as sums of tangible elements
with same modulus. Also, by \Cref{l1}\ref{l1-6}, there exists $d_l,d'_l\in \tT$ such that
$d_l+d'_l\in \mcA_0$, $d_l$ is a summand of $a_{i_l,i_l}$ and
$d'_l$ a summand of $a_{i_l, i_{l+1}}$.
Then, one can write $b_2=b_3+b_4$ with $\mu(b_4)\leq\mu(b_3)$
and $b_3=\prod_{l=1}^k d_l +\prod_{l=1}^{k} d'_l$.
Normalizing by the $d_l$ and applying \Cref{dom2}\ref{dom2-1b}, we deduce that $b_3\in \mcA_0$.
Then, applying successively \Cref{l1}\ref{l1-4}, and the property that $\mcA_0$ is a semimodule (and $\mcA$ is a semiring), we deduce that $|A|_+\in\mcA_0$ and then that $|A|_+\nabla |A|_-$.

Assume now that $k-1$ is odd, as is then  the permutation $\sigma$.
Then, $|A|_+=a_{\id}+b_+$ and $|A|_-=a_{\sigma}+b_-$ for some $b_{\pm}\in \mcA$, such that $\mu(b_{\pm})\leq \mu(a_{\id})=\mu(a_{\sigma})$.
Applying the same arguments as above to $|A|_++|A|_-^\dag$,
instead of $|A|_+$, we need to show that
$b_3=\prod_{l=1}^k d_l +\one^\dag\prod_{l=1}^{k} d'_l\in \mcA_0$.
Normalizing by the $d_l$ and applying \Cref{dom2}\ref{dom2-1a}, we deduce that $b_3\in \mcA_0$, and then $|A|_++|A|_-^\dag\in\mcA_0$.
Similarly, applying the same arguments to $|A|_+^\dag+|A|_-$,
we need to show that
$b'_3=\one^\dag\prod_{l=1}^k d_l +\prod_{l=1}^{k} d'_l\in \mcA_0$.
Normalizing by the $d'_l$ and applying \Cref{dom2}\ref{dom2-1a}, we deduce that $b'_3\in \mcA_0$, and then $|A|_+^\dag+|A|_-\in\mcA_0$.
\end{proof}

\subsection{Cramer's rule} \label{Cr} $ $

In {\cite[Theorem 4.18]{AGG2}},  Cramer's rule was proved for a semiring with a balance relation, under several assumptions.
\Cref{theoalphaJ10bis} shows that the conclusions of this theorem follow from Condition \A{1} and unique negation. They are indeed
equivalent to these conditions by \Cref{A1alt}.

\begin{theoalpha}[Cramer's rule, Part 1]\label{theoalphaJ10}
 Let $(\mathcal A,
\mathcal A_0,(-))$  be a $\tT$-semiring   \triple.
Then, for any vector $\mathbf v$,
the vector $\mathbf w  = {\adj A\mathbf v }$ satisfies $|A|\mathbf v \, \nablaminus \, A \mathbf w.$
In particular, for $|A|$   invertible in
$\tT$, then
 $\mathbf x:= \frac 1 {|A|}{\adj A\mathbf v }$ satisfies
$\mathbf v \, \nablaminus \, A \mathbf x.$
\end{theoalpha}

\begin{proof}
    ${|A|}I\, \nablaminus \, A \,{\adj
A}$  in any $\tT$-semiring triple, implying ${|A|}\mathbf v  \,\nablaminus \, A ({\adj A} \mathbf v) =
A\mathbf w.$
\end{proof}

\begin{theoalpha}[Cramer's rule, Part 2]\label{theoalphaJ10bis}
 Let $(\mathcal A,\mathcal A_0,(-))$  be a $\tT$-semiring \triple\ satisfying Condition \A{1}.
Then, for any vectors $\mathbf v\in \mcA^{(n)}$ and $\mathbf w\in \tTz^{(n)}$ satisfying $A \mathbf w\nablaminus \, \mathbf v$,
we have
$|A|\mathbf w\nablaminus  {\adj A\mathbf v }$.
In particular, for $|A|$~invertible in
$\tT$, then
 $\mathbf x:= \frac 1 {|A|}{\adj A\mathbf v }$ satisfies
$\mathbf w \, \nablaminus \, \mathbf x.$

Moreover, if $\mcA$ has unique negation, and $\adj A\mathbf v\in \tTz^{(n)}$ then $ \mathbf x$ is the unique solution $w$ of $A \mathbf w\nablaminus \, \mathbf v$.
\end{theoalpha}
\begin{proof} The first assertion is the equivalence between
Condition \A{1} and  \Cref{A1alt}{\rm \ref{A1alt3}}.
The two other assertions are clear.
\end{proof}

\begin{rem}
Taking $\mathbf v=(\zero)$ shows
$\Det A \in \mcA_0,$ yielding Condition \A{1}.
\end{rem}
As in \cite[Corollary~4.19]{AGG2}, we can obtain the true Cramer's rule as a consequence.

 \begin{corollary}\label{forlater} Let $(\mathcal A,\mathcal A_0,(-))$  be a $\tT$-semiring \triple\ satisfying Condition \A{1}
 and having unique negation.
 Suppose that $A$ is an $n \times (n+1)$ matrix over $\mcA$,
 and let $A_{(k)}$ denote the matrix
obtained by deleting the $k$ column of $A$.
Let $\mathbf v  = (v_1, \dots, v_{n+1})$ where $v_k = (-)^{n-k+1}\Det{A_{(k)}}.$
Assume that $\mathbf v\in \tTz^{(n+1)}$.
 Then any tangible solution of $Ax \in \mcA_0^{(n)}$ is a tangible multiple of $\mathbf v $. \end{corollary}

Using  the  surpassing relation $\preceqzero$, we can   state a   result which strengthens \Cref{theoalphaJ10}.

\begin{theoalpha}[improving \Cref{theoalphaJ10}]\label{theoalphaJ2}  Let $(\mathcal A,
\mathcal A_0,(-))$  be a triple,  and take the  surpassing relation ~$\preceqzero$.   For any vector $\mathbf v$,
the vector $\mathbf x  = {\adj A\mathbf v }$ satisfies $|A|\mathbf v
\preceqzero A \mathbf x.$

In particular, for $|A|$ tangible invertible
and ${\adj A\mathbf v }$ tangible,
 $\mathbf x:= \frac 1 {|A|}{\adj A\mathbf v }$ satisfies
$\mathbf v \preceqzero A \mathbf x$.
Furthermore, if $A$ is $\mcA_0$-cancellative in the sense that $A \mathbf y\in \mcV_0$ implies $ \mathbf y\in \mcV_0$,
and if $\mathbf v \preceqzero A \mathbf x'$, then $ \mathbf x\, \nablaminus \,  \mathbf x'$.
 \end{theoalpha}
 \begin{proof}
     An immediate application of $|A| I\preceqzero A \adj A,$ which follows at once from \cite{zeilberger85}. For the last assertion,
     $A( \mathbf x(-) \mathbf x') = A \mathbf x (-) A \mathbf x' \succeq \zero,$
     so $ \mathbf x(-) \mathbf x'\in \mathcal V_0,$ implying $ \mathbf x\, \nablaminus \,  \mathbf x'.$
 \end{proof}

 Cramer's rule often provides a solution  to Condition \A{1},
  since   $\Det A  \in \mcA_0$ implies $\Det A \mathbf{v}\in \mathbf{V}_0$ for  arbitrary~$\mathbf{v} \in \mathbf{V}$.
 But unfortunately, the vector $\mathbf w$ in Theorem~\ref{theoalphaJ10} might not be
tangible; to overcome this obstacle
  requires more work, to be done below in
Proposition~\ref{theoalphaN}.

Considerably stronger assumptions are required to make $\mathbf x$ tangible,
 as in \cite{IzhakianRowen2008Matrices2}.

\subsection{Translation of terminology from \cite{AGG2}}  $ $

 Let us now discuss the assumptions used in  \cite[Theorem~4.18]{AGG2}, to obtain the same conclusions as in \Cref{theoalphaJ10} and \Cref{theoalphaJ10bis}.
To see this, we need to correlate the concepts of
 \cite{AGG2} with the terminology of this paper.

We assume throughout that $(\mathcal A,\mathcal A_0,(-))$ is a $\tT$-semiring triple, and that $\nabla = \nablaminus$. (Lacking this, the terminology of  \cite{AGG2} would not be meaningful.)
The ``thin set'' $S^\vee$
of~\cite{AGG2} is replaced here by ~$\tTz$.

The main properties in  \cite{AGG2} are related to balancing.

 \begin{definition}[\cite{AGR2}]\label{tangbal}   $\mcA$ is  \textbf{tangibly balanced} if for all  $b_1,b_2\in \mcA$
 such that  $b_1\nabla b_2$, there exists  $a\in \tTz$ such that
 $b_1\nabla a$ and $a\nabla b_2$.
   $\mcA$ has  \textbf{tangible balance elimination} if $b_1\nabla a$ and $a\nabla b_2$ imply $b_1\nabla b_2$, for all $b_1,b_2\in \mcA$ and $a\in \tTz$.
    \end{definition}

\subsubsection{Weak balance elimination}$ $

 In the following definition we use the terminolgy of \cite{AGG2}.
\begin{definition}\label{bale}$ $
A semiring pair $(\mathcal A,\mathcal A_0)$ allows \textbf{weak balance elimination} with respect to a balance relation~$\nabla$ if  the relations    $(\sum _j b_{j} a_j) \nabla
d$ and $
a_j \nabla c_j$ for $a_j \in \tT$
and $c_j,b_j,d\in \mcA$ together imply $(\sum _j b_{j} c_j) \nabla d$.
 \end{definition}

\begin{rem}\label{ord3}$ $
  \begin{enumerate}
\item
 \cite[Property 4.1]{AGG2} translates to unique negation.
\item \cite[Property 4.2]{AGG2} is in our running hypotheses that $\tT_0$ is a monoid.
\item \cite[Property 4.3]{AGG2} (``weak transitivity of systems of balances'') is a generalization of \Cref{bale} in which there several balance equations and a multiplicative factor in front of the $a_j$.
This is equivalent to \Cref{bale}, since $\tT_0$ is a monoid in this paper.
\item  \cite[Definition 4.4]{AGG2} defines ``weak balance elimination'' as
\cite[Property 4.3]{AGG2} together with \cite[Property 4.2]{AGG2}, which  is equivalent to
\Cref{bale} since \cite[Property 4.2]{AGG2}   always holds in this paper.
\item \cite[Definition 4.4]{AGG2}  also defines ``strong balance elimination'' which is weak balance elimination together with unique negation.
\item \cite[Property 4.5]{AGG2} (``weak transitivity of balances'') is equivalent to tangible balance elimination in \Cref{tangbal}, that we also used in \cite{AGR2}.
\item \cite[Property 4.6]{AGG2}  (``weak transitivity of scalar balances'') is the restriction of \Cref{bale} to the case of one summand.
  \item \cite[Property 4.7]{AGG2} translates into $\mcA=\mcA_0\cup \tT$.
\item \cite[Property 4.8]{AGG2} follows from the property that $\mcA$ is admissible.
\item \cite[Lemma 4.11]{AGG2} states that \cite[Properties 4.5, 4.2, 4.8]{AGG2} imply weak balance elimination. Under the assumptions of this paper, this translates to ``tangible balance elimination implies weak balance elimination'' for an admissible pair.
 \end{enumerate}

\end{rem}
The statement of \cite[Theorem~4.18]{AGG2} uses  weak balance elimination as an assumption
  to obtain the first assertion of \Cref{theoalphaJ10bis}, and strong balance elimination
  as an assumption to obtain the last assertion of \Cref{theoalphaJ10bis}.
  However, the proof was the same as in \cite[Theorem~6.4]{AGG1}, which also
  uses \cite[Property 4.7]{AGG2}, that is that $\mcA=\mcA_0\cup \tT$. So   \cite[Theorem~4.18]{AGG2} implies the following:

  \begin{theorem}[Corollary of \protect{\cite[Theorem~4.18 and Lemma 4.11]{AGG2}}]\label{theocramerGut}
  Condition \A{1} is true when $(\mathcal A,\mathcal A_0)$ is admissible, has tangible balance elimination and
  $\mcA=\mcA_0\cup \tT$.
  \end{theorem}

 Strong balance
 elimination provides uniqueness in Cramer's rule and so the conclusion of  \Cref{forlater} as in \cite[Corollary~4.19]{AGG2}.

 We now give other proofs  of \Cref{theocramerGut} without the condition   $\mcA=\mcA_0\cup \tT$.



Metatangible pairs are particularly well-behaved in these terms.

\begin{proposition}\label{lemdep1}
Tangible balance
elimination
holds if $(\mathcal A,\mathcal A_0)$ is N-transitive metatangible.
\end{proposition}
\begin{proof}  We need to
show that if $a \in \tTz$ and $a \nabla b_i$, $i=1,2$, then $b_1\nabla b_2$.

We know that either $b_i = a_i^\circ$ or $b_i=m_i a_i$ with $m_i\in \Net$ (with $b_i = a_i^\circ$ if $m_i=2$ and $\mcA$ being of the first kind if $m_i\geq 3$). For $m_1=m_2=1$ the assertion follows from \Cref{balNpropN}, since metatangible pairs have property N.

If $b_1$ and $b_2$ are both in $\mcA_0$, so in particular if $m_1=m_2=2$, then we are done.

For $m_1 =1$ and $b_2 =a_2^\circ$ we know that $a_1(-)a \in \mcA_0,$ so $a_1(-)a = ae.$ We also have $a_2^\circ(-)a\in \mcA_0$. If $a(-)a_2 \in \mcA_0,$ then $a (-)a_2 =a^\circ= a_2^\circ$,   so $a_1(-)a_2^\circ = a_1(-)a^\circ = ae +a =a (-)a_2^\circ $ by \Cref{moduu-2} of \Cref{moduu}, so
we are done. Hence we may assume $a (-)a_2 \in \tT,$ so $a(-)a_2^\circ = (a (-)a_2)+ a_2 = a_2^\circ,$ by~Property~N. Thus $a_1(-)a_2^\circ =(a_1(-)a)(-)a_2^\circ = a ^\circ (-)a_2^\circ \in \mcA_0.$

The same argument holds if $m_1=2$ and $m_2 =1.$
These calculations yield the result when  $(\mathcal A,\mathcal
A_0)$ is of the second kind, or both $m_1,m_2\leq 2$, or both $b_1,b_2$ are in $\mcA_0$ so we may assume that $(\mathcal A,\mathcal A_0)$ is of the first kind and that some $m_i>2$ and some $b_i\not\in\mcA_0$.

 Take $m_i'$ maximal such that $m_i'a_i+a \in \tT.$ Then $m'_i<m_i$ and
 $(m_i'+1)a_i + a \in \mcA_0,$ so 
\begin{equation}
    \label{e0}
    (m_i'+1)a_i + a = 2 a_i=2 m'_ia_i+2 a ,
\end{equation}by Property~N.

Then, if $m_i\geq 2$, we can write
\begin{equation}
    \label{e2}b_i =  (b_i-2a_i ) +(m_i'+1)a_i +a = (m_i+m'_i-1) a_i+a. 
\end{equation}
Assume that $m_1>2$. Using  \eqref{e2}, we obtain
\begin{equation}
    \label{e1} b_1+b_2= (m_1+m'_1-1) a_1+(a+b_2).
\end{equation}

Assume that $b_1\notin \mcA_0.$
Thus, $m_1$ is odd and  \eqref{e2} implies $m_1'$ is even, so \eqref{e1} implies that
$b_1+b_2\in \mcA_0$. So we are done unless $b_1\in \mcA_0$.
In that case, $b_2\not \in\mcA_0$. If $m_2>2$, the same arguments as above show that we are done.

 The case remains where $m_1>2$ and $m_2 =1,$ and thus $m_2'=0,$ so $b_2+a=a_2+a= 2a$ by~\eqref{e0}.
Then using \eqref{e1}, $m'_1<m_1$, and  \eqref{e0} for $i=1$,  we obtain
$ b_1+b_2= (m_1+m'_1-1) a_1+(a+b_2)= (m_1+m'_1-1) a_1+2 a = (m_1-m'_1-1) a_1+ 2m'_1 a_1+2 a=(m_1-m'_1-1) a_1+ (m'_1+1) a_1+ a= b_1+a\in\mcA_0.$
\end{proof}


 \begin{lemma}[Compare with \protect{\cite[Lemma 4.11]{AGG2}}]  \label{equivbalance}
      If   $\mathcal A$ is \admissible, 
      then {weak balance elimination}
      is equivalent to tangible balance elimination.
  \end{lemma}

  \begin{proof} $(\Rightarrow)$
  Let $d,c\in \mcA$,  $a\in \tTz$, such that $d\,\nabla\, a$ and $a\,\nabla\, c$. Then, applying weak balance elimination to the equations $a\nabla d$ and $a\nabla c$ with one term, we get $d\nabla c$.

$(\Leftarrow)$ Assume tangible balance elimination.
Consider $a_j \in \tT$
and $c_j,b_j,d\in \mcA$ satisfying   $(\sum _{j=1}^t b_{j} a_j) \, \nabla \,
d$ and $
a_j \, \nabla \, c_j$. We want to prove  $(\sum _{j=1}^t b_{j} c_j) \, \nabla \,
d$.  By decomposing the $b_j$, we may assume that $b_j\in\tT$
for all $j$. We want to prove $(\sum _{j=1}^t b_{j} c_j) \, \nabla \,
d$.

We proceed by induction on $t.$
For $t=1,$ we are given $ba\, \nabla \,
d$ and $
a \, \nabla \, c$, with $a \in \tT.$ Multiplying by $b$ the second equation,
we get also $ba\, \nabla\, b c$, and $ba\in \tT$,
so by tangible balance elimination we obtain $bc\, \nabla \,
d$.

In general we have
$b_1 a_1 \, \nabla \, ((-)(\sum _{j\neq 1} b_{j} a_j) +d$)
and $b_1 a_1\, \nabla \, b_1 c_1$. Since $b_1 a_1\in \tT$, apply the case $t=1$ or tangible balance elimination
to get that $b_1 c_1 \, \nabla \, (-)((\sum _{j\neq 1} b_{j} a_j) +d)$, that is, $(\sum _{j\neq 1} b_{j} a_j) \, \nabla \, d (-)b_1c_1$. By  induction we get
 $(\sum _{j\neq 1} b_{j} c_j) \, \nabla \, d(-) b_1c_1$, which
yields $(\sum _j b_{j} c_j) \, \nabla \, d $.
  \end{proof}


 \subsubsection{Condition \A{1} using weak balance elimination}$ $





We say that a set $S$ is \textbf{pairwise balanced} with respect to a balance relation $\nabla $ if $b_i \nabla b_j$ for all $b_i,b_j\in S.$ $(\mathcal A,
\mathcal A_0,(-))$ is \textbf{multiply balanced} if any pairwise balanced set is balanced by some tangible element.
This is a generalization of the tangibly balanced property.
Moreover, this property holds in the following situation.

\begin{lemma}\label{mb1}  Pairs satisfying $\mcA=\tT \cup\mcA_0$ are multiply balanced. In particular,
 metatangible pairs of the second kind are multiply balanced.
\end{lemma}
\begin{proof}
Consider a pairwise balanced set $S$ and show it is balanced by some tangible element.
If one element of $S$ is tangible, take this one, otherwise all elements are in $\mcA_0$ so take $\zero$.
\end{proof}

See Example~\ref{metex}(i)(b) with $m=5$ for an  example of a pair of the first kind  which is not multiply balanced, since $3+4\in \mcA_0$ but $3+1\notin \mcA_0$.

\begin{theoalpha}[Condition \A{1} via  weak balance elimination]\label{theoalphaJ}
 Let $(\mathcal A,
\mathcal A_0,(-))$  be a $\tT$-semiring triple
which is multiply balanced  with respect to the balance relation $\nabla $, and
 satisfies  weak balance elimination.
 Then Condition \A{1} holds.
 \end{theoalpha}

\begin{proof}[Proof of \Cref{theoalphaJ}]
Here we improve  the proof of \cite[Theorem~6.4]{AGG1}, for otherwise more assumptions would be required.
We verify
the implication  \eqref{nulldet1}, which is equivalent to Condition \A{1} by \Cref{A1alt}.

 We prove this property by induction on $n$.
 The property trivially holds for $n=1$.
 If this property holds for $n-1$, let us consider the $n\times n$ matrix $A$
 such that $A\one_{n}\, \nabla \,\zero$.
 Let $\tilde{A}$ be the $(n-1)\times n$ submatrix of $A$
 obtained by eliminating the last row.
 For every $j<j'$ let us consider the
 $n\times n$ matrix
 $B$ obtained from $\tilde{A}$ by adding the columns $j$ and $j'$ put in the place of the $j$ column.
 Then, $B\one_{(n-1)}\, \nabla \,\zero$,
 so by the induction assumption, $|B|\, \nabla \,\zero$, which is equivalent to
$(-)^j a'_{n,j}\, \nabla\, (-)^{j'} a'_{n,j'} $,
where  $(a'_{ij})$ is the comatrix of~$A$.
Under the multiply balanced hypothesis, taking $S = \{(-)^j a'_{n,j}: 1\le j\le n\}, $
there exists $a\in \tTz$ such that
$a\nabla\, (-)^j a'_{n,j} $ for all $1\le j\le n$.
Multiplying the last balance equation
$\sum_{j=1}^{n} a_{n,j}\, \, \nabla \,\zero$
by $a$, then using weak balance elimination to replace
$a a_{n,j}$ by $(-)^j a'_{n,j}a_{n,j}$,
and using \Cref{theoalphaH}, we deduce that
$|A| = \sum _{j=1}^n (-)^{i+j} a'_{n,j}a_{n,j} \, \nabla \,\zero$.
\end{proof}

\begin{corollary}\label{tran1}  Condition \A{1}  holds for N-transitive metatangible pairs of the second kind, and for N-transitive metatangible pairs of the first kind which are multiply balanced.
\end{corollary}
\begin{proof}
    Combine \Cref{theoalphaJ}, \Cref{lemdep1}, \Cref{equivbalance} and \Cref{mb1}.
\end{proof}

 \subsection{Positive Results for Condition \A{3}}\label{pos3}$ $

We treat  Condition \A{3} before  Condition \A{2} since it is easier to obtain.
We start with $n = 3.$ First we note that the counterexample of Equation~\eqref{impma1} often  cannot be cut
to a $3 \times 3$ matrix.

 \begin{theoalpha}
     \label{theoalpha0} If the row rank   of a $3\times  3$ matrix $A$ over a null-bipotent pair with unique negation is  $3$, then $A$ is nonsingular.
 \end{theoalpha}
 \begin{proof} We prove the contrapositive, assuming that $A$ is singular. We claim that
the rows are  dependent unless $A$  can be obtained by
multiplying the matrix
\begin{equation}\label{goodmat3}\left(\begin{matrix} \one & \one & \one \\
\one & \one & a^{-1}  \\\one & a &\one
\end{matrix}\right)\end{equation}
by invertible diagonal matrices. Write $\mathbf v_i = (a_{i,1},
a_{i,2}, a_{i,3})$. For any $i$, $|A| = \sum _{j=1}^3
a'_{i,j}a_{i,j}\in \mathcal A _0.$
  Thus we are done if  $ a'_{i,j}  \in \tT$ for each $1 \le j \le 3$, so we may assume that
  the   $2 \times 2$ matrix $(a'_{i,j_i})$ is singular for suitable $j_1, j_2, j_3$. But in the $2 \times 2$   the two rows are proportional, by Proposition~\ref{pair3}.  If two of the $j_i$ are the same
  then the proportionality constants match, and the corresponding two rows are proportional, and we are done. Thus we may assume that
$j_1, j_2, j_3$ are distinct, and by symmetry we may assume that
$j_i = i.$ This means our vectors are:
$$\mathbf v_1 = (a_{1,1}, a_{1,2}, a_{1,3}), \qquad  \mathbf v_2 = (\beta_1 a_{1,1}, \beta_1 a_{1,2}, \beta_2\beta_3 a_{1,3}), \qquad
 \mathbf v_3 = ( \beta_2 a_{1,1}, \beta_1\beta_3^{-1} a_{1,2}, \beta_2 a_{1,3}).$$
for suitable $\beta_1 \in \tT$, which we may assume are all $\ne
\zero$ since otherwise the assertion is trivial. Normalizing, we may
assume that each $a_{1,j} = \one$, so
$$\mathbf v_1 = (\one, \one, \one), \qquad  \mathbf v_2 = (\beta_1  , \beta_1 , \beta_2\beta_3) , \qquad \mathbf v_3 = (\beta_2 , \beta_1\beta_3^{-1}  ,  \beta_2) .$$
Multiplying the second and third rows by $\beta_1 ^{-1}$ and
$\beta_2 ^{-1}$ respectively puts our matrix in the form
\eqref{goodmat3}. Thus $\Det{(A,\zero)} =(\one +a^{-1} +a, \one +\one+ \one)$ so
$$(a+\one)(a^{-1}+\one) = \one + (\one +a^{-1} +a) = \one +\one+ \one+\one.$$
If $(\mathcal A, \mcA_0)$ is of  the first kind, then $(a+\one)(a^{-1}+\one)\in \mcA_0,$
implying $a+\one \in \mcA_0$ or $a^{-1}+\one \in \mcA_0,$ so  $\mathbf v_1 +  \mathbf v_2 \in
\mcA_0^{(3)} $ or  $\mathbf v_1 +  \mathbf v_3 \in
\mcA_0^{(3)} .$
If $(\mathcal A, \mcA_0)$ is of  the second kind, then $\one+a +a^{-1}=\one,$
so, by \lemref{bi8}, $\one$ equals  $a$ or $a^{-1},$ and again $\mathbf v_2 \in
\mcA_0^{(3)} $ or  $\mathbf v_1 +  \mathbf v_3 \in
\mcA_0^{(3)} ,$ proving the claim.

But then, by the claim, the tracks are $1$, $a$ and $a^{-1},$ and the determinant is $2 \ne(-)\one (-) a(a^{-1}(-)\one)^2,$ so $a = \one,$ and thus $(-)\bfv_1 +\bfv_2 = (e,e,e).$
 \end{proof}

Condition \A{3} for $n > 3$ is much more intricate.

\begin{lemma}\label{bi8} $ $  Suppose that $(\mcA,\mcA_0)$ is a null-bipotent  pair of the second kind satisfying \PropertyN. For any $a\in \tT$, either \begin{enumerate}

       \item $a+a^{-1} = a+\one = a+a^{-1}+\one=a,$  or
        \item $a+a^{-1 }= a^{-1}+\one =a^{-1},$ or
         \item  $a^\circ = (a^{-1})^\circ = a+\one= a^{-1}+\one =a+a^{-1}+\one=e.$
 \end{enumerate}
 \end{lemma}
\begin{proof}
$a+\one = a( \one +a^{-1}).$

If $a+\one  =a$ then $$(a+\one) +a^{-1} = a+(\one+a^{-1}) = a+ a^{-1}(a+\one) = a+ a^{-1}a=a+\one = a,$$ and $a + a^{-1} = (a +1)+ a^{-1}=a.$

If $a+\one  =\one$ then $\one + a^{-1}= a^{-1}, $ and analogously
 $$a + a^{-1} = a^{-1}.$$

 Hence we may assume that  $a+\one = e,$ so $a^\circ =e$ and, by the analogous argument we may assume that $ {a^{-1}} +\one ={a^{-1}}^\circ =e$. Then $(a+\one) +a^{-1} = a+e = ae^+ = ae = e.$
 \end{proof}

For the remainder of this section, we assume $A\in M_n(\tTz).$


 When lacking weak transitivity and $\mu$-Noetherianity, we still want  an explicit
result in a sufficiently well-behaved situation.


 The proof
of Proposition~\ref{pair3} might seem to contain the kernel of a desired
proof of Condition \A{2} for general~$n$, but we already have our
counterexamples for $n=3$. Here is a tantalizing observation, which
also points to what could go wrong.

\begin{proposition}\label{pair4} Over a triple $(\mcA,\mcA_0,(-))$, if  $|A| \in \mathcal A_0$  then, for all $j$,
$\sum _{i=1}^n (-)^{i+j} a'_{i,j}\bfv_{i} \in \mathcal A_0^{(n)}.$
\end{proposition}

\begin{proof}
The $k$ entry of $\sum _{i=1}^n (-)^{i+j} a'_{i,j}\bfv_{i}$ is
equal to the $(j,k)$ entry of
$\adj  A A$. We have $\adj  A A\succeq_0 |A| I$ which is   in $\mathcal A_0$ by hypothesis,
so all the entries of $\adj  A A$ are in $\mathcal A_0$.
\end{proof}

 Unfortunately, we have no assurance  that the $ a'_{1,j}\in \tT$
 when $n>2.$  
We  have the following  counterexample for pairs of the first kind.

\begin{example}
\begin{equation}\label{goodmat4}\left(\begin{matrix} \one & \one & \one & \one \\
\one & \one & \zero & \one  \\ \zero & \one & \one &  \one  \\ \one
& \zero & \one & \one  \\
\end{matrix}\right)\end{equation}
whose $(-)$-determinant is in $\mathcal A_0$ in the metatangible uniquely negated truncated semiring pair of \Cref{metex}, (b)
with $m=4$
 where $\mathbf 4 = \mathbf {11}$,
even though the rows are  $\mcA_0$-independent (since any tangible linear combination is the
sum of a subset of rows, and any such sum has at least one odd entry which is less than $m=4$).
  \end{example}

To obtain Condition \A{3}
 we introduce the notation:

 \begin{definition}\label{bar0}
     Suppose $(\mathcal A, \mcA_0)$ is an $e$-final metatangible pair. (In particular $\mcA = \tT \cup \tT^\circ$ by Lemma~\ref{tt}.)

     $\bar a =a$ when $a\in \tT$.

   When $b\in \tT^\circ,$  $\bar b $ denotes an element $a\in \tT$
 (not necessarily unique) such that $a^\circ = b^\circ.$
$(\mathcal A, \mcA_0)$
 \end{definition}

\begin{lemma}\label{bar}
    In any  pair   $(\mathcal A, \mcA_0)$   of tropical type , \begin{enumerate}
        \item ${\sum a_i \bar b_i}^\circ =  (\sum a_i \bar b_i)^\circ  $ for $a_i\in \tT$, $b_i\in \mcA.$

         \item If $A$ is a matrix such that $\bar A$ has two equal rows,  then $\Det A \in \mcA_0.$
    \end{enumerate}
\end{lemma}
\begin{proof} (i)
    $(\sum a_i \bar b_i)^\circ = \sum a_i (\bar b_i^\circ) ,  $ so we need only observe that $\bar b ^\circ = b ^\circ$ because $(\mathcal A, \mcA_0)$ is $e$-final. (There is a fine point that for $b\in \mcA_0$ one could take two different values $a$ and $a'$ for $\bar b$, but even so, $  \bar b+\bar b =  a_1+a_2  =b ,$ by \Cref{theoalphaW}.

    (ii) Taking the determinant by partitioning, we may assume that $ A = (b_{i,j}),$ $1\le i,j\le 2$. If all $b_{i,j}\in \tT$ the assertion is obvious, so we assume that $b_{1,1} \in \mcA_0.$ If $b_{2,1}= b_{1,1}$ then $\Det A = b_{1,1}(b_{1,2}(-)b_{2,2})\in \mcA_0,$ so assume that $b_{2,1}\in \tT $, and thus $b_{1,1} = b_{2,1}^\circ.$ Likewise $\Det A \in \mcA_0$ unless $b_{1,2}\in \tT.$ If $b_{2,2}=b_{1,2}$ then $\Det A = b_{1,2}(b_{2,1}^\circ (-)b_{2,1} =b_{1,2}b_{2,1}^\circ\in \mcA_0.$ Finally, if $b_{2,2}=b_{1,2}^\circ$ then $\Det A = b_{2,1}^\circ b_{1,2}^\circ (-) b_{2,1} b_{1,2} = b_{2,1} b_{1,2}+b_{2,1} b_{1,2}(-)b_{2,1} b_{1,2} (-) b_{1,2} b_{2,1}(-)b_{2,1} b_{1,2} = (b_{2,1} b_{1,2})^\circ$, by $e$-finality.
\end{proof}

  \begin{theoalpha}\label{theoalphaR3} Condition \A{3} holds for   any  triple   $(\mathcal A, \mcA_0,(-))$   of  tropical type.
\end{theoalpha}
\begin{proof}
Given a matrix $A = (a_{ij}),$ we form a matrix $\bar A = (\bar a_{ij}),$  cf.~Definition~\ref{bar0}). Of course, when $A$ is tangible, $\bar A = A$. When $A$ is not tangible, we are assuming as always that $\mcA$ is a commutative semiring, and we reserve our choice of $\bar a_{ij}$.

Use the order of \Cref{theoalphaW}.
We say that a track    $a_\pi$ is  \textbf{null-dominant} if  $\Det A =a_\pi^\circ $, i.e., $a_\pi^\circ$ is maximal for the tracks of $A$.

Let $A$ be a singular $n\times n$ matrix with entries in $\mcA$.
    After multiplying the rows by suitable scalars and permuting columns, we may assume that the diagonal $a = a_{(1)}$  is null-dominant, and normalizing the rows, we may assume each $a_{i,i}\in \{\one, e\}.$   Write the adjoint matrix $\adj A = (a'_{i,j})$.
By \Cref{pair4}, we are done  unless each row of $\adj A$ has some entry in $\mcA_0.$

Let us push this idea.    We have $\adj A A= (\sum_{k=1}^n a'_{i,k} a_{k,j})$. 

For all $j\ne i,$
  $(\adj A A)_{ij}=  \sum_{k=1}^n a'_{i,k} a_{k,j}
   \in \mcA_0$ by Lemma~\ref{bar}(ii). If $a_{1,1}\in \mcA_0,$  then $ a'_{i,1} a_{i,1}\in \mcA_0$ contains the diagonal track, so $a'_{i,1} a_{i,1} \ge \Det A,$ implying $a'_{i,1} a_{i,1} = \Det A$. Hence $a'_{1,1} a_{1,1}=\Det A, $ by Lemma~\ref{bar}(i), implying $\sum _i\bar a'_{i,1} a_{i,j}    \in \mcA_0$ for all $j,$ by Lemma~\ref{bar}.

  Hence we may assume that $a_{1,1}\in \tT,$ and by symmetry each $a_{i,i}\in \tT,$ and, using Lemma~?? we could assume that another null-dominant track $a_\pi$ with $a_{(1)}+a_\pi \in \mcA_0$ includes $a_{2,1}$. Take $\bar a'_{1,1} = a_{2,2}\dots a_{n,n}.$ Then we can choose $\bar a'_{2,1}$ such that $a_{1,1}\bar a'_{1,1} +a_{2,1}\bar a'_{2,1}\in \mcA_0$ and choose  $\bar a'_{i,1}$ arbitrarily for $i>2$.

  Then   $\sum_{i=1}^n \bar a'_{i,1}a_{i,1} = \bar a'_{1,1}a_{1,1} + \bar a'_{2,1}a_{2,1} = \overline{\Det A}^\circ \in \mcA_0$.

\todo[inline]{SG: it may be the case that $\adj A$ is identically zero (no matchings of cardinality $n-1$ in the digraph of $A$). Then, we cannot argue by saying that a tangible selection of $\adj(A)$ yields the linear combination. In AGG14, we were solving this difficulty by induction, using Koenig-Frobenius theorem. I.e., when $\Det A$ was nonzero, we had a proof by sign setting/ Jacobi which in essence is related to your selection arguments. When $\Det A$ was identically zero, we were using Koenig-Frobenius to find a $p\times q$ zero matrix with $p+q=n+1$, see case 1 of the proof of Theorem 6.9 of \cite{AGG2}. Ie, some argument similar to case 1 is missing here.}
\todo[inline]{MA+SG: is Theorem 0 true for metatangible + uniquely negated + e final ?}
\todo[inline]{SG: a good example of semiring is polynomial functions with entries in smax or extended tropical semirings. I would expect the main body of results of Cramer type to be true in this semiring.}
We must contend with the case that all adjoints are 0. But then by the Koenig-Frobenius Theorem we can rearrange the rows and columns of $A$ such that there is $k$ such that $a_{i,j}=0$ for $1\le i \le n+1-k,$ $1\le j \le  k.$ Thus the $k \times k$  matrix in the upper right has a 0 column, so by induction on $n,$ the rows $\bfv_1, \dots, \bfv_k$ already are dependent.

\end{proof}

Here is an ``almost'' counterexample to the previous result.
  \begin{example} Take $\tTz$   to be the max-plus algebra with absorbing minimal element $\zero_{\mcA} = -\infty$, and let $(\mcA,\tGz)$ be the corresponding {supertropical pair}   arising from the ``ghost'' isomorphism $\mu: \tTz \to \tGz,$ addition defined as in~\Eref{modtr1}(i); thus $\mu(g)=g$ for all $g\in \tG.$
 Consider the matrix $A= \left(\begin{matrix}
    9 & 9 & \zero  & \zero  \\  9 & \zero  &  9 & 1    \\
    \zero  & 9 & 9   & 2 \\ 2 & 1 & \zero & 9
    \end{matrix}\right).$
    Then $\Det A = 81 +81 +81 = \nu(81),$ but no linear combination of the rows can be in $\tGz$. Indeed, suppose $\sum a_i \bfv_i\in \tGz^{(4)},$ where $\bfv_i$ denotes the $i$ row of $\mcA.$ Looking at the third column shows $a_2=a_3.$ But now the fourth column shows $2 a_2 =(2+1)a_2 = 9,$ so $a_2 = 7,$ and hence $a_1=7.$
  \end{example}

{\bf Question:}
    Does Condition \A{3} hold for   any pseudo-classical triple?

    Perhaps, the classical proof of Condition \A{3} could be adapted to the pseudo-classical case.
Later, in \S\ref{Jacobi1} we shall introduce ideas from \cite{AGG2}.

\subsection{Positive Results for Condition \A{2}}\label{pos}$ $

We turn next to   situations in which   Condition \A{2} holds
after all.
 Note that
the hypothesis of  Condition~\A{2} is not affected by multiplying  any
row through by a given tangible element, which we do repeatedly
throughout the proofs.

\begin{rem}   Condition \A{2} is obvious for $n =1$. Furthermore, the entries of any $1 \times 2$ matrix are dependent,
so the column rank of any nonzero $1 \times n$ matrix is 1.
\end{rem}


 \begin{proposition}\label{pair3} For any~$n$, if $(\mcA,\mcA_0)$ is uniquely negated, and $A$ is a tangible $2 \times
n$ matrix of submatrix rank 1,  then either one row is $(\zero)$ or the two rows of $A$ are tangibly proportional.\end{proposition}

\begin{proof} Let $\mathbf v_i = (a_{i,1}, \dots, a_{i,n})$  be the rows of the  matrix $A = (a_{i,j})$
for $i = 1,2$. First assume that $a_{1,1} = \zero.$ Then the
principal $2 \times 2$ minor has $(-)$-determinant $(-)a_{1,2}a_{2,1}$, and $a_{1,2}a_{2,1} \, \in \mcA_0 $ implies
  $a_{1,2}$ or $a_{2,1}$
is~$\zero$. If $a_{2,1} = \zero$ we conclude by induction on $m$;
otherwise we go down the line to get  $a_{1,j}=\zero$ for each $1
\le j \le n,$ and the proposition is clear.

Hence we may assume that $a_{1,1} \ne \zero$. By symmetry we may
assume that each $a_{i,j} \ne\zero.$
Localizing, we may assume that  $(\mcA,\mcA_0)$ is a gp-pair. Now we may normalize
(multiplying row $i$ by $a_{i,1}^{-1}$ for $i=1,2$) and assume that $a_{i,1} =
\one$ for $i = 1,2.$ But then the $1,j$ minor being singular implies $a_{1,j}=
a_{2,j}.$
\end{proof}

\begin{corollary}
Condition \A{2}  holds for  any tangible $2 \times
n$ matrix over a uniquely negated null-bipotent pair, for any~$m$.
\end{corollary}

\begin{theoalpha}\label{theoalphaR} Condition \A{2} holds for any semiring pair   $(\mathcal A, \tT,
(-))$  of tropical type .
\end{theoalpha}
 \begin{proof}[Proof of \Cref{theoalphaR}]  Suppose that $m \le n$, and the $m \times n$
matrix $A = (a_{i,j})$, whose rows are $\mathbf v_1, \dots, \mathbf
v_m$, has the property that every $m\times m$ submatrix $A'$
satisfies $\Det{ A' } \in \mcA_0$. We shall prove that the vectors $\mathbf v_1, \dots,
\mathbf v_m \in \mathcal A^{(n)}$ are
 dependent. We
induct on $n$, having proved the theorem \ref{theoalphaR3} for $m=n$. Thus, we may
%
For each   $j\le n$  we define $\mathbf v_i^{(*j)}$ to be the vector obtained
by deleting the $j$ entry, and $A^{(*j)}$ to be the submatrix of $A$
obtained by deleting the $j$ column of $A.$ In other words,  the
rows of $A^{(*j)}$ are   $\mathbf v_1^{(*j)}, \dots, \mathbf v_m^{(*j)}.$
By induction, we have a  dependence for $\mathbf v_1^{(*j)}, \dots,
\mathbf v_m^{(*j)};$ i.e., there are $\gamma_{i,j} \in \tTz$ such that
$\sum_{i=1}^m \gamma_{i,j}\mathbf v_i^{(*j)} \in \mathbf
A_0^{(n)}.$


We are done if $\sum _i \gamma_{i,j}a_{i,j} \in \tAz$
for some $j$ (since then $\sum _i \gamma_{i,j}\mathbf v_i \in \mathbf
A_0^{(n)})$, so we may assume for each $j$ that $\sum _i
\gamma_{i,j}a_{i,j} \in \tT$. Pick $i_j$ such that $\sum _i
\gamma_{i,j}a_{i,j} = \gamma_{i_j,j}a_{i_j,j}.$

Since there are $n\geq m+1$ values of $i_j,$ $j=1,\ldots n$ the pigeonhole
principle says that two are the same. To ease notation, we assume
that $i_{j'} = i_{j''} = 1$ for suitable $j',j''$. We say that $a_1$ \textbf{monopolizes} $a_2$ if $a_1+a_2 = a_1.$  Thus,
$\gamma_{1,j'}a_{1,j'}$ monopolizes all $\gamma_{i,j'}a_{i,j'}$, and
$\gamma_{1,j''}a_{1,j''}$ monopolizes all $\gamma_{i,j''}a_{i,j''}.$
Let
$$ \gamma_i =
\begin{cases} \gamma_{1,j''}\gamma_{i,j'}& \quad \text{if}\quad \gamma_{1,j''}\gamma_{i,j'}=
 \gamma_{1,j'}\gamma_{i,j''};\\
\gamma_{1,j''}\gamma_{i,j'} + \gamma_{1,j'}\gamma_{i,j''} & \quad
\text{if} \quad \gamma_{1,j''} \gamma_{i,j'} \ne
\gamma_{1,j'}\gamma_{i,j''} .\end{cases}$$

We shall conclude   by proving that \begin{equation}\sum_i \gamma_i
v_i = \sum _i \gamma_{i,j'} v_i + \sum_i \gamma_{i,j''} v_i\in
\mathcal A_0^{(n)}.\end{equation} We need to show that
\begin{equation}\label{check1}\sum_i \gamma_i a_{i,j}  \in \mathcal A_0\end{equation} for each $j$. The
verification of \eqref{check1}
 for $j \ne j', j''$ is immediate since we are
given $\sum_i \gamma_{i,j'} a_{i,j} \in \tAz$ and $\sum_i
\gamma_{i,j''} a_{i,j} \in \mathcal A_0$, implying at once that $\sum_i
\gamma_i a_{i,j} \in \mathcal A_0$. Thus, we need to check \eqref{check1}
for $j = j'$ and $j = j''$; by symmetry, we assume that $j = j'$. By
assumption, $\gamma_{1,j'} a_{1,j'} $ monopolizes $\gamma_{i,j'}
a_{i,j'} $ for each $i.$ Thus, $\gamma_{1 }  a_{1,j'}=
\gamma_{1,j'}\gamma_{1,j''} a_{1,j'}  $ monopolizes
$\gamma_{1,j''}\gamma_{i,j'} a_{i,j'} $. On the other hand,
$\gamma_{1,j'}\gamma_{1,j''} a_{1,j'}$ is monopolized by $$ \sum _i
\gamma_{1,j'}\gamma_{i,j''} a_{i,j'} = \gamma_{1,j'}\sum _i
\gamma_{i,j''} a_{i,j'} \in \mathcal A_0,$$ by the  dependence for
$v_1^{(j'')}, \dots, v_m^{(j'')};$ so we conclude 
that
$$\sum_i \gamma_i a_{i,j}   =  \gamma_{1,j'}\sum _i
\gamma_{i,j''} a_{i,j'} \in \mathcal A_0,$$ as desired. 
\end{proof}




\subsection{Condition  \A{4}}$ $

   Condition \A{4}, being  weaker than Condition \A{3},
 holds in many more situations, so we study
 Condition \A{4}, before Condition \A{3}.

 \begin{example}\label{impmA2a}
 Suppose $(\mcA,\mcA_0)$ is a  uniquely negated pair of the second kind  (such as a doubled  pair). Take the matrix
 \begin{equation}\label{comp1}\left(\begin{matrix}
                \one & \zero & \zero  & \one  \\
                 \zero & \one  & \one & \zero \\  \one & \zero  &  \one &\zero    \\
                \zero  & (-)\one&  \zero  & \one \\ c_1 & c_2 & c_3 & c_4
               \end{matrix}\right),
 \end{equation} with $c_1>c_2>c_3>c_4.$  Suppose the rows $\mathbf v_i$ are dependent, say $\sum_{i=1}^4 a_i \mathbf v_i = a \mathbf v_5.$ If the  first four rows were dependent, then $a_1=(-)a_3=a_2= a_4$ by columns 1, 3, and 2, but $a_1=(-)a_4$ by column 4. Hence  $a\ne 0.$ Replacing $c_i$ by $ac_i,$ we may assume that   $a = \one.$ We claim that there is a unique solution. The fourth column says $a_1=c_4$ or $a_4=c_4,$ or $a_1 = (-)a_4$ with $a_1>c_4.$

 \begin{enumerate}
         \item If  $a_1= c_4$ then the first column says $a_3 = c_1.$ Hence the third column says $a_2=(-)c_1$ and  the second column says $a_4=c_1$,
contradicting the fourth column.

         \item If  $a_4= c_4$
 then the second column says $a_2=c_2$
and the third column now says   $a_3 = (-)c_2.$ But the first column now says $a_1 = c_1$, contradicting the fourth column.

  \item We are left with  $a_1 = (-)a_4$ with $a_1>c_4.$
  \begin{itemize}
      \item If $a_1 <c_1$ then  the first column says $a_3 = c_1.$ Now the second column says $a_2=c_1$, contradicting the third column.

            \item If $a_1 =c_1$ then  the first column says $a_3 = (-)c_1.$ Then the third column   says  $a_2=c_1$.   Now the second column says $a_4=c_1$, contradicting the fourth column.

             \item If $a_1 =(-)c_1$ then  the first column says $a_3 = (\pm)c_1.$ If $a_3 =  c_1,$ then we have a contradiction as before. However, if $a_3 = (-)c_1,$ then the third column   says  $a_2=c_1$.   Now the second column says $a_4=c_1$.

  \end{itemize}
 \end{enumerate}

 \end{example}

But there is a counterexample.
\begin{example}[The negated power set of a group]\label{hkrashy22}$ $
Take any  abelian group.

\begin{enumerate}
    \item Recall the \textbf{symmetric difference of sets} $S \Delta S' = S\cup S' \setminus (S\cap S').$
Define addition on the power set $\mathcal P(G)$ by putting $S\oplus S' =  S \Delta S'.$ (The zero element is the empty set.)  In  particular, $$g_1 \oplus g_2 = \begin{cases} \emptyset \text{ if } g_1=g_2;\\  \{g_1,g_2\} \text{ if } g_1\ne g_2.
    \end{cases}$$
This addition is associative, since the symmetric difference is associative.
Furthermore $$g(\oplus_{i=1}^n g_i) = g(\{g_1\}\Delta \dots \Delta \{g_n\} )=
 \{gg_1\}\Delta \dots \Delta \{gg_n\} =
\oplus gg_i,$$ for $g,g_i\in G.$ Thus   $\mathcal P(G)$ becomes a $G$-module of characteristic 2 (since $S \oplus S = \emptyset$).

 \item  Suppose $G$ has distinct elements $\one,x.$
The
vectors   $v_1 =( \one,\one)$,
 $v_2 = ( \one,x)$,  $v_3 = ( \zero ,\one)$ are independent, Indeed,
 in order for $(\zero) =\sum a_i v_i$ we must have $a_1=a_2.$
 If $a_1\ne \zero$ then  $\sum a_i v_i = (\zero, \{a_1 , a_1x\} \oplus  \{a_3\}) $, which cannot be $(\zero).$
\end{enumerate}
\end{example}

Thus, we have a counterexample to Condition A3.
However, \Eref{hkrashy22} is not a hyperfield, since the addition in sets is not elementwise!
  \Cref{hkrashy22} does meet the conditions  of \cite{NakR}.
  Varied verifications are given in \S\ref{hy1}
  for Condition \A{4} over a hyperfield pair, but we do not manage to resolve it over hyperfield pairs.

\begin{rem}
To get a proof of Condition \A{4}, for hyperfield pairs, one can try to solve $x\preceq Ax +b$ and construct
$x_{k+1}$ as
$x_k\le x_{k+1}\preceq A x_k+b$.
On need for that that for all $a'\in \tT$, $a,b\in \mcA$
such that $a'\preceq a\le b$ there exists $b'\in \tT$ such that
$a'\le b'\preceq b$.
This would follow from
for all $a, c\in \tT$, there exists $b\in\tT$
such that $a\le b\preceq a+c$.
that is there exists $b'\in \mcA$ such that $b=a+b'\in\tT $
and $a+b'\preceq a +c$.
\end{rem}

 Lacking Condition  \A{4} in general, we present a weaker result which still can be useful.

\begin{proposition}\label{tanplu4}
      The rows of an $m\times n$ matrix $A = (a_{i,j})$ are dependent if $m \ge 2^n,$
when either   $(\mcA,\mcA_0)$ is metatangible   or $(\mcA,\mcA_0)$ is  a hyperpair. \end{proposition}
\begin{proof}
 We may assume $m = 2^n.$  Let  $A_1 = (b_{i,j})_{1\le i \le 2^{m-1},\, j = 1,\dots, n-1}$ and $A_2 = (b_{i,j})_{{2^{m-1}+1}\le i \le 2^{m},\, j = 1,\dots, n-1}$. By induction on $n,$
 the rows  are dependent, so there are $a_i\in \tT$ such that $\sum_{ 1\le i \le 2^{m-1}, \ 1\le j \le {n-1}} a_i b_{i,j}\in \mcA_0$ and  $\sum_{ {2^{m-1}+1}\le i \le 2^{m}} a_i b_{i,j}\in \mcA_0$, for $1\le j \le n-1$.

 Let $b_1 = \sum_{ 1\le i \le 2^{m-1}  } a_i b_{i,n} $ and  $b_2 =\sum_{ {2^{m-1}+1}\le i \le 2^{m}} a_i b_{i,n}.$ By (\Cref{tanplu2}) and (\Cref{tanplu3}) there are $a_i'\in \tT$ for which $a_1'b_1+a_2'b_2\in \mcA_0.$ Hence
 $\sum_{ 1\le i \le 2^{m} \ 1\le j \le {n}} a_1'a_i b_{i,j}  +\sum_{   {2^{m-1}+1}\le i \le 2^{m}} a_2'a_i b_{i,j}\in \mcA_0$ for~$   1\le j \le {n},$ giving the desired dependence.\end{proof}

 \begin{corollary}
     If  $(\mcA,\mcA_0)$ is as in Proposition~\ref{tanplu4}, then any $2^n$ vectors of $\mcA^{(n)}$ are dependent.
 \end{corollary}

\subsubsection{Condition \A{4} with a modulus}
$ $

\begin{theoalpha}\label{theoalphaP} Conditions \A{3} and \A{4} are equivalent, with respect to $\nablaT,$ for any    $\tT$-semiring   triple  $(\mathcal A,
\mathcal A_0,(-))$ with a $\tT$-modulus  $\mu$ satisfying the property:
  $\mu(a) = e$ implies $a^\circ = e = e+\one.$
\end{theoalpha}
\begin{proof}
(Based on \cite{IzhakianRowen2008Matrices}.)

In view of \Cref{imps},
we need to prove that Condition \A{4} implies Condition \A{3} under the given hypotheses.
Suppose $| A| \in \mcA_0$. We want to prove that its  rows
$\mathbf v_1, \dots, \mathbf v_n$ are
 dependent.

Taking fractions, we may assume that    $\tT$ is a group.  By induction on
$n$, we assume that the theorem is true for all matrices of smaller
size. (The assertion is vacuous for $n=1$, since either hypothesis is equivalent to the single entry of $A$ being in $\mcA_0.$)

With regard to the given  $\tT$-modulus $\mu$,   take a dominant track
$$a _\pi = a_{ \pi(1),1 } \cdots a_{ \pi(n),n },$$
where $\pi$ is a permutation  on $\{1, \dots, n \}$.

Rearranging the
vectors and their columns, we
 assume that the diagonal is dominant,
 i.e., $\pi = (1),$ and
$$a: =  a_{1,1} \cdots a_{n,n};$$
If $a = \zero$ are done by  Condition \A{4}, so we assume
that $a \ne \zero$.

By convention we write $b<b'$ if $\mu(b)<\mu(b').$ By assumption, $a _ \sigma \le_\mu a ,$  for all permutations $ \sigma$ on
$\{1,\dots, n \}$.

Assume first that some $a_{i,i} \in \mcA_0;$
we may assume that $i=1$. Here we  proceed similarly to the proof of Case I in the proof of
\cite[Theorem~6.5]{IzhakianRowen2008Matrices}. Write the
  matrix $\adjsym A={a_{i,j}'}^{\operatorname{doub}}$.

  Namely, take $c_i\in \widehat{\tTz}$
of the same $\mu$-value as ${a_{i,j}'}^{\operatorname{doub}}$.
 Then $\sum_{i=1}^n c
_i a_{i,1}$ has the same
 $\mu$-value as $\sum a'_{i,1}a_{i,1} = \Det A,$ but is in $\mathcal
 A_0$
 since by hypothesis $c_1 a_{1,1} $ is dominant and $a_{1,1}\in
\mathcal
 A_0.$ As in \cite{zeilberger85} or \cite{AGG1} (i.e., checking that the extra terms match in pairs),
 $\sum_{i=1}^n c _i a_{i,j}   \in \mathcal A_0,$ for all
$j \ne 1.$ Thus, $\sum_{i=1}^n c _i v_i   \in {\mathcal
 A_0}^{(n)},$ as
desired.

Hence we may  assume that   each $a_i:=a_{i,i} \in \tT.$ The same argument enables us to assume that every dominant track $\sigma$ is tangible.
We consider the path $i_1,$ $i_2 =  \sigma (i_1),$ \dots, $i_k =
\sigma (i_{k-1})$ in the reduced digraph  of $A$, where we
stop when $ \sigma (i_k) = i_u$ for some $1 \le u\le k$; since any
permutation $ \sigma$ is 1:1, we must have $ \sigma(i_k) = i_1$,
 so our path is a cycle. Furthermore we take  $ \sigma$ such that $k$ is maximal
 possible. Let $a'_i:= a_{i,\sigma(i)}.$  Note that because our path follows a permutation
of $(i_1, \dots, i_k),$ which could be extended (via the identity) to a  permutation
of $( 1, \dots, n),$ we must have
$ \mu(a_{i_1' }a_{i_2'}\dots
a_{i_k'}) \le  \mu(a _{i_1} \cdots a
_{i_k}),$
by hypothesis, and thus \begin{equation}\label{track1}  \mu(a_{i_1' }a_{i_2'}\dots
a_{i_k'}) =\mu(a _{i_1} \cdots a
_{i_k}).\end{equation} 

Renumbering the indices, we assume that $i_u = u$ for each $1 \le u
\le k.$ Thus, the upper $k\times k$ corner of~$A$ is the matrix $\diag{(a_i)} + \sum_{i=1}^k a'_ i e_{i,\sigma(i)} +$  other entries, e.g., it   somewhat resembles the
matrix \begin{equation}\label{det54} \left(\begin{matrix} a_{1 } & a_{1,2} & * & \dots & * & * \\
* & a_{2 } & a_{2,3} & \dots & *  & *  \\ * & * & a_{3 } &
\dots & *  & *\\ \vdots & \vdots & \vdots & \ddots & \vdots & \vdots
\\ * & * & * & \dots & a_{k-1 } & a_{k-1,k} \\ a_{k,1} &
* & * & \dots & * & a_{k}
\end{matrix}\right)\end{equation}
where here $a_i' = a_{i,i+1}$ for $i<n$ and $a_n' = a_{n,1}.$ (This is for illustration; actually $a_i'$ could be in any non-diagonal position in the $i$-th row.)

 Replacing the
row $\mathbf v_2$ through by $\frac {a'_{1}}{a_{2}}\mathbf v_2,$ and
then replacing the row $\mathbf v_3$ by $\frac {a_{2}'}{a_{3}}\mathbf
v_3,$ and continuing until~$\mathbf v_k$, we may assume that
$a'_{j-1} = a_{j}$ for $2 \le j \le k-1.$
Also \eqref{track1} implies  $\mu(a'_{k})  = \mu(a_{1} ).$

Next we normalize columns to obtain   $a'_{j-1} = a_{j} = \one$
for each $1<j \le k,$ and $a_{1} = \one.$
 Now, for any
entry $a_{s,t}$, $s \ne t \le k$, such that $\sg^\ell(s) = t,$
we take   the set
$J_{s,t} =\{ s,  {\sigma(s)}, \sigma^2(s)),\sigma^\ell(s)=t\},$ and have
\begin{equation}\label{det8} a_{s,t}   \prod_{i=1}^{\ell-1} a'_i \prod_{i \notin J_{s,t}}a_i = \mu (a_{1}\cdots
a_{s-1}a_{s+1}\cdots a_{k} a_{s,t} ) ,\end{equation}
implying $a_{s,t}  \le_\mu   a_{s} = e$ (since otherwise
\eqref{det8} monopolizes $a $), and likewise when $\sg^\ell(t) = s,$ we
again have the term

\begin{equation}\label{det7}
a_{s,t}   \prod_{i=1}^{\ell-1} a'_i \prod_{i \notin J_{s,t}}a_i = \mu (a_{1}\cdots
a_{s-1}a_{s+1}\cdots a_{k} a_{s,t} )        ,\end{equation}
implying $a_{s,t}  \le_\mu a_{t}= e$.
 In other words, each entry
in the $j$-th column for the first $k$ vectors has $\mu$-value~$\le e $.

Now, we rearrange the remaining columns $(j>k)$ such that first we
take those $j$ for which $a_{i,j} \ne \zero  $ for some $i\le k$.
Taking that $s\le k$ for which $\mu(a_{s,t}) $ is maximal (for
fixed $t$), we multiply $\mathbf v_{t}$ through by~$\frac
{a_{s,t}}{a_{t }}$; thus, we have $a_{t }  = a_{s,t}
.$ We repeat the argument of the previous paragraph to show that
$a_{t,j}\le_\mu  a_{ j}$ for all $1\le j \le k.$ (Namely, when $\sg^\ell(j) =s$, if
$a_{t,j}
>_\mu a_{j }$ , then letting $J_{s,t} =\{ s,  {\sigma(s)}, \sigma^2(s)),\sigma^\ell(s)=t\},$
$$\begin{aligned}  a_{t,s} \prod_{i \notin J_{s,t}} a_{i} \prod _{i=1}^{\ell-1} a'_{\pi^\ell(j)} >_\mu
   a_{1 }\cdots a_{j -1 }a_{j,j}a_{j+1, j+1}\cdots
a_{s}a_{t}\cdots a_{s+1}\cdots
a_{k},\end{aligned}$$ contrary to assumption on $a;$ if
 $\sg^\ell(s) =j$, then we use the analogous term.)

 We continue in this way, until we reach $k'$ such that, for
all $t>k'$, $a_{i,t} = \zero $ for all $i\le k$.

First assume that $k'<n$, i.e., $A$ is block triangular; then one of the blocks is singular, so by induction its rows are dependent. If the lower right-hand block is singular,  we take the dependence of its rows and put the coefficients
of the upper rows to be $\zero$.
 If the upper left-hand block is singular
then we start with the dependence of its rows
and solve for the coefficients
of the lower rows.

Hence we may assume that $k' = n.$
Now we normalize all columns such that the dominant entries  all have $\mu$-value
$e$. We may replace all entries of $\mu$-value $<e$ by $\zero.$

We want to prove that the sum of the rows is in
$\mathcal A_0^{(n)}.$
 If $k<k'=n$
then we can write down the dependence, so assume that $k = n$.

CASE I.  $(\mcA,\mcA_0)$ is of the first kind.
Since by hypothesis $\one + a = e$ for all tangible $a,$
we can replace all the nonzero entries by $\one.$

If $\mathbf 3 = e$ then the sum of the rows is clearly $e.$

If $\mathbf 3 = \one$,
then the  classical
determinant modulo 2 is 0, so the sum of the rows is in $\mathbf
A_0^{(n)}$, being a sum of $\one$ an even number of times.

CASE II.   $(\mcA,\mcA_0)$ is of the second kind with $e^+=e$.
We choose dominant tracks $\pi$ and $\sg$ with $a_\pi + a_\sigma\in \mcA_0.$

Replacing all  entries other than those specified in \eqref{track1}
by $\zero$ does not affect the $(-)$-determinant (since they do not
further contribute anyway because of the hypothesis), so after reduction we are left
with the matrix
$$I_n + \sum _{i=1}^n a_i'e_{i,\sg(i)},$$
where $a_i' = \one$ for all $i>1.$

First assume that $\sigma$ is even.
If
$n = 2m-1$ is odd then $\sum _{i=1}^m v_{\sg^{2i-1}(i) } +a'_1 \sum _{i=1}^{m-1} v_{\sg^{2i}(i)}  \in \mcV_0.$

If $n =2m$ is even then $({\Det A} _+ , {\Det A} _-) = (\one,a'_1)$
so $a_1' \nablaT \one,$ say $\one+c,\ a'_1+c \in \mcA_0$ for $c \in \tT $. Then
$$\sum _{i=1}^m (v_{\sg^{2i-1}(i) } +c v_{\sg^{2i}(i)} ) \in \mcV_0.$$

When $\sigma$ is odd we reverse the dependence.

$n = 2m-1$ is odd then $({\Det A} _+ , {\Det A} _-) = (\one,a')$
where $a' \nablaT \one,$ say $\one+c, a'_1+c \in \mcA_0$ for $c \in \tT $, so
$\sum _{i=1}^m (v_{\sg^{2i-1}(i) } +c v_{\sg^{2i}(i)} ) \in \mcV_0.$

If $n =2m$ is even then $\one + a'_1 = e,$ so  $\sum _{i=1}^m v_{\sg^{2i-1}(i) } +a' \sum _{i=1}^{m} v_{\sg^{2i}(i)}  \in \mcV_0.$
\end{proof}

\subsubsection{Condition \A{4} for hyperfield pairs}\label{hy1}$ $

For the remainder of this section, we shall consider hyperfields $\mathcal{H}$.
As noted by Baker and  Zhang \cite{BZ},  the solution of Condition \A{4} for  quotient hyperfield pairs, defined in \Cref{Krs}, follows directly from classical linear algebra,  cf.~Lemma~\ref{krd}.\footnote{There is some overlap with a paper submitted a few months earlier, by \cite{HoJ}, who called this condition  FETVINS.} Let us consider other  hyperfields. We write $a_1-a_2$ for $a_1\boxplus(-a_2),$ in a hyperfield $\mcH.$

Let us check that the   familiar examples in the hyperfield literature all satisfy Condition \A{4}.

 \begin{lemma}\label{tr1}$ $
     \begin{enumerate}
         \item
     $\zero \in \boxplus _{i=1}^n a_i $ if and only if $- a_n\in  \boxplus _{i=1}^{n-1} a_i $.
         \item If $a \boxplus (-a) \supseteq \mcH\setminus \{a\}$ for all $a\ne \zero$, then Condition \A{4} holds.
          \item If $a_1 \boxplus (-a_2) =\mcH\setminus \{a_1, a_2, \zero\}$ for all $a_1\ne -a_2 \ne \zero$, then Condition \A{4}  holds when $|\mcH|\ge 6.$
     \end{enumerate}
 \end{lemma}

\begin{proof}
    (i) By definition.

    (ii) We shall prove the vectors $\mathbf{v}_1, \dots, \mathbf{v}_{n+1}$ are dependent.  We may assume $a_{1,1}=\one.$ If all the other $a_{i,1} =\zero$ we can erase the first column, take $a_1=\zero$ and apply induction. So we may assume that $a_{2,1} \ne \zero$, so multiplying $\mathbf{v}_2$ through by $-a_{2,1}^{-1},$ we may assume that $a_{2,1} =-\one$. We continue the process, and have $\one$ appearing with $-\one$ in each column, so their hypersum with any other nonzero component in that column is $\mcH\setminus \{\pm \one\},$ which contains $\zero.$ (If all other  components in that column are $\zero,$ then we get $\zero \in \mcH \setminus \{\one\} = \one \boxplus \one.)$

   (iii) Take $\zero\ne a_4 \notin \{a_1,a_2,a_3\}$. Note that if $a_1, a_2, a_3$ are pairwise not negatives of each other and nonzero, then $(a_1 \boxplus a_2) \boxplus a_3$ contains $\{a_4, a_3\} \boxplus a_3$ which contains $\mcH \setminus \{ a_3, a_4\}$ and thus (letting $a_4$ vary) contains $\mcH \setminus \{ a_3\}$, whereas likewise
      $a_1 \boxplus (a_2 \boxplus a_3))$ contains $\mcH \setminus \{ a_1\}$, so by associativity  $a_1 \boxplus a_2 \boxplus a_3 =\mcH.$ It follows by a similar argument that the sum of any four nonzero elements is $\mcH.$

      Thus we may assume (after normalizing) that each column has
      two nonzero elements $\one,-\one$ or three equal elements nonzero elements $\one,-\one,a_j$. We assume that the first $k$ columns are of the form  $\one,-\one$. This leaves $(n+1)-(k+1) =n-k$ rows that we can adjust,
        multiply the rows so that each of these columns has $a_j',\pm \one, a_j$ when $a_j' \ne \zero \ne -a_j.$
\end{proof}

This result might seem trivial, but occurs quite frequently, i.e., if $\{a_1,a_2\} \subseteq a_1\boxplus a_2$ for all $a_1 \ne -a_2$, cf.~\cite[Proposition~2]{Mas}. In fact, it covers all of the examples given in \cite{Mas}!

 Here are two non-quotient examples from \cite{Nak}.

\begin{example}\label{hex} In each case $\mcH = G\cup \{\zero\}$ where $G$ is a multiplicative group, and $\zero$ is an absorbing element under multiplication and a neutral element under $\boxplus$.
    \begin{enumerate}
     \item  Hyperaddition is as follows, for all $a_i\in G$:
    \begin{itemize}
        \item      $a\boxplus a = \mcH \setminus \{a\},$
\item   $ a_1 \boxplus a_2=\{a_1 ,a_2\}$. \end{itemize} Condition  \A{4} holds, by Lemma~\ref{tr1}(ii).
        \item  Hyperaddition is as follows, for $a_i\in \mcH$:
        \begin{itemize}
            \item  $ a\boxplus a=\{\zero,
a\},$

 \item  $ a_1 \boxplus a_2=\mcH\setminus \{a_1 ,a_2,\zero \}$,      $\forall  a_1 \ne a_2 \ne \zero $.
        \end{itemize}
         Condition  \A{4} holds, by Lemma~\ref{tr1}(iii).
\end{enumerate}
\end{example}

Next, let us see that Condition \A{4} holds for  the examples in Viro \cite{Vi}. We say a hyperfield $\mathcal{H'}$ is \textbf{subsumed} in a hyperfield $\mathcal{H}$ if  $\mathcal{H'}= \mathcal{H}$ and $a_1 \boxplus' a_2  \subseteq a_1  \boxplus  a_2$ for $a_i \in \mathcal{H'}.$ (\cite{Mas} calls $\mathcal{H}$  an \textbf{augmented} hyperfield.)

\begin{lemma}\label{subs}
 Suppose that a hyperfield $\mathcal{H'}$   is {subsumed} in a hyperfield $\mathcal{H}$.
   \begin{enumerate}    \item  Vectors  which are dependent over  $\mathcal{H'}$, also are dependent over  $\mathcal{H}$.
    \item  If $\mathcal{H'}$ satisfies Condition \A{2}, then $\mathcal{H}$  satisfies Condition \A{2}.
       \item  If $\mathcal{H'}$ satisfies Condition \A{4}, then $\mathcal{H}$  satisfies Condition \A{4}.
   \end{enumerate}
\end{lemma}
 \begin{proof}
(i) $\zero \in \boxplus' a_i v_i \subseteq  \boxplus a_i v_i .$

(ii) and (iii) follow from (i).\end{proof}

\begin{example}\label{exn}$ $
  \begin{enumerate}
    \item Let us modify the triangle hyperfield of Viro \cite{Vi},   to $\mathcal{H} = \mathbb F^+$ for an ordered field $F$, with addition given by $a_1 \boxplus a_2 = \{ |a_1 - a_2|, a_1+a_2 \}$, the positive part of $\{ \pm a_ 1 \pm a_2\}$ since half of these are positive. Hence $\mathcal{H} $ is  identified with the quotient hyperfield $F/ \{\pm \one\}$ of \Cref{theoalphaE}, and thus satisfies Condition \A{4}.

    \item The triangle hyperfield ($\mathbb R^+,\boxplus),$  with $a_1 \boxplus a_2 := \{a \in \mathbb R^+: |a_1-a_2| \le a \le a_1+a_2\} $), subsumes~(i), so   satisfies Condition \A{4} by Lemma~\ref{subs}.

\item The ultratriangle hyperfield of \cite[\S 5.3]{Vi} satisfies Condition \A{4}, seen by taking logarithms.
 \end{enumerate}
\end{example}

\begin{example} \cite{Mas}, \cite[Theorem 9]{MasM}  $G$ is a multiplicative group, and $\mcH  := (G\times  \{\pm 1\}) \cup \{\zero\},$ where $\zero$ is an absorbing element under multiplication and a neutral element under $\boxplus$.
 \begin{enumerate}
     \item Multiplication is given componentwise.

 \item Hyperaddition is as follows, for $a_j\in G, i_j \in \{\pm 1\}$:
       \begin{enumerate}
           \item $(a_1,i_1)\boxplus (a_2,i_2)=\{(a_1, \pm 1), (a_2,\pm 1)\}$  for $(a_1,i_1)\ne (a_2,i_2) $,
             \item $(a_1,i_1)\boxplus (a_1,i_1)=   \mcH \setminus \{ (a_1,i_1),(a_1,-i_1),\zero\},$
             \item
                $(a_1,i_1)\boxplus (a_1,-i_1)=   \mcH \setminus \{ (a_1,i_1),(a_1,-i_1)\}.$
       \end{enumerate}
 \end{enumerate}

 This again satisfies  Condition \A{4} since for $|G|>2$, give distinct vectors with nonzero coefficients, we can multiply the third vector to make its coefficients differ from the first two vectors, and then take their sum. For $|G|=2$ one needs  a more careful, but easy argument.
\end{example}

\cite{HoJ,Ho} provided many new examples of non-quotient hyperfields, but they all satisfy Condition~\A{4}. We repeat the question of \cite{BZ}, whether Condition \A{4} holds for all hyperfield pairs.

\subsection{The Jacobi algorithm approach}\label{Jacobi1}$ $

Following \cite{AGG2}, by adding on extra
 natural assumptions, which they call ``Properties,'' we can  obtain a ``Jacobi algorithm'' which  yields Condition \A{3} at times.

\begin{lemma}\label{sim7}$ $
 Any  pair of tropical type  satisfies Properties 6.2-6.5 (in
\cite{AGG2}).
 \end{lemma}
\begin{proof}
  Property 6.2 holds by Lemma~\ref{tt}.

Property 6.3 means: $x^\circ = \zero =\zero^\circ $
implies $x = \zero$.

 Property 6.4 follows from the definition of $\mcA_0$-bipotence.

For Property 6.5, write $x  = b^\circ$ and assume that $y \le x ^\circ =
(b^\circ)^\circ = b^\circ.$
 Then $y^\circ \preceq b^\circ.$ If $y+b =y$ then $b^\circ \preceq y^\circ,$ implying
 $b^\circ = y^\circ $. But then $y+b \in \mcA_0$,
  so $x+y =  e^+b = e b
\in \mathcal A^\circ$.
 \end{proof}

\begin{rem}\label{Sec6}$ $

 \begin{enumerate}
\item \cite[Property~6.2]{AGG2} holds for metatangible pairs of the second kind.


 \item \cite[Property~6.3]{AGG2} is a running assumption.

\end{enumerate}
\end{rem}

but we do have  the following   consolation.
Since our main results will be formulated for   pairs, let us
compare this to \cite[Property~4.7]{AGG2}.

\begin{rem}\label{bipo}
Property~6.5 fails if  $e^+ =
 \one$ (taking $x = e$ and $y = \one$). Similarly it fails
 whenever $e^+ = \mathbf {2k+1}$ with $ \mathbf {2k} \ne  \mathbf
 {2k+1}.$ But any metatangible pair satisfying Property~6.5
 is   null-bipotent, by Theorem~\ref{theoalphaA}.\end{rem}
%


\subsubsection{Fibers of a modulus, cf.~Definition~\ref{mod1}}$ $

Although \cite[Definition~4.5(v)]{Row21} requires
all elements of $\tT$ to be incomparable under $\preceq$,  this need
not hold with respect to a modulus $\mu.$

\begin{definition}\label{fib0} We fix a modulus $\mu: \mcA\to \tG.$
 The  \textbf{fiber}  $\mcF(g)$ of
$g\in \tG$ is $\{ a \in \mathcal A: \mu  (a) = g\}.$ The $\tT$-\textbf{fiber}  $\mcF_\tT(g)$
of $g\in \tG$ is $  \tT \cap \mcF(g) .$

 The $\tT$- fiber  $\mcF_\tT(g)$ has \textbf{depth 1} if $a_1 \le a_2 \in \mcF_\tT(g)$ implies $a_1=a_2.$
(This is \cite[Property 5.4]{AGG2}.)
\end{definition}

%



%

\begin{rem} When $\tG$ is idempotent, the $\tT$-fibers have an (algebraic) semilattice structure given by
$$\mcF(g) \vee \mcF(h) = \mcF(g+h).$$ (We need idempotence so that $a
\vee a = a.$) \end{rem}

We write $b \ |\! \preceq \!\!|_\mu \  b'$ if $b \preceq  b'$ with
$\mu(b)  = \mu(b') .$
%

  \begin{rem}\label{Agg254} $ $
 \begin{enumerate}  \item   \cite[Property 5.4]{AGG2} says for $a_1,a_2\in \tT$ that  $a_1 \ |\! \preceq \!\!|_\mu \  a_2$ if $a_1 \preceq  a_2$ with
$\mu(a_1)  = \mu(a_2) .$
     \item  \cite[Property 5.4]{AGG2} holds for a uniquely negated null-bipotent pair, by~\Cref{theoalphaW}.
 \end{enumerate}
\end{rem}

\begin{example}
    An easy example where \cite[Property 5.4]{AGG2} fails is the supertropical algebra for $\tTz =\R,$ $\tG_0 = \R^+,$ and $\mu$ the absolute value. Then $(-1)+1 = e= 1+1 = (-1)+(-1)$.
\end{example}

\begin{lemma}\label{lemdep2} For $\tT$ a  group,
 \cite[Property
5.5]{AGG2}  holds with respect to
a modulus $\mu$.
\end{lemma}
\begin{proof} Take $\tilde d = d^{-1}$ for $d \in \tT.$
\end{proof}

\subsubsection{The Jacobi algorithm}\label{matra03}$ $

 The following paired version of \cite[Property 5.3]{AGG2}  suffices for our needs:
\begin{definition} A  semiring pair $(\mathcal A,\mathcal A_0)$ with a modulus $\mu$ is  $\mu$-\textbf{Noetherian}, if no $\tT$-fibers
have infinite ascending chains in $\tT$ with respect to the natural pre-order $\le$ of \lemref{bip8a}.
\end{definition}

The idea of using ``$\mu$-Noetherian'' is that   in any ascending chain one
passes (if possible) to a leading term under the $\mu$-preorder, and then any
further ascending chain has to remain in that fiber, and thus
terminates.

We are ready for another approach to Cramer's rule, making use of a property, which we call \textbf{modular descent}, which is a translation of
 \cite[Property 5.2]{AGG2} (Fibers are defined in \Cref{fib0}):
 \begin{equation} \label{goodprop} \text{If } a \in \tTz \text{ and } a \le c, \text{ then  }\mcF_\tT(\mu (c))
  \text{ contains some } a'\in \tT \text{ with }
 a  \le  a' \le c.
\end{equation}
\medskip

\begin{lemma}[Jacobi decomposition]\label{Jacobi} Suppose the
$\preceq_{\zero}$-  $\tT$-semiring   triple  $(\mathcal A,
\mathcal A_0,(-))$ has a
$\tT$-modulus~$\mu$ satisfying~\eqref{goodprop}. Then   any  
nonsingular matrix $A \in M_n(\mathcal A)$ with a dominant diagonal
has a \textbf{Jacobi decomposition} $A = D + N$ of matrices~$D$ and
$N \in M_n(\mathcal A)$ such that $D$ is a diagonal matrix with
entries in $\tTz$, and $ \mu( D) =  \mu(A)$.  \end{lemma}
\begin{proof} As in \cite[Proposition~5.18]{AGG2}.
\end{proof}

We look for a solution to the matrix equation $A x \succeq v,$ for a
given matrix $A$ and vector $\mathbf v$. (This strengthens
\cite[Theorem~5.20(2)]{AGG2}.)

\begin{proposition}
    [The Jacobi algorithm]\label{theoalphaN} The ``Jacobi algorithm'' of
 \cite{AGG2} applies to any  $\mu$-Noetherian
pair $\mathcal A$ satisfying \eqref{goodprop}, and a 
nonsingular matrix $A \in M_n(\mathcal A )$ having a dominant
diagonal. Explicitly, let $A = D + N$ be a Jacobi decomposition.
Then
\begin{enumerate}
 \item
There is an ascending sequence $\mathbf x_1, \mathbf x_2, \dots,$ for
$\mathbf x_k \in \tTz^{(n)}$ (with respect to the natural pre-order~$\le$),
satisfying $D \mathbf x_{k+1}\, |\! \preceq \!\!| _\mu  \, (N\mathbf x_k
+ v)$.

 \item When $\mathcal A$  satisfies \eqref{goodprop}, the  sequence $\mu(\mathbf x_1), \mu(\mathbf x_2), \dots,$  stabilizes
at $n$, i.e.,   $\mathbf x_k = \mathbf x_{k+1}  $ for all $k \ge n$,
    and moreover  $\mu(\mathbf x_n) =
\mu( |A|)^{-1}\mu(\adj{A} \mathbf v)$ when   $\mu(|A|)$ is invertible.

 \item When $\mathcal A$  satisfies \cite[Property
5.2]{AGG2} and is $\mu$-Noetherian, the  sequence $\mathbf x_1, \mathbf x_2,
\dots,$ stabilizes, i.e., there is $m$ such that $\mathbf x_k = \mathbf
x_{k+1} $ for all $k \ge m$  and,
 moreover,
the limit $\mathbf x_m$ is a tangible solution of $A\mathbf x \nabla
v$.

 \item When $\tT$-fibers have depth 1, one can take $m = n$ in (iii).
\end{enumerate}
\end{proposition}
\begin{proof}
(i) Follows from \cite[Equation~5.6]{AGG2}.

(ii)   We copy the proof of \cite[Theorem~5.20(2)]{AGG2}, which
 utilizes ascending chains, and requires  the $\mu$-Noetherian
property for chains to stabilize.

(iii) We get
 \cite[Property
5.5]{AGG2} from  Lemma~\ref{lemdep2}.

(iv) When $\tT$-fibers have depth 1, $x_n = x_{n+1}$ in (iii).
\end{proof}

 \begin{theoalpha} [{see \cite[Theorem 5.11]{AGG2}}]\label{A1exdist} Suppose that
 $(\mathcal A,\mcA_0)$ is a
 $\mu$-Noetherian pair  with a  $\tT$-modulus~$\mu$,
  satisfying~\eqref{goodprop} (which is ``convergence of monotone algorithms''  of \cite{AGG2}). 
 For any 
nonsingular matrix $A$ with $\mu (|A|)$ invertible, and
  every $\mathbf v \in \mathcal V $, there exists   $ \mathbf x\in
\tTz^{(n)}$ with  $\mu (\mathbf x) = \mu(|A|)^{-1}\mu(\adj{A}\mathbf
v).$
\end{theoalpha}
\begin{proof} As in the proof of {\cite[Theorem 5.11]{AGG2}}. (Applying a row permutation, we may assume that $A$ has a dominant
diagonal.)
\end{proof}

\begin{rem}\label{A1exdist2}
Proposition~\ref{theoalphaN} and Theorem~\ref{A1exdist} apply most readily
when every nonzero element of $\tT$ is invertible. But then one can
extend this to (cancellative) $\tT$-bimodules
 by taking fractions.
\end{rem}

 \begin{definition}
    The \textbf{tangible summand condition} is that if $b_1+b_2\in \tT$ then $b_1 \in \tT$ or $b_2 \in \tT$.
\end{definition}

 \begin{corollary} [{see \cite[Theorem 5.14]{AGG2}}] Condition \A{3} holds
 under the tangible summand condition,
 when  $\mathcal A$ is a
 $\mu$-Noetherian pair with respect to a $\tT$-modulus $\mu$ and the natural pre-order $\le$,
  satisfying~\eqref{goodprop}.
\end{corollary}
\begin{proof} As in  {\cite[Theorem 5.14]{AGG2}}, utilizing
Proposition~\ref{theoalphaN} to obtain its condition.
\end{proof}

\begin{theoalpha}\label{theoalphaQ} Over any   $\mu$-Noetherian semiring
pair of tropical type  with a surpassing relation
$\preceq$  which is a PO on $\mathcal A^\circ$,
if  $  A $ is singular then $\mathbf v_1, \dots, \mathbf v_n$
are
dependent. 
\end{theoalpha}
\begin{proof}
 Since  Properties
6.2-6.5 in \cite{AGG2} hold by Lemma~\ref{sim7}, we can follow the
proof of \cite[Theorem 6.9]{AGG2} (``homogeneous balances''). (We
must check in the proof of Subcase~2.2 in \cite[Theorem~6.9]{AGG2},
that \cite[Property~6.2]{AGG2}
 can be modified to $\mcA_0$-bipotence.)
The only place (iii) (\cite[Property 6.4]{AGG2}) is used is in the
computation of the $(-)$-determinant as the sum of products of elements of
$A$, which are each in $\tTz$ when $A \in M_n(\tTz)$.
\end{proof}

\section{$\preceq$-spanning and independence  }\label{sd}$ $

One can recast the theory of pairs to include a pre-surpassing relation. We assume throughout that $(\mcA,\mcA_0)$ is a  pair with a pre-surpassing relation  $\preceq$.

Although the approach of $\mcA_0$-dependence used in the body of this paper yields the closest connection to matrices, there is another notion of dependence, which sometimes ties in better to the structure theory,
relying on the following mainstay of \cite{Row21}.

\begin{definition}\label{deps} Let   $V=\mathcal A^{(n)}, $ and $V_0 = \mathbf A^{(n)}_0.$
\begin{enumerate}
  \item
The pre-surpassing relation
 $\preceq$ is defined componentwise on vectors in $V$, i.e., $(v_i) \preceq (v_i')$ if $v_i \preceq v_i'$ for each $i$.

  \item A $\preceq$-\textbf{morphism} $f:(V,V_0)\to (W,W_0)$ of vector spaces over a pair   with a surpassing relation is a
        multiplicative map also satisfying
        \begin{itemize}
      \item If $\bfv \preceq \bfv'$, then $f(\bfv)\preceq f(\bfv').$
      \item  $f(\sum \bfv_i)\preceq  \sum f(\bfv_i),$ for all $\bfv_i\in V.$
 \end{itemize}

    \item Vectors $\mathbf v_i, {i\in I},$ are  \textbf{strongly $\preceq$-independent} if they satisfy the property:\medskip

If a finite sum $\sum_{i\in I} b_i \mathbf v_i \preceq \sum_{i\in I} b_i' \mathbf v_i$ for $b_i,b_i'\in\mcA$,  then $b_i\preceq b_i'$ for each $i\in I.$
\end{enumerate}
\end{definition}

We defined $\mcA_0$-dependence of vectors $\mathbf v_i$ in Definition~\ref{dep1}.

\begin{example}\label{metex2}$ $\begin{enumerate}
                 \item $(\one,\zero)$ and $(\zero,\one)$ are strongly  $\preceq$- independent. Indeed, $b_1(\one,\zero)+b_2(\zero,\one) \preceq b_1'(\one,\zero)+b_2'(\zero,\one)$
                 implies $b_1\preceq b_1'$  and $b_2\preceq b_2',$ seen by checking the first and second components respectively. In general the standard base $\{\mathbf e_1,\dots, \mathbf  e_n\}$ is strongly $\preceq$-independent.
              \item $(\one,\zero)$ and $(\one,\one)$ are $\mcA_0$-independent, but are not
             strongly  $\preceq$-independent if  $a+a'=a'$ for some $a,a' \in \tT$ (cf.~\lemref{bA2}),
                 since then
                 $a(\one,\zero) +a'(\one,\one)= (a',a')=a'(\one,\one).$

               \end{enumerate}
\end{example}

Thus strong $\preceq$-independence is perhaps too strong a condition. One can use a more lenient  approach for spanning.
\begin{definition}\label{depsp}$ $
\begin{enumerate}
\item A vector $\mathbf v$ is $\preceq$-\textbf{spanned} by vectors $\mathbf v_1,
\dots, \mathbf v_n$ if $\mathbf v  \preceq  \sum_{i=1}^n a_i \mathbf v_i $
for suitable $a_i \in \tTz$.
A subspace $W \subseteq V$ $\preceq$-\textbf{spans}  $V$ if each $\mathbf v\in V$ is $\preceq$-{spanned} by elements of $W.$

\item A $\preceq$-\textbf{spanning} set for $W \subseteq V$ is a set of
vectors $\{ \mathbf v_i : i \in I\}$ $\preceq$-spanning each
$\mathbf w$ in $W$. (Note that we do not require that the $\mathbf v_i$
are in $W$.)

\item A subset $S\subseteq V$ is $\preceq$-\textbf{independent} if $S\setminus v$ does not $\preceq$-\textbf{span} $v,$ for any $v\in V.$
\end{enumerate}
\end{definition}

\begin{rem}
    $\preceq$-spanning has the advantage of being transitive,   whereas  $\mcA_0$-dependence is not transitive.
 \end{rem}

%
%

\begin{proposition}\label{firstc37}  Suppose that $\mathbf v_n$ is $\preceq$-{spanned} by vectors $\mathbf v_1,
\dots, \mathbf v_{n-1}$.
\begin{enumerate}
    \item
When $(\mcA,\mcA_0)$ satisfies \PropertyN,  then $\mathbf v_1,
\dots, \mathbf v_{n}$ are
dependent.
    \item For $\mathbf v_i$ of length $n,$
 the matrix $A$ whose rows are $\mathbf v_1, \dots, \mathbf v_n$ is singular.
\end{enumerate}\end{proposition}
\begin{proof}
(i) $  \mathbf v_n \preceq \sum _{i=1}^{n-1} a_i \mathbf v_i
$ implies $  \mathbf v_n +\sum _{i=1}^{n-1} a_i^\dag \mathbf v_i
\in \mcV_0$ when $a_i+a_i^\dag\in \mcA_0$.

(ii) This is easy when $ \mathbf v_n =  \sum _{i=1}^ {n-1} a_i \mathbf v_i, $ seen by breaking up $\mathbf v_n$ into its summands, so the assertion follows by applying Remark~\ref{modu-order2} to the determinantal formula.
\end{proof}
Thus  the minimal
number of $\preceq$-spanning rows of a matrix is at least the row rank and the submatrix rank.
We can also improve Proposition~\ref{firstc37a}, using the surpassing relation $\preceqzero$.

\begin{theoalpha}\label{firstc37b}  Assume that triple $(\mcA, \mathcal A _0,(-))$  with $\mcA = \tT \cup \mcA_0$.
     Suppose the tangible vectors $\bfv_1, \dots, \bfv_{n+1}$ are the rows of an $(n+1)\times n$ matrix $A$. Then either $A$ has a $\nabla$-singular $n\times n$ submatrix or each vector is $\preceqzero$-spanned by the others.
\end{theoalpha}
\begin{proof}
    We follow the proof of  \Cref{firstc37a}, which says that $\bfv_{n+1}$ is $\preceqzero$-spanned by $\bfv_1, \dots, \bfv_{n},$ and this works for every $\bfv_i.$
\end{proof}
\begin{example}
    The matrix
\begin{equation}\label{impma} \left(\begin{matrix}
                \one & \zero &  \one^\circ\\
               \zero &  \one^\circ & \one \\ \one^\circ & \one & \zero
               \end{matrix}\right), \end{equation}
over the super-Boolean pair, has  row  rank 2, but no row is $\preceq$-spanned by the other two rows.

However, this matrix is not tangible.
\end{example}

There also is a $\preceq$-version of \Cref{theoalphaJ10} with the same proof:

\begin{theoalpha}[Cramer's rule with $\preceq$, Part 1]\label{theoalphaJ11}
 Let $(\mathcal A,
\mathcal A_0,(-))$  be a $\tT$-semiring \triple with a pre-surpassing relation $\preceq$.
Then, for any vector $\mathbf v$,
the vector $\mathbf w  = {\adj A\mathbf v }$ satisfies $|A|\mathbf v \, \preceq \, A \mathbf w.$
\end{theoalpha}


When we limit our attention to tangible vectors, we often return to $\mathcal V_0$-dependence.

\begin{proposition}
    Suppose tangible vectors $\bfv_1,\dots,\bfv_n$ are $\mathcal V_0$-dependent. Then $\bfv_1,\dots,\bfv_n$ are $\preceq$-dependent under either of the following situations:
    \begin{enumerate}
        \item $(\mcA,\mcA_0)$ is a hyperfield pair.

  \item $(\mcA,\mcA_0)$ is a uniquely negated metatangible pair either
  \begin{enumerate}
      \item of second kind, or

   \item  of first kind, of $\mcA_0$-characteristic 0.
\end{enumerate}\end{enumerate}
\end{proposition}
\begin{proof} Write $\bfv_j = (v_{i,1},\dots,v_{i,m})$ for $v_{i,j}\in \tT,$ and write $\sum _{i=1}^n a_i \bfv _i \in \mcV_0$ with $a_n\ne 0.$ Thus $\sum_i a_i v_{i,j} \in \mcA_0 $ for each $j.$

    (i) $0 \in \boxplus _{i=1}^n a_i v_i $ implies $v_n \in \boxplus _{i=1}^{n-1} -a_n^{-1} a_i v_i $.

    (ii) Follows immediately from \Cref{A1partA}(i).
\end{proof}

One troublesome feature of tropical geometry is that there are
``too many'' vector subspaces; for example over the max-plus algebra $\mathbb{Q}$ one has the semiplanes spanned by $(0,0)$ and $(0,a)$ for every $a\in \mathbb Q,$ so ACC and DCC fail for sub-spaces.

\appendix\section{Categories arising in this paper}\label{cat1}$ $

Let us formulate the theory of categories whose objects are pairs. In the class discussed in this paper, which we call \textbf{Mon-categories},  the underlying monoid $\tT$ is not fixed. 





 \subsection{Doubling}

\begin{theoalpha}\label{theoalphaF}
The doubling procedure is functorial, from $\tT$-modules to pairs.
\end{theoalpha}
\begin{proof}
The weak morphisms of the doubled category are pairs $\hat f =(f_0,f_1):(\widehat{\mcA},\widehat{\mcA_0})\to (\widehat{\mcA'},\widehat{\mcA'})$, where $f_0,f_1:\mcA \to \mcA'$ are weak morphisms, and we define  $$\hat f(b_0,b_1) = (f_0(b_0)+f_1(b_1),f_1(b_0)+ f_0(b_1)).$$

If $\sum (b_{0,i}b_{1,i}) \in \widehat{\mcA_0} $, then $\sum b_{0,i}+\sum b_{1,i}\in \mcA_0,$  so $\sum f_j(b_{0,i})+\sum f_j(b_{1,i})\in \mcA_0,$ implying
\begin{equation}
    \begin{aligned}
        \sum \hat{f}(b_{0,i}b_{1,i}) & = \left(\sum f_0(b_{0,i})+ \sum f_1(b_{1,i}),\sum f_0(b_{1,i})+\sum f_1(b_{0,i})\right)\in \mcA_0.
    \end{aligned}
\end{equation}

Given a morphism $f:(\mcA,\mcA_0) \to (\mcA',\mcA'),$
        we define $\hat f:(\widehat{\mcA},\widehat{\mcA_0})\to (\widehat{\mcA'},\widehat{\mcA'})$, by putting
        $\hat f = (f,f).$

The rest is routine, as in the proofs of \cite[Theorem~4.2(iii),(iv)]{AGR2}.
        \end{proof}

\begin{rem}
    Although one can double module morphisms, one has difficulties in doubling multiplicative maps in the category of semiring pairs, since one needs to compare $$f((b_0,b_1)(c_0,c_1))= f(b_0 c_0 +b_1c_1, b_0 c_1 +b_1c_0)$$ with $$(f(b_0)f(c_0)+f(b_1)f(c_1),f(b_0)f(c_1)+f(b_1)f(c_0)),$$ which need not be equal if $f$ is not an additive homomorphism. \end{rem}




\begin{definition}\label{pmor}
  A \textbf{$\preceq$-morphism}
    from a    bimodule $\mcA$ over~$\tT$  with a pre-surpassing relation $\preceq$ to  a $\tT'$-bimodule $\mcA'$  with a pre-surpassing relation  $\preceq'$, is a weak morphism
    satisfying, for all $b_i \in \mcA,$\begin{enumerate}
\item \label{precm}
$ f( b_1 +b_2) \preceq'  f(b_1)+f(b_2), \ \forall b_i \in \mcA,$
\item If $b_1\preceq b_2$ then  $f(b_1) \preceq' f(b_2).$
  \end{enumerate}
       \end{definition}
\begin{lemma}  $f(a^\circ) \preceq 'f(a)^\circ$, for
     any $\preceq$-morphism to a pair $(\mcA',\mcA_0')$ with \PropertyN\ satisfying $\mcA_0' = \mcA'_\Null$.
\end{lemma}
\begin{proof}
    $f(a+a^\dag) \succeq' f(\zero) = \zero,$ so $ f(a)+f(a^\dag) \succeq 'f(a+a^\dag) \succeq\zero,$ so by \PropertyN, $ f(a)+f(a^\dag) = f(a)^\circ$, yielding  $f(a^\circ) \preceq f(a)^\circ.$
\end{proof}
A $\subseteq$-morphism of a hyperpair is just a $\subseteq$-{hypermorphism}.

\subsection{Direct products}$ $

\begin{example}\label{dirpdt}
   Given \admissible\ $\tT_i$-semirings $\mathcal A_i$ for $i \in I$ we can form their categorical product, the \textbf{direct product} $\prod _{i \in I}\, \mathcal
A_i$  as usual, with $\preceq$ defined componentwise. This is denoted
$\mathcal A^{I}$ when each $\mathcal A_i = \mathcal A.$

 When the   $I$  is
 finite, and $(\mathcal A_i, {\mathcal {A}_i}_0)$ are \admissible\ pairs, then
$(\prod \mathcal A_i,\prod {\mathcal {A}_i}_0)$
 is an \admissible\ pair.
\end{example}
\begin{lemma}
   If each $(\mathcal A_\ell, {\mathcal A_\ell}_0)$ has \PropertyN, resp.~is negated, then    $\prod_{\ell \in I}\, \mathcal
(\mathcal A_\ell, {\mathcal A_\ell}_0)$ has \PropertyN, resp.~has  a negation map defined componentwise.   If each $(\mathcal A_\ell, {\mathcal A_\ell}_0)$ has a pre-surpassing relation, the  pre-surpassing relation is defined componentwise. Thus the direct product is a categorical product in
the various categories under consideration.
\end{lemma}
\begin{proof}
    Check componentwise.
\end{proof}

In another class of categories,  the base set  $\tT$ is  fixed, continuing~\cite{AGR2}. This would be appropriate to module theory, cf.~\cite{Row24}, which was not considered in depth in this paper, since $\tT$-modules are not  \admissible\ in general.


\bibliography{bibliolouis}
\bibliographystyle{alpha}
\end{document}